\documentclass[11pt,twoside]{amsart}
\usepackage{amsmath,amsfonts,amssymb,amscd,amsthm,amsbsy,epsf}
\usepackage{fontenc}
\usepackage{graphicx} 
\DeclareGraphicsExtensions{ .tiff, .jpg, .png}
\usepackage{epstopdf} 
\usepackage{pdfsync}
\usepackage{cite}
\usepackage{color}
\newtheorem{thm}{Theorem}[section]
\newtheorem{lem}[thm]{Lemma}
\newtheorem{cor}[thm]{Corollary}
\theoremstyle{remark}

\theoremstyle{definition}
\newtheorem*{defn}{Definition}
\numberwithin{equation}{section}

\newcommand{\ka}{\varkappa}
\newcommand{\ep}{\varepsilon}
\newcommand{\p}{\partial}
\newcommand{\Om}{\Omega}
\newcommand{\B}{\mathcal{B}}

\def\sfint{\mathop{\int\mkern-16mu \raise.15ex\hbox{$\scriptstyle\diagup$}}\nolimits}
\def\fint{\mathop{\int\mkern-19mu {\diagup}}\nolimits}

\def\ds{\displaystyle}
\def\chix{{\raise.4ex\hbox{$\chi$}}}
\def\ep{\varepsilon}
\def\supp{\operatorname{supp}}
\def\real{{\mathbb R}}
\def\B{{\mathcal B}}
\def\C{{\mathcal C}}
\newcommand{\finedim}{{\unskip\nobreak\hfil\penalty50
   \hskip2em\hbox{}\nobreak\hfil\mbox{$\Box$ \qquad}
   \parfillskip=0pt \finalhyphendemerits=0\par\medskip}}

\begin{document}
\title[On a long range segregation model]
{On a long range segregation model}
\author{L.  Caffarelli} 
\address{The University of Texas at Austin\\
Department of Mathematics -- RLM 8.100\\
2515 Speedway -- Stop C1200\\
Austin, TX~78712-1202}
\email{caffarel@math.utexas.edu}
\author{S. Patrizi}
\address{The University of Texas at Austin\\
Department of Mathematics -- RLM 8.100\\
2515 Speedway -- Stop C1200\\
Austin, TX~78712-1202}
\email{spatrizi@math.utexas.edu}
\author{V. Quitalo}
\address{Purdue University\\
Department of Mathematics\\150 N. University Street\\
West Lafayette\\ IN 47907-2067}
\email{vquitalo@math.purdue.edu}

\thanks{S. Patrizi was supported by the
ERC grant 277749 ``EPSILON Elliptic
Pde's and Symmetry of Interfaces and Layers for Odd Nonlinearities''}


\begin{abstract}
In this work we study the properties of  segregation processes modeled by a family of equations
$$
L(u_i) (x) = u_i(x)\: F_i (u_1, \ldots, u_K)(x)\qquad i=1,\ldots, K
$$
where $F_i (u_1, \ldots, u_K)(x)$ is a non-local factor that takes into consideration the values of the functions $u_j$'s in a full neighborhood of $x.$   We consider as a model problem
$$\Delta u_i^\ep (x) = \frac1{\ep^2} u_i^\ep (x)\sum_{i\neq j} H(u_j^\ep)(x)$$
where $\ep$ is a small parameter and $H(u_j^\ep)(x)$ is  for instance
$$H(u_j^\ep)(x)= \int_{\mathcal{B}_1 (x)} u_j^\ep (y)\, \text{d}y$$
or
$$H(u_j^\ep)(x)= \sup_{y\in \mathcal{B}_1(x)} u_j^\ep (y).$$
Here the set $\mathcal{B}_1(x)$ is the unit ball centered at $x$ with respect to a smooth, uniformly convex norm $\rho$ of $\real^n$. 
Heuristically,  this will force the populations to stay at $\rho$-distance 1, one from each other, as $\ep\to0$.
\end{abstract}

\date{\today }
\maketitle
\pagestyle{plain}
\baselineskip=24pt

\section{Introduction}
Segregation phenomena occur in many areas of mathematics and science: from equipartition problems in geometry, to social and biological processes (cells, bacteria, ants, mammals), to finance (sellers and buyers). 
There is a large body of literature in connection to our work and we would like to refer to \cite{caffarelli_geometry_2009, caffarelli_singularly_2008, conti_coexistence_2008, conti_minimal_2009, conti_global_2011, conti_variational_2003, conti_neharis_2002, conti_asymptotic_2005, conti_regularity_2005, dancer_positive_1995-2, dancer_positive_1995-1, dancer_free_2003,dancer_spatial_1999,  dancer_competing_1995, mimura_effect_1991, shigesada_effects_1984, cushman_density_1988, paine_ecological_1984, ohsawa_how_2003, noris_uniform_2010, dancer_limit_2012, soave_uniform_2015,wei_asymptotic_2008 }  and the references therein. We particularly  would like to point out the articles \cite{ohsawa_how_2003, cushman_density_1988, shigesada_effects_1984, paine_ecological_1984,mimura_effect_1991 } where spatial separation due to competition for resources is discussed among ant nests, mussels and sessile animals.

They study a family of models arising from different applications whose main two ingredients are: in the absence of competition species follow a ``propagation" equation involving diffusion, transport, birth-death, etc, but when two species overlap, their growth is mutually inhibited by competition, consumption of resources, etc. 
The simplest form of such models consists, for species $\sigma_i$  with spatial density $u_i,$   on a system of equations
$$
L (u_i)=u_i \: F_i (u_1, \ldots, u_K).
$$
The operator $L$ quantifies diffusion, transport, etc, while the term $u_i\: F_i$ does attrition of $u_i$ from competition with the remaining species.

In these models, the interaction is punctual, i.e. $u_i(x)$ interacts with the remaining densities also at position $x$. There are many processes, though where the growth of $\sigma_i$  at $x$ is inhibited by the populations $\sigma_j$ in a full area surrounding $x.$

The purpose of this work is a first attempt to study the properties of such a segregation process. 
Basically, we consider a family of equations,
$$
L(u_i )(x) = u_i(x)\: F_i (u_1, \ldots, u_K)(x)
$$
where $F_i (u_1, \ldots, u_K)(x)$ is now a non-local factor that takes into consideration the values of $u_j$ in a full neighborhood of $x.$   
%
%
%
%
Given the previous discussion a possible model problem would be the system
$$\Delta u_i^\ep (x) = \frac1{\ep^2} u_i^\ep (x)\sum_{i\neq j} H(u_j^\ep)(x),\quad i=1,\ldots, K$$
where $\ep$ is a small parameter and $H(u_j^\ep)(x)$ is a non-local operator, for instance 
$$H(u_j^\ep)(x)= \int_{B_1 (x)} u_j^\ep (y)\, \text{d}y$$
or 
$$H(u_j^\ep)(x)= \sup_{y\in B_1(x)} u_j^\ep (y)\ .$$

 To study the limit configuration when  the competition for resources is  very high, we consider the limit when $\ep $ tends to 0. Heuristically, the non-local term  forces the populations to stay at distance 1, one from each other. As an example, as we will prove, in the case of two populations in dimension two,  we will have strips of length precisely one between the regions where the populations live. At ``edge" points, that we will define as singular points,  the angles of the asymptotic cones have to be the same,  see Figure \ref{fig:example limit configuration}. Here $S_i=S_i^1\cup S_i^2$, $i=1,2$, represents the region where the 
 the population $\sigma_i$ with density $u_i$ exists. 
Moreover, the ratio between the normal derivatives at regular points across the free boundary,  depends on the ratio of the respective curvature $\ka$. For example, if  $Z_1 \in \partial S_1^1$ and $Z_2 \in \partial S_2^1$,  $Z_1$ and $Z_2$ are not  ``edge'' points, and $d(Z_1,Z_2)=1$ then 
 $$
\frac{u_\nu^1(Z_1)}{u^2_\nu (Z_2)}= \frac{\ka (Z_1) }{\ka (Z_2)}\quad\text{if }\ka(Z_2)\neq 0, \quad \mbox{and} \quad u_\nu^1(Z_1)=u^2_\nu (Z_2) \quad\text{if }\ka(Z_2)= 0.
 $$

 \begin{figure}[h]
\begin{center}
\includegraphics[scale=0.4]{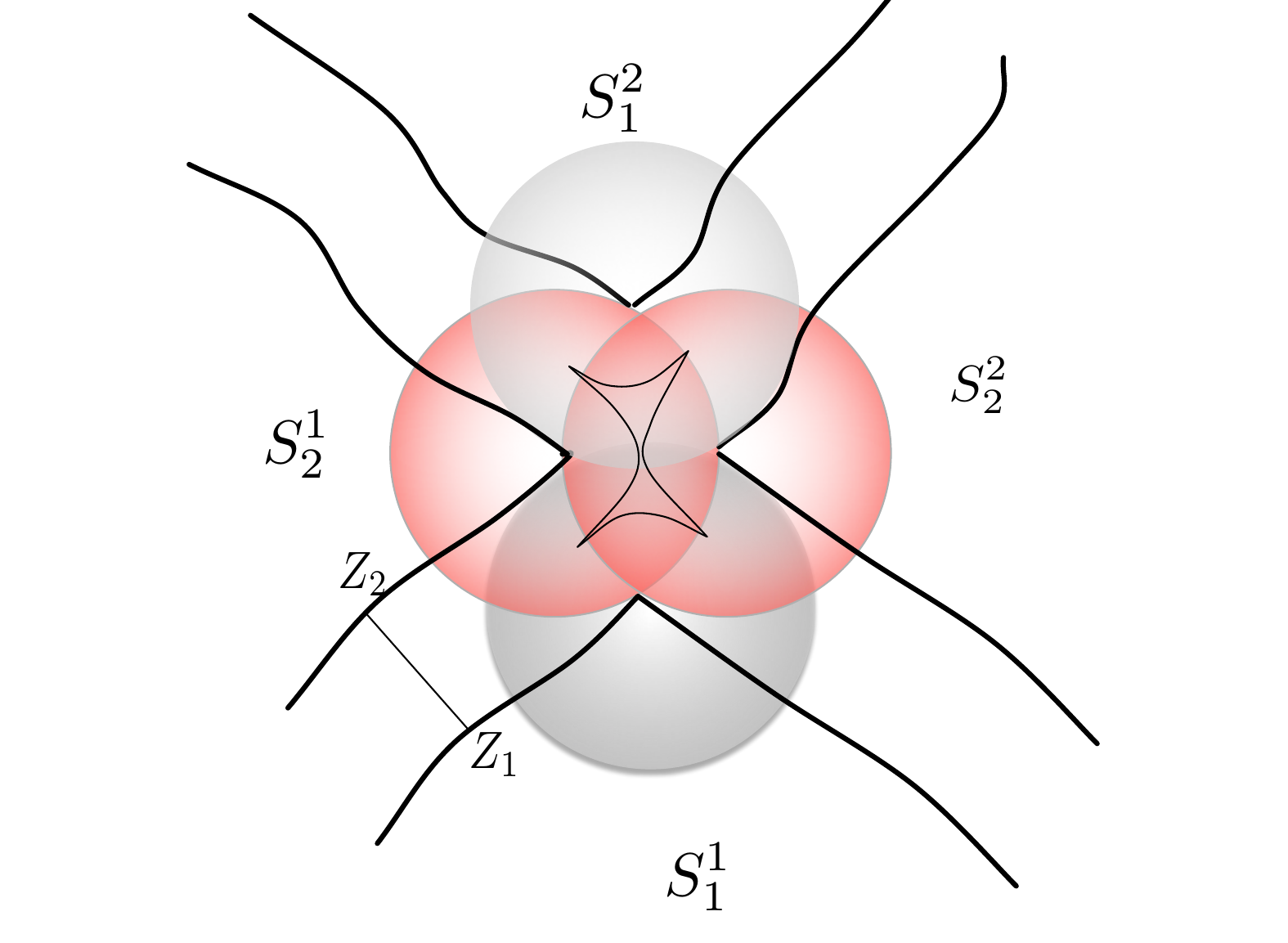}
\caption{Example of a limit configuration for $ K=2$, $n=2$}
\label{fig:example limit configuration}
\end{center}
\end{figure}

We will consider instead of the unit ball in the Euclidean norm  $B_1(x)$, 
the translation at $x$ of a general smooth set $\mathcal{B}$ that is also uniformly convex, bounded and symmetric with respect to the origin. 
The set $\mathcal{B}$ defines a smooth, uniformly convex norm $\rho$ in $\real^n$. 

Let us note that there is some similarity also with the Lasry-Lions model of price formation 
(see \cite{caffarelli_price_2011,lasry_mean_2007}) where selling and buying prices are separated by a gap due to transaction cost.
\medskip
\section{Notation and statement of the problem}

Let $\B$ be an open bounded domain of $\real^n$, convex, symmetric with respect to the origin and 
with smooth boundary. 
Then $\B$ can be represented as the unit ball of a norm $\rho : \real^n\to \real$, $\rho \in C^\infty 
(\real^n\setminus \{0\})$, called the defining function of $\B$, i.e., 
$$\B = \left\{ x\in \real^n \mid \rho (x) < 1\right\}.$$
We assume that $\B$ is uniformly convex, i.e., there exists $0<a\le A$ such that in $\real^n\setminus \{0\}$
\begin{equation}\label{eq:stmt1}
aI_n \le D^2 \left( \frac12 \rho^2\right) \le AI_n\ ,
\end{equation}
where $I_n$ is the $n\times n$ identity matrix.
In what follows we denote 
\begin{align*}
\B_r & := \left\{ y\in \real^n \mid \rho (y) < r\right\},\\
\B_r (x) & := \left\{ y\in \real^n \mid \rho (x-y) < r\right\}.
\end{align*}
So through the paper we will always refer to the Euclidean ball as $B$ and to the $\rho$-ball as $\B$. 
For a given closed set $K$, let 
$$d_\rho (\cdot, K)=\inf_{y \in K} \rho(\cdot -y)$$
be the distance function from $K$ associated to $\rho$. 
Then there exist $c_1,c_2 >0$ such that 
\begin{equation}\label{eq:stmt2}
c_1 d(\cdot, K) \le d_\rho (\cdot,K) \le c_2 d(\cdot, K)\ ,
\end{equation}
where $d(\cdot,K)$ is the distance function associated to the Euclidian norm  $|\cdot|$ of $\real^n$.

Let $\Omega \subset \real^n$ be a bounded Lipschitz  domain.  We will denote by $(\partial\Om)_1$ the $\rho$-strip of size~1 around $\partial\Om$
in the complement of $\Omega$ defined by 
$$(\partial\Om)_1 := \left\{ x\in \Omega^c : d_\rho (x,\partial\Om) \le 1\right\}\ .$$
For $i = 1,\ldots,K$, let $f_i$ be non-negative  functions defined on $(\partial\Omega)_1$ 
with  supports at $\rho$-distance equal or greater  than 1, one from each other:
\begin{equation}\label{eq:stmt4}
 d_\rho(\supp f_i,\supp f_j)\ge1\ ,\quad \text{for }\ i\ne j\ .
\end{equation}
We will  consider the following system of equations: for $i = 1,\ldots,K$
\begin{equation}\label{eq:stmt3}
\begin{cases}
\ds \Delta u_i^\ep (x) = \frac1{\ep^2} u_i^\ep (x) \sum_{j\ne i} H(u_j^\ep) (x)&\text{in }\Omega,\\
\noalign{\vskip6pt}
\ds u_i^\ep = f_i &\text{on }(\partial\Omega)_1.\\
\end{cases}
\end{equation}
The functional $H(u_j)(x)$ depends only on the restriction of $u_j$ to $\B_1(x)$. 

We will consider, for simplicity, 
\begin{equation}\label{H1}
H(w)(x)  = \int_{\B_1(x)} w^p(y) \varphi \big(\rho (x-y)\big)\mbox{d}y,\qquad 1\le p<\infty
\end{equation}
or 
\begin{equation}\label{H2}
H(w) \big(x\big) = \sup_{\B_1(x)} w
\end{equation}
with $\varphi$ a strictly positive smooth function of $\rho$, with at most polynomial decay at 
$\partial\B_1$:
\begin{equation}\label{varphidecay}
\varphi (\rho) \ge C(1-\rho)^q,\quad q\ge 0.
\end{equation}

  In rest of the paper, when we refer to consider $u_1^\ep,\ldots,u_K^\ep$,
viscosity solutions of the problem \eqref{eq:stmt3}, we mean that $u_1^\ep,\ldots,u_K^\ep$ are continuous functions that satisfy in the viscosity sense the system of equations
 \eqref{eq:stmt3}.  
Moreover, we make  the following assumptions: for $i=1,\ldots,K$,
\begin{equation}\label{mainassumpts}\begin{cases}
 \ep>0,\,  \Omega \text{ is a bounded Lipschitz domain of }\real^n,\\ 
 f_i: (\partial\Omega)_1\to\real,\,f_i\ge 0,\,f_i\not\equiv0,\, f_i\text { is  H\"older continuous, }\\
 \exists\text{ $c>0$ s. t.   }\forall x\in\partial\Om\cap\supp\,f_i,\, |\B_r(x)\cap \supp\,f_i|\ge c|\B_r(x)|,\\
 \text{(\ref{eq:stmt4}) holds true},\\ H \text{ is either of the form (\ref{H1}) or (\ref{H2}) and 
(\ref{varphidecay}) holds.}
\end{cases}
\end{equation}

\section{Main results}
For the reader's convenience  we present  our main results below. Assume that \eqref{mainassumpts} holds true, then:
\begin{enumerate}
\item[] {\bf Existence} {\bf (Theorem \ref{thm:existence}):} \\{\em There exist continuous functions $u_1^\ep,\ldots,u_K^\ep$, depending on the parameter $\ep$, 
viscosity solutions of the problem \eqref{eq:stmt3}.}
\item[] {\bf Limit problem (Corollary \ref{convergencecor}):} \\{\em There exists a subsequence 
$(\vec u)^{\ep_m}$ converging locally uniformly, as $\ep\to0$, to a function
$\vec u = (u_1,\ldots, u_K)$, satisfying the following  properties:

\begin{itemize}
\item[i)] the $u_i$'s  are  locally Lipschitz continuous in $\Om$ and have supports at distance at least 1, one from each other, i.e. 
$$u_i\equiv 0\quad\text{in the set }\quad  \{x\in\Om\,|\,d_{\rho}(x,\text{supp } u_j)\le1\}\quad\text{for any }j\neq i.$$
\item[ii)] $\Delta u_i =0$ when $u_i >0$.
\end{itemize}}
\item[] {\bf Semiconvexity of the free boundary (Corollary \ref{tangebtballcor}):}\\
{\em If $x_0 \in \partial \{u_i>0\}$ there is an exterior tangent $\rho$-ball of radius 1 at $x_0$.}
\item[] {\bf  The supports of $u_i$  are  sets of finite perimeter (Corollary \ref{hausmescor}):}\\
{\em The set $\{u_i>0\}$ has finite perimeter.
\item[] {\bf Sharp characterization of the interfaces (Theorem \ref{lem:6.1}):} \\
Under the additional assumption that  $p=1$ in (\ref{H1}), 
the supports of the limit functions are at distance exactly 1, one from each other, i.e, 
if $x_0 \in \partial \{u_i >0\}\cap\Om$, 
then there exists $j\ne i$ such that 
\begin{equation*}\overline{\B_1 (x_0)} \cap \partial \{u_j >0\} \ne \emptyset\ .\end{equation*}
\item[]{\bf Classification of singular points in dimension 2 (Lemma \ref{singisolated}, Theorem \ref{equalanglethm},  Corollary \ref{corc1reg}, Corollary \ref{sigulapoints2dthm}):}\\
For $n=2$, under the additional assumption that  $p=1$ in (\ref{H1}), for $i\neq j$, let $x_0\in\partial \{u_i>0\}\cap\Om$ and $y_0\in \partial \{u_j>0\}\cap\Om$ be points such that
 $\{u_i>0\}$ has an angle $\theta_i$ at  $x_0$,  
$\{u_j>0\}$ has an angle $\theta_j$ at  $y_0$ and $\rho(x_0-y_0)=1$. Then we have
$$\theta_i=\theta_j.$$
If $x_0\in\partial \{u_i>0\}\cap\partial \Om$ and $y_0\in \partial \{u_j>0\}\cap\Om$, then
\begin{equation*}\theta_i\leq \theta_j.\end{equation*}}

\medskip

{\em \noindent  Moreover, singular points, i.e. points where the free boundaries have corners, are isolated and finite.}
{\em \noindent  If the domain is a strip and there are only two populations,  under additional monotonicity assumptions   on  the boundary data, the free boundary sets $\partial \{u_i>0\}$, $i=1,2$, are of class $C^1$.}
\item[] {\bf Lipschitz regularity for free boundary for the obstacle problem associated in dimension 2  (Theorem \ref{lipfreeboundary}):}\\
{\em \noindent  For $n=2$, under the additional assumption that  $p=1$ in (\ref{H1}), $f_i \equiv 1$ and additional conditions about the regularity of $\p \Omega$,  if $(u_1^\ep,\ldots,u_K^\ep)$ is a particular  solution of  \eqref{eq:stmt3} which satisfies the associated obstacle problem \eqref{obstaclepbfree} with  $(u_1,\ldots,u_K)$  the limit as $\ep\to 0$, then the free boundaries $\partial\{u_i>0\}$,  $i=1,\ldots,K$, are Lipschitz curves of the plane.}

\item[] {\bf Free boundary condition (Theorem \ref{thm: fbcond}):}\\
 {\em In  any dimension,  assume that we have  2 populations,  $H$ is defined as in  \eqref{H1}  with $\varphi\equiv 1$, $p=1$ and $\B_1(x)=B_1(x)$ is the Euclidian ball, $0\in\partial\{u_1>0\}$, $e_n\in \partial\{u_2>0\}$, and   $\partial\{u_1>0\}$ and $\partial\{u_2>0\}$ are of class $C^2$ in a neighborhood of 0 and $e_n$ respectively. 
 Let  $\ka_i(0)$ denote  the principal curvatures of $\partial\{u_1>0\}$ at 0, where outward is the positive direction and  let  $\ka_i(e_n)$ denote  the principal curvatures of $\partial\{u_2>0\}$ at $e_n$ where now inward is the positive direction.  Then, we have the following relation on the exterior normal derivatives of $u_1$ and $u_2$:
$$
\frac{u_\nu^1(0)}{u^2_\nu (e_n)}=\prod_{i=1\atop \ka_i(0)\neq 0 }^{n-1} \frac{\ka_i(0)}{\ka_i (e_n)}\quad\text{if }\ka_i(0)\neq 0\text{ for some }i=1,\ldots,n-1,
$$ and
\begin{equation*}u_\nu^1(0)=u^2_\nu (e_n)\quad\text{if }\ka_i(0)= 0\text{ for any }i=1,\ldots,n-1.\end{equation*}}

\end{enumerate}

\section{Existence of solutions}

This proof  follows the same steps as in \cite{quitalo_free_2013} and it is written below for the reader's convenience.   
\begin{thm}\label{thm:existence} Assume \eqref{mainassumpts}. Then there exist continuous positive functions $u_1^\ep,\ldots,u_K^\ep$, depending on the pa\-ra\-meter $\ep$, 
viscosity solutions of the problem \eqref{eq:stmt3}.\end{thm}
\begin{proof} 
The proof uses a fixed point result.
Let $B$ be the Banach space of bounded continuous vector-valued functions defined on the domain 
$\Omega$ with the norm 
$$\| (u_1 , u_2,\ldots, u_K)\|_B 
:= \max_i \Big( \sup_{x\in\Omega} |u_i(x)|\Big)\ .
$$
For $i=1,\ldots,K$, let $\phi_i$ be the solutions of  
\begin{equation}\label{phibarrier}
\begin{cases} \Delta \phi_i = 0&\text{ in }\ \Omega,\\
\phi_i = f_i&\text{ on }\ \partial\Omega.
\end{cases}
\end{equation}
Let $\Theta$ be the  subset of bounded continuous functions in $\Om$, that satisfy prescribed boundary data, 
and are bounded from above and from below as stated below:
\begin{equation*}\begin{split}
\Theta = \Big\{ &(u_1, u_2,\ldots, u_K)\,|\,u_i:\Om\rightarrow\real \text{ is continuous, }\, 0\le u_i \le \phi_i\ \text{ in }\ \Omega,\,
u_i = f_i\ \text{ on }\ (\partial\Omega)_1\Big\}.
\end{split}
\end{equation*}
Notice that $\Theta$ is a closed and convex subset of $B$. 
Let $T^\ep$ be the operator that is defined on $\Theta$ in the following way: 
$T^\ep\big(( u_1, u_2, \ldots, u_K)\big) := (v_1^\ep, v_2^\ep,\ldots, v_K^\ep)$ 
if for any $ i=1,\ldots,K$, $v_i^\ep$ is solution to the following  problem:
\begin{equation}\label{Tepproblem}
\begin{cases} 
\ds \Delta (v_i^\ep)(x) = \frac1{\ep^2} v_i^\ep(x) \sum_{j\ne i} H(u_j)(x)&\text{in }\Om\\
\noalign{\vskip6pt}
\ds v_i^\ep = f_i&\text{on } (\partial\Omega)_1,\\
\end{cases}
\end{equation}
where $u_j$, $j\ne i$ are given. 
Observe that if $T^\ep$ has a fixed point
\begin{equation*}
T^\ep\big( (u_1^\ep, u_2^\ep,\ldots, u_K^\ep)\big) = (u_1^\ep, u_2^\ep,\ldots, u_K^\ep)
\end{equation*}
then $(u_1^\ep, u_2^\ep,\ldots, u_K^\ep)$ is a solution of problem \eqref{eq:stmt3}.

In order for $T^\ep$ to have a fixed point, we need to prove that it satisfies the hypothesis of the 
Schauder fixed point Theorem, see \cite{gilbarg_elliptic_2001}:
\begin{itemize}
\item[(1)] $T^\ep(\Theta) \subset \Theta$ :
\item[{}] Classical existence results guarantee 
the existence of  a viscosity solution  $(v_1^\ep, v_2^\ep,\ldots, v_K^\ep)$ of  problem  \eqref{Tepproblem} which is smooth in $\Om$. 
Since   $f_i\ge0$ and $f_i\not \equiv0$,  the strong maximum principle implies  
\begin{equation*}v_i^\ep > 0\quad\text{in }\Omega.\end{equation*}
This implies that 
\begin{equation}\label{visubarmobstacle}\Delta v_i^\ep\ge 0\quad \text{ in }\Om,\end{equation}
and, again from the  comparison principle, we have  
$$ v_i^\ep \le \phi_i\quad\text{in }\Omega.$$ 

We have proved  that $T^\ep\big( (u_1, u_2,\ldots, u_K)\big)\in\Theta$.
\medskip 
\item[(2)] {\em $T^\ep$ is continuous}\,:
\newline  Let us assume that  $((u_1)_m,\ldots, (u_K)_m) \to 
(u_1,\ldots, u_K)$ in $B$ meaning that when $m$ tends to $+\infty$,
$$\max_{1\le i\le K} \|(u_i)_m - u_i \|_{L^\infty} \to 0\ .$$
We need to prove that for each fixed $\ep>0$
$$\|T^\ep\big( (u_1)_m,\ldots, (u_K)_m\big) - T^\ep(u_1,\ldots, u_K)\|_{B} \to 0$$
when $m\to +\infty$. 
Let
$$T^\ep\big( (u_1)_m,\ldots, (u_K)_m\big) = \big( (v_1^\ep)_m ,\ldots, (v_K^\ep)_m\big),$$ then
if we prove that there exists a constant $C_\ep$ independent of $m$, so that we have the estimate, for $i=1,\dots,K$ 
$$\|(v_i^\ep)_m - v_i^\ep\|_{L^\infty} \le C_\ep \max_j \|(u_j)_m - u_j\|_{L^\infty},$$
the result follows.
For all $x\in \Omega$ and for fixed $i$, let $\omega_m$ be the function 
$$\omega_m (x) = (v_i^\ep)_m (x) - v_i^\ep(x)\ ,$$
and suppose  for instance that there exists $y\in \Omega$ such that 
\begin{equation}\label{eq:sols1}
\omega_m (y) > r^2 D\max_j \| (u_j)_m - u_j \|_{L^\infty}\ ,
\end{equation}
for some large $D>0$, where $r$ is such that $\Omega \subset B_r $, and $ B_r $ is the  ball centered at 0 of radius $r$ in the Euclidean norm.
We want to prove that this is impossible if $D$ is sufficiently large. 
Let $h_m$ be the concave radially symmetric function
$$h_m (x) = \gamma_m \big( r^2 - |x|^2\big)\ ,$$
with $\gamma_m = D\max_j \|(u_j)_m - u_j\|_{L^\infty}$. 
Observe that: 
\begin{itemize}
\item[(a)] $h_m(x) =0$ on $\partial B_r $; 
\item[(b)] $h_m(x) \le r^2 D \max_j \|(u_j)_m - u_j\|_{L^\infty}$ for all $x$ in $B_r$;
\item[(c)] $0= \omega_m(x) \le h_m(x)$ on $\partial\Omega$, since $(v_i^\ep)_m$ and $v_i^\ep$ 
are solutions with the same boundary data.
\end{itemize}
Since we are assuming \eqref{eq:sols1}, there exists a negative minimum of $h_m- \omega_m$ in $\Om$. 
Let $x_0\in\Om$ be a point where the minimum value of $h_m -\omega_m$ is attained. Then 
$$h_m(x_0) - \omega_m (x_0) < 0\quad\text{and}\quad\Delta  (h_m - \omega_m)(x_0) \ge 0 .$$ Then, we have
\begin{equation*}
\begin{split} 
\Delta \omega_m(x_0)& = \Delta \big( (v_i^\ep)_m\big)(x_0)  - \Delta v_i^\ep(x_0)  \\
& = \frac1{\ep^2} \bigg( ((v_i^\ep)_m(x_0) - v_i^\ep(x_0)) \sum_{j\ne i} H((u_j)_m)(x_0)\\&  - v_i^\ep(x_0) \sum_{j\ne i}\left(H ( u_j)(x_0) - H((u_j)_m)(x_0)\right ) \bigg)\\
& \ge \frac1{\ep^2}\bigg(\left( (v_i^\ep)_m(x_0) - v_i^\ep(x_0) \right)\sum_{j\ne i}H((u_j)_m)(x_0)\\& - v_i^\ep(x_0) (K-1) 
C\max_j\left\|  (u_j)_m-u_j \right\|_{L^\infty (\Omega)}\bigg)
\end{split}
\end{equation*}
adding and subtracting $\frac1{\ep^2} v_i^\ep(x_0) \sum_{j\ne i} H ((u_j)_m)(x_0)$, where $C$ depends on the  $f_j$'s and $\varphi$. 
Then 
\begin{equation*}
\begin{split}
 0 & \le \Delta (h_m  -\omega_m)(x_0)\\
& \le -2\gamma_m n - \frac1{\ep^2}\bigg(( (v_i^\ep)_m - v_i^\ep) (x_0) 
\sum_{j\ne i} H(( u_j)_m)(x_0) \\&
- v_i^\ep(x_0) (K-1) C\max_j\left\| (u_j)_m- u_j\right\|_{L^\infty}
\bigg) \\
& \le - 2n D\max_j \|(u_j)_m - u_j \|_{L^\infty} 
+ \frac1{\ep^2} v_i^\ep(x_0) (K-1) C\max_j\left\| (u_j)_m- u_j\right\|_{L^\infty}\\&
\le - 2n D\max_j \|(u_j)_m - u_j \|_{L^\infty} +\frac{ \widetilde{C}}{\ep^2}\max_j\left\| (u_j)_m- u_j\right\|_{L^\infty}
\end{split}
\end{equation*}
because $0< h_m(x_0) < \omega_m(x_0)=\big( (v_i^\ep)_m - v_i^\ep\big) (x_0)$ and $\sum_{j\ne i} H((u_j )_m)(x_0)\ge 0$ 
and so 
$$- \frac1{\ep^2} \big( (v_i^\ep)_m - v_i^\ep\big) (x_0) \sum_{j\ne i}  H((u_j )_m)(x_0)\le 0\ .$$
Taking $D=D_\ep > \frac{\widetilde{C}}{2n\ep^2}$, we obtain that 
$$0 \le \Delta (h_m - \omega_m) (x_0) <0$$
which is a contradiction. 
\medskip
\item[(3)] {\em $T^\ep(\Theta)$ is precompact}\,: 
\newline
Let $ \big((u_1)_m,\ldots, (u_K)_m\big) $ be a bounded sequence in $B$ and let
$$ \big( (v_1^\ep)_m ,\ldots, (v_K^\ep)_m\big)=T^\ep\big( (u_1)_m,\ldots, (u_K)_m\big). $$  
Then by standard H\"older estimates for viscosity solutions, $ \big( (v_1^\ep)_m ,\ldots, (v_K^\ep)_m\big)$ is bounded in the space of H\"older continuous functions on $\overline{ \Om}$.
Since the subset of $\Theta$  of H\"older continuous functions on $\overline{ \Om}$ is precompact in $\Theta$, we can extract from $\big( (v_1^\ep)_m ,\ldots, (v_K^\ep)_m\big)$ a subsequence which is converging in $B$.
\end{itemize}

We have proven the existence of   a solution $(u_1^\ep,\ldots,u_K^\ep)$  of  \eqref{eq:stmt3}. The same argument as in (1) shows that $u_i^\ep>0$ in $\Om$.
This concludes the proof of the theorem.
\end{proof}

\section{Uniform in $\ep$ Lipschitz estimates}\label{Lipestsection}
In this section we will prove uniform in $\ep$ Lipschitz estimates that will imply the convergence, up to subsequences, of  the solution $(u_1^\ep, \ldots, u_K^\ep)$ of \eqref{eq:stmt3} to a limit function $(u_1,\ldots,u_K)$ as $\ep\to 0$. We will show that the functions $u_i$'s are locally Lipschitz continuous in $\Om$ and harmonic inside their support. Moreover, 
$u_i\equiv 0$ in the $\rho$-strip of size 1 of the support of $u_j$ for any $j\ne i$, i.e., the supports of the limit functions are  at distance at least 1, one from each other.
We start by proving general properties of subsolutions of uniform elliptic equations.  
\begin{lem}\label{lem:gen1}
Let: 
\begin{itemize}
\item[a)] $\omega$ be a subharmonic function in $\B_1$, such that 
\begin{itemize}
\item[a$_1$)] $\omega \le 1$ in $\B_1 $;
\item[a$_2$)] $\omega (0) = m>0$.
\end{itemize}
\item[b)]  $D_0$ be a smooth convex set with bounded curvatures 
$$|\ka_i (\partial D_0)|  \le C_0,\quad i=1,\ldots,n-1$$
(like $\B_1$ above). 
\end{itemize}
Then, there exists a universal $\tau_0 = \tau_0 (C_0,n,\rho)$ such that, if the distance $d_\rho(D_0,0) \le \tau_0 m$,  
then $$\sup_{\partial D_0 \cap \B_1} \omega \ge \frac{m}{2}.$$
\end{lem}

\begin{proof} Assume w.l.o.g. that $0 \notin D_0$ and let $h$ be harmonic in $\B_1 \setminus D_0$ and such that
$$\begin{cases} h = 1\ \text{ on }\ (\partial \B_1)\setminus D_0\\
\noalign{\vskip6pt}
h = \frac{m}{2}\ \text{ on }(\partial D_0) \cap \B_1\ .
\end{cases}$$
By assumption (b), the set  $\B_1 \setminus D_0$ satisfies an exterior uniform ball condition at any point of  $\partial D_0\cap\B_1$, therefore, by a standard barrier argument, $h$ grows no more than  linearly away from $\partial D_0$ in $\B_\frac{1}{2}$, i.e., there exist $k_1,k_2>0$ depending on $C_0$ and $n$ such that, if $x\in \B_\frac{1}{2}\setminus D_0$ and $d(x,\partial D_0)\leq k_2$, then $h(x)\leq k_1d(x,\partial D_0)+\frac{m}{2}$. To prove that $h(0)<m$ observe that if  $\tau_0\leq k_2c_1$, where $c_1$ is given by  \eqref{eq:stmt2},  then $d(0,\partial D_0)\leq\tau_0 m/c_1\leq k_2 m\leq k_2$ and  therefore, if in addition  $\tau_0$ is so  small that $\frac{k_1}{c_1}\tau_0\le\frac12$, we have
$$h(0)\leq k_1d(0,\partial D_0)+\frac{m}{2}\leq \frac{k_1}{c_1}d_\rho(0,\partial D_0)+\frac{m}{2}\leq\frac{k_1}{c_1} \tau_0 m+\frac{m}{2} < m.$$
Hence,  we must have $\sup_{(\partial D_0 \cap \B_1)} \omega \ge \frac{m}{2}$, otherwise the
comparison principle would imply $\omega(x)\leq h(x)$ in $\B_1 \setminus D_0$, which is a contradiction at $x=0$. 
\end{proof}


\begin{lem}\label{lem:gen2}
 Let $\omega$ be a positive subsolution of a uniformly elliptic equation, $(\lambda^2 I \le a_{ij} \le \Lambda^2 I)$
$$a_{ij} D_{ij} \omega \ge \theta^2 \omega\quad\text{ in }\B_r.$$
Then there exist $c,C>0$ such that
$$\frac{\omega (0)}{\sup\limits_{\B_r} \omega}\le C\mbox{\rm e}^{-c\theta r}.$$
\end{lem}

\begin{proof}
The function 
$$g(x) = \sum_{i=1}^n \cosh \left(\frac{\theta}{\Lambda} x_i\right)$$
is a supersolution of the equation $a_{ij} D_{ij} u = \theta^2u$. Moreover, using the convexity of the exponential function, it is easy to check that it satisfies 
$$g(x)\ge C_1 e^{c\theta r}\quad\text{for any }x\in\partial\B_r.$$ Then, the comparison principle implies 
$$\frac{\omega (x)}{\sup\limits_{\B_r} \omega}\le \frac{g(x)}{C_1 e^{c\theta r}}\quad\text{for any  }x\in\B_r.$$ 
The result follows taking $x=0$.
\end{proof}

The next lemma says that if $u_i^\ep$ attains a positive value $\sigma$ at some interior point, then all the other functions $u_j^\ep$, $j\neq i$, go to zero exponentially in a 
$\rho$-ball of radius $1+c\sigma$ around that point. 

\begin{lem}\label{uj=0closui} Assume \eqref{mainassumpts}.  Let $(u_1^\ep,\ldots,u_K^\ep)$ be a
viscosity solution of the problem \eqref{eq:stmt3}. For $i=1,\ldots,K$, $\sigma>0$, and $0<r<1$ let 
$$\Gamma_i^{\sigma,r}:=\{y\in\Omega\,:\, d_\rho(y,\supp\, f_i)\geq 2r,\,u_i^\ep=\sigma\}$$  
and 
$$m:=\frac{\sigma}{\sup_{\partial\Om} f_i}.$$ Then, there exists a universal constant $0<\tau<1$ such that,  in the sets 
$$\Sigma_{i,j}^{\sigma,r}:=\left\{x\in\Om\,:\,d_\rho(x,\Gamma_i^{\sigma,r})\leq 1+\frac{\tau m r}{2},\,d_\rho(x,\supp\,f_j)\geq \frac{\tau m r}{4}\right\}$$
we have 
$$u_j^\ep\leq Ce^{-\frac{c\sigma^\alpha r^\beta}{\ep}},\quad\text{for }j\neq i,$$
for some positive $\alpha$ and $\beta$ depending on the structure of $H$ ($p$ and $q$).
\end{lem}
\begin{proof}Let $0<\tau<1$ to be determined. For  $0<r<1$, let us consider the set $ \Sigma_{i,j}^{\sigma,r}$ defined above and let 
$\overline{x}\in \Sigma_{i,j}^{\sigma,r}$. We want to show that for $j\neq i$, we have
\begin{equation}\label{lemma4.3ineq1}
\Delta u^\ep_j\geq \frac{C\sigma^{\overline{\alpha}}  r^{\overline{\beta}}}{\ep^2}u^\ep_j\quad\text{in }\B_\frac{\tau m r}{4}(\overline{x})
\end{equation}
for some $\overline{\alpha}, \overline{\beta}>0$. 
Let us prove it for $\overline{x}$ such that $d_\rho(\overline{x},\Gamma_i^{\sigma,r})=1+\frac{\tau m r}{2}$, which is the hardest case.
First of all, remark that since  $d_\rho(\overline{x},\supp\,f_j)\geq \frac{\tau m r}{4}$,
the ball $\B_\frac{\tau m r}{4}(\overline{x})$ does not intersect $\supp\,f_j$. Therefore, $u_j^\ep$ (which is eventually zero in $\B_\frac{\tau m r}{4}(\overline{x})\cap \Om^c$) satisfies 
\begin{equation}\label{usubsoluiinBoverlinex}\Delta u_j^\ep \geq \frac1{\ep^2} u_j^\ep  \sum_{k\ne j} H(u_k^\ep) \quad\text{in }\B_\frac{\tau m r}{4}(\overline{x}).\end{equation}
Next, the ball $\B_{1-\frac{\tau m r}{2}}(\overline{x})$ is at distance  $\tau m r$ from a point $y\in\Gamma_i^{\sigma,r}$. Remark that 
since  $\B_{2r}(y)\cap \supp\, f_i=\emptyset$, 
 the function $u_i^\ep$ (which is eventually equal to zero  in $\B_{2r}(y)\cap\Omega^c$)
satisfies
$\Delta u_i^\ep\geq 0$ in $\B_{2r}(y)$. Moreover, since $u_i^\ep$ is subharmonic in $\Om$, it attains its maximum at the boundary of $\Om$, so that 
$u_i^\ep/\sup_{\partial\Om} f_i\leq 1$ in $\Om$. In particular $m=\frac{\sigma}{\sup_{\partial\Om} f_i}\leq 1$. Set 
\begin{equation}\label{vlemma5.3}v(x):=\frac{u_i^\ep(y+rx)}{\sup_{\partial\Om} f_i},\end{equation} then 
$v\leq 1$ and $v(0)=u_i^\ep(y)/\sup_{\partial\Om} f_i=\sigma/\sup_{\partial\Om} f_i=m$ and $\Delta v\geq 0$ in $\B_1$.
Let $$D_0:=\B_{\frac{1}{r}-\frac{\tau m}{2}}\left(\frac{\overline{x}-y}{r}\right),$$ then the principal curvatures of $D_0$ satisfy
$$|\ka_i(\partial D_0)|\leq\frac{C_\rho}{\frac{1}{r}-\frac{\tau m}{2}}=\frac{2rC_\rho}{2-r\tau m}<2rC_\rho<2C_\rho.$$ Moreover $D_0$ is at distance $\tau m$ from 0. 
Hence, from Lemma \ref{lem:gen1} applied to the function $v$ given  by \eqref{vlemma5.3} with $D_0$ defined as above, if $\tau=\min\{1,\tau_0\}$, where $\tau_0$ is the universal constant given by the lemma, then 
there is a point $z$ in $\partial\B_{1-\frac{\tau m r}{2}}(\overline{x})\cap\B_{r}(y)$,  such that $u_i^\ep(z)\ge\sigma/2$. Next, remark that 
if $x\in \B_\frac{\tau m r}{4}(\overline{x})$ then 
$$\B_1(x)\supset  \B_\frac{\tau m r}{4}(z)$$
(since $d_\rho(x,z)\leq d_\rho(x,\overline{x})+d_\rho(\overline{x},z)\leq\frac{\tau m r}{4}+ 1-\frac{\tau m r}{2}=1-\frac{\tau m r}{4}$).

Let us first consider the case $H$ defined as in \eqref{H2}. Then for any $x\in \B_\frac{\tau m r}{4}(\overline{x})$ we have 
$$H(u_i^\ep)(x)=\sup_{\B_1(x)}u_i^\ep\geq u_i^\ep(z)\geq \frac{\sigma}{2},$$ which, with together \eqref{usubsoluiinBoverlinex}, implies
 \eqref{lemma4.3ineq1} with $\overline{\alpha}=1$ and  $\overline{\beta}=0$. 

Next, let us turn to the case $H$ defined as in \eqref{H1}.
Remark that since $z\in\B_{r}(y)$ and 
$d_\rho(y,\supp\, f_i)\geq 2r$,
 we have that 
$\B_{r}(z)\cap\supp\, f_i=\emptyset$ 
and  therefore 
the function $u_i^\ep$ (which is eventually equal to zero in $\B_{r}(z)\cap\Omega^c$) 
satisfies $\Delta u_i^\ep\geq 0$ in $\B_{r}(z)$. This implies that $(u_i^\ep)^p$, $p\ge 1$,  is subharmonic in  $\B_{r}(z)$ and by the mean value inequality 

\begin{equation}\label{lemma4.3ineq2}\fint_{B_s(z)} (u_i^\ep )^pdx\geq \left(\frac{\sigma}{2}\right)^p\end{equation} in any Euclidian ball $B_s(z)\subset\B_r(z)$,  for any $p\geq 1$.  
Since $d_\rho$ and the Euclidian distance are equivalent, there is an $s\sim \tau m r$ such that 
\begin{equation}\label{ballsubsetslem4.3}B_s(z)\subset  \B_\frac{\tau m r}{8}(z)\subset \B_\frac{\tau m r}{4}(z)\subset \B_1(x).\end{equation}
Moreover, if $y\in B_s(z)$ and $x\in \B_\frac{\tau m r}{4}(\overline{x})$, then 
$$\rho(y-x)\leq \rho(y-z)+\rho(z-\overline{x})+\rho(\overline{x}-x)\leq  \frac{\tau m r}{8}+\left(1-\frac{\tau m r}{2}\right)+\frac{\tau m r}{4}=1-\frac{\tau m r}{8},$$ that is
\begin{equation}\label{1-rholem4.3}1-\rho(y-x)\geq \frac{\tau m r}{8}.\end{equation}
Hence, using \eqref{ballsubsetslem4.3},  \eqref{varphidecay}, \eqref{1-rholem4.3} and \eqref{lemma4.3ineq2}, for all $x\in \B_\frac{\tau m r}{4}(\overline{x})$  we get
\begin{equation*}\begin{split}H(u_i^\ep)(x)&=\int_{\B_1(x)}(u_i^\ep)^p(y)\varphi(\rho(y-x))dy\\&
\geq \int_{B_s(z)}(u_i^\ep)^p(y)C(1-\rho(y-x))^qdy\\&
\geq\int_{B_s(z)}(u_i^\ep)^p(y)C\left(\frac{\tau m r}{8}\right)^qdy\\&\geq C\sigma^{\overline{\alpha}} r^{\overline{\beta}}
\end{split}\end{equation*} where $\overline{\alpha}$ and
 $\overline{\beta}$ depend on $p$, $q$ and on the dimension $n$. This and  \eqref{usubsoluiinBoverlinex} imply \eqref{lemma4.3ineq1}.
 
 Now, by Lemma \ref{lem:gen2} we get
$$u^\ep_j(\overline{x})\leq Ce^{-\frac{c\sigma^\alpha r^\beta}{\ep}}$$ for $\alpha=\frac{{\overline{\alpha}}}{2}+1$ and $\beta=\frac{{\overline{\beta}}}{2}+1$,
and the lemma is proven. 

\end{proof}

\begin{cor}\label{convedisjointcor}  Assume \eqref{mainassumpts}.   Let $(u_1^\ep,\ldots,u_K^\ep)$ be a
viscosity solution of the problem \eqref{eq:stmt3}.  Let $y$ be a point in $\Om$ such that  
$$u_i^\ep(y)=\sigma, \quad d_\rho(y,\supp\,f_j)\geq 1+\tau m r ,\quad i\neq j \quad\text{and}\quad d_\rho(y,\partial\Om)\geq2r,$$
where $m=\frac{\sigma}{\sup_{\partial\Om} f_i}$,  $0<r<1$,  $\ep\leq \sigma^{2\alpha}r^{2\beta}$ and 
$\tau,\, \alpha$  and $\beta$ are given by Lemma  \ref{uj=0closui}. 
Then there exists a  constant $C_0>0$ such that in $\B_{\frac{\tau m r}{4}}(y)$ we have  
\begin{equation}\label{lipcor5.4}|\nabla u_i^\ep|\leq \frac{C_0}{r}\end{equation} and
\begin{equation}\label{harmoniccor5.4}\Delta  u_i^\ep\rightarrow 0 \text{ as }\ep\rightarrow 0 \text{ uniformly.}\end{equation}
\end{cor}
\begin{proof}
First of all, remark that  $m\leq 1$, as $u_i$ attains its maximum at the boundary of $\Om$. Since in addition $\tau<1$, we have that $\B_{\frac{\tau m r}{2}}(y)\subset \B_{2r}(y)\subset\Om$.
Therefore, we use \eqref{eq:stmt3} to estimate $\Delta u_i^\ep(z)$, for $z\in \B_{\frac{\tau m r}{2}}(y)$. In order to do that, we need to estimate $H(u_j^\ep)(z)$ for $j\neq i$. But  $H(u_j^\ep)(z)$ involves points $x$ at $\rho$-distance 1 from $z$. Let $x$ be such that $d_\rho(x,z)\leq 1$, then   $d_\rho(x,y)\leq 1+\frac{\tau m r}{2}.$ Moreover,  since 
$d_\rho(y,\supp\,f_j)\geq 1+\tau m r$, we have 
$d_\rho(x,\supp\,f_j)\geq \frac{\tau m r}{2}.$  Hence, by Lemma \ref{uj=0closui}, for any $j\neq i$ $$u_j^\ep(x)\leq Ce^{-\frac{c\sigma^\alpha r^\beta}{\ep}}\quad \text{for }x\in\B_1(z).$$
 From the previous estimate and \eqref{eq:stmt3}, it follows that for $z\in \B_{\frac{\tau m r}{2}}(y)$ we have

\begin{equation}\label{deltauepcor5.4}0\leq\Delta  u_i^\ep(z)\leq  u_i^\ep(z) \frac{Ce^{-\frac{c\sigma^\alpha r^\beta}{\ep}}}{\ep^2}
\leq  u_i^\ep(z) \frac{Ce^{-c\ep^{-\frac{1}{2}}}}{\ep^2}=o(1)\quad\text{as }\ep\to0,\end{equation}
for $\ep\leq \sigma^{2\alpha}r^{2\beta}$.
If we normalize the ball $\B_{\frac{\tau m r}{2}}(y)$ in a Lipschitz fashion:
$$\overline{u}_i^\ep(\overline{z}):=2\frac{ u_i^\ep\left(\frac{\tau m r}{2}\overline{z}+y\right)}{\tau m r},$$ then we have 
$$\overline{u}_i^\ep(0)=2\frac{ u_i^\ep\left(y\right)}{\tau m r}=\frac{2\sup_{\partial\Om} f_i}{\tau r},$$ and 
$$0\leq \Delta \overline{u}_i^\ep(\overline{z})\leq\frac{\tau m r}{2} \overline{u}_i^\ep(\overline{z})\sum_{j\neq i}\frac{1}{\ep^2}H( u_j^\ep)\left(\frac{\tau m r}{2}\overline{z}+y\right)\quad\text{for }
\overline{z}\in \B_1(0),$$ where 
$$\frac{\tau m r}{2} \overline{u}_i^\ep(\overline{z})\sum_{j\neq i}\frac{1}{\ep^2}H( u_j^\ep)\left(\frac{\tau m r}{2}\overline{z}+y\right) \leq \frac{Ce^{-c\ep^{-\frac{1}{2}}}}{\ep^2}.\quad$$
Then, by the Harnack inequality (see e.g. Theorem 4.3 in \cite{caffarelli_cabre_1995}), we get
$$\sup_{\B_\frac{1}{2}(0)} \overline{u}_i^\ep\leq C_n\left(\inf_{\B_\frac{1}{2}(0)} \overline{u}_i^\ep+\frac{Ce^{-c\ep^{-\frac{1}{2}}}}{\ep^2}\right)\leq 
C_n\left(\frac{2\sup_{\partial\Om} f_i}{\tau r}+\frac{Ce^{-c\ep^{-\frac{1}{2}}}}{\ep^2}\right)\leq \frac{C}{r}.$$
Lipschitz estimates then imply that 
$|\nabla  \overline{u}_i^\ep|\leq C/r$  in $\B_\frac{1}{2}(0)$ and\eqref{lipcor5.4} follows.

Further, \eqref{deltauepcor5.4} implies \eqref{harmoniccor5.4}. 
\end{proof}

The next lemma says that in a $\rho$-strip of size 1 of support of the $f_j$'s, the function $u^\ep_i$, $i\neq j$, decays to 0 exponentially. 
\begin{lem}\label{uj=0closfi}   Assume \eqref{mainassumpts}. Let $(u_1^\ep,\ldots,u_K^\ep)$ be a
viscosity solution of the problem \eqref{eq:stmt3}.  For $j=1,\ldots,K$, $\sigma>0$,  let $\overline{\Gamma}_j^\sigma:=\{f_j\ge\sigma\}\subset\Om^c.$ Then  on the sets 
$$\{x\in\Om\,:\,d_\rho(x,\overline{\Gamma}_j^\sigma)\leq 1-r\},\quad 0<r<1$$
we have 
$$u_i^\ep\leq Ce^{-\frac{c\sigma^\alpha r^\beta }{\ep}},\quad\text{for }i\neq j,$$
for some positive $\alpha$ and $\beta$ depending on the structure of $H$ ($p$ and $q$) and the modulus of continuity of $f_j$. 
\end{lem}
\begin{proof}
Let $\overline{x}\in\Om$  and $y\in \overline{\Gamma}_j^\sigma$ be such that $d_\rho(\overline{x},y)\leq 1-r$. We want to estimate $H(u^\ep_j)(x)$, for any $x\in\B_\frac{r}{2}(\overline{x})$. Let $x\in\B_\frac{r}{2}(\overline{x})$, then 
\begin{equation}\label{firstequlemma4.5}
d_\rho(x,y)\leq 1-\frac{r}{2}.
\end{equation}
Let us first consider the case $H$ defined as in \eqref{H2}. We have
 $$H(u_j^\ep)(x)=\sup_{ \B_1(x)}u_j^\ep\geq f_j(y)\geq \sigma.$$ 
 Next, let us turn to the case $H$ defined as in \eqref{H1}. 
Let  $r_0:=\min\{\sigma^\gamma,r/4\}$, for some $\gamma$ depending on the  modulus of continuity of $f_j$, 
 then  $f_j\geq \sigma/2$ in the set 
$ \B_ {r_0}(y)\cap \supp\,f_j$.   Moreover, remark that from \eqref{firstequlemma4.5} and  $r_0\le r/4$, we have
$$ \B_ {r_0}(y)\cap \supp\,f_j\subset \B_\frac{r}{4}(y)\subset \B_\frac{r}{2}(y)\subset \B_1(x),$$
and for any $z\in  \B_ {r_0}(y)\cap \supp\,f_j$
$$\rho(x-z)\leq \rho(x-y)+\rho(y-z)\le1-\frac{r}{2}+r_0\leq 1-\frac{r}{4}.$$   Therefore, using in addition  \eqref{varphidecay}, and that,
by \eqref{mainassumpts}, $|\B_ {r_0}(y)\cap \supp\,f_j|\ge c|\B_ {r_0}(y)|$,  we get
\begin{equation*}\begin{split}H(u_j^\ep)(x)&=
\int_{ \B_1(x)}(u_j^\ep)^p(z)\varphi(\rho(x-z))dz\\&
\geq \int_{ \B_ {r_0}(y)\cap \supp\,f_j}(u_j^\ep)^p(z)(1-\rho(x-z))^qdz\\&
\geq \int_{ \B_ {r_0}(y)\cap \supp\,f_j}(f_j)^p(z)C\left(\frac{r}{4}\right)^q dz\\&
\geq C\sigma^pr_0^{\overline{\beta}},\end{split}\end{equation*} where $\overline{\beta}$ depends on $q$ and on the dimension $n$.

Then, for  $H$ defined as in  \eqref{H1} or 
\eqref{H2}, the function  $u_i^\ep$, $i\neq j$ (which is eventually zero in $B_\frac{r}{2}(\overline{x})\cap\Om^c$) is subsolution of $$\Delta u_i^\ep\ge u_i^\ep \frac{ C\sigma^pr_0^{\overline{\beta}}}{\ep^2}$$
in $B_\frac{r}{2}(\overline{x})$, where $p=1$ and $\overline{\beta}=0$ in the case \eqref{H2}. 
The conclusion follows as in Lemma \ref{uj=0closui}.
\end{proof}

The following corollary is a consequence of Lemma  \ref{uj=0closui}, Corollary \ref{convedisjointcor} and Lemma \ref{uj=0closfi}.

\begin{cor}\label{convergencecor} Assume \eqref{mainassumpts}.  Let $(u_1^\ep,\ldots,u_K^\ep)$ be a 
viscosity solution of the problem  \eqref{eq:stmt3}.
Then, there exists a subsequence $(u_1^{\ep_1},\ldots,u_K^{\ep_l})$ and continuous functions $ (u_1,\ldots, u_K)$ such that, 
\begin{equation*} (u_1^{\ep_l},\ldots,u_K^{\ep_l})\to  (u_1,\ldots, u_K)\quad\text{as }l\to+\infty,\quad\text{a.e. in }\Om\end{equation*} and the convergence of 
$u_i^{\ep_l}$ to $u_i$ is locally uniform  in the set  $\{x\in\Om\,:\,d_\rho(x,\text{supp}\,f_j)>1,\,j\neq i\}$.
Moreover, we have:
\begin{itemize}
\item[i)] the $u_i$'s  are  locally Lipschitz continuous in $\Om$ and have disjoint supports, in particular 
$$u_i\equiv 0\quad\text{in the set }\quad  \{x\in\Om\,|\,d_{\rho}(x,\supp\, u_j)\le1\}\quad\text{for any }j\neq i.$$
\item[ii)] $\Delta u_i =0$ when $u_i >0$.
\end{itemize}
\end{cor}
\begin{proof}
Fix an index $i=1,\ldots,K$. Let us denote 
$$\Omega_i:=\{x\in\Om\,|\,d_\rho(x,\text{supp}\,f_j)>1\text{ for any }j\neq i\},$$ and 
$$B_i:=\Omega\setminus \overline{\Omega}_i.$$
{\em Claim 1:  $u_i^\ep(x)\to 0$ as $\ep\to0$ for any $x\in B_i$.}

Indeed, let $x_0$ belong to  $B_i$, then there exists $j\neq i $ such that $d_\rho(x_0,\text{supp}\,f_j)<1.$ 
Remark that 
$$\{x\in\Om\,|\,d_\rho(x,\text{supp}\,f_j)<1\}\subset \cup_{r,\sigma>0}\{x\in\Om\,|\,d_\rho(x,\overline{\Gamma}_j^\sigma)\leq1-r\},$$ where 
$\overline{\Gamma}_j^\sigma=\{f_j\ge\sigma\}.$
Therefore, there exist $r,\sigma>0$ such that $x_0\in \{x\in\Om\,|\,d_\rho(x,\overline{\Gamma}_j^\sigma)\leq1-r\},$ and by Lemma \ref{uj=0closfi} we have that 
$u_i^\ep(x_0)\leq Ce^{-\frac{c\sigma^\alpha r^\beta }{\ep}}$,  for some $\alpha,\beta>0$. Claim 1 follows.

{\em Claim 2: there exists a subsequence $(u_i^{\ep_l})_l$   locally uniformly convergent in $\Omega_i$ as $l\to+\infty$ to a locally Lipschitz continuous function $u_i$.}

Fix, $0<r<1$, 
and define 
$$\Omega^{r}_i:=\{x\in \Omega_i\,|\,d_\rho(x, \partial \Omega)>2r,\,d_\rho(x,\supp\,f_j)\geq 1+\tau r\text{ for any }j\neq i \},$$

Fix $\theta<\frac{1}{2\alpha}$ and set $\sigma_\ep=\ep^\theta>0$
and consider
$\tau,\,\alpha$  and $\beta$ as given by Lemma  \ref{uj=0closui}. Since $\ep=\sigma_\ep^{2\alpha}\sigma_\ep^{\frac{1}{\theta}-2\alpha}=\sigma_\ep^{2\alpha}\ep^{\theta(\frac{1}{\theta}-2\alpha)}$ and $\frac{1}{\theta}-2\alpha>0$, 
we can fix  $\ep_0=\ep_0(r)$ so small that  for any $\ep<\ep_0$ we have that $\ep\leq \sigma_\ep^{2\alpha}r^{2\beta}$. Then,  by Corollary \ref{convedisjointcor}, the functions 
$$v_i^\ep:=(u_i^\ep-\sigma_\ep)_+=(u_i^\ep-\ep^\theta)_+$$  are Lipschitz continuous in $\Omega_i^{r}$. 
Indeed, when $u_i^\ep (x)<\ep^\theta$, then $v_i^\ep(x)=0$. Next, let $x$ such that 
 $u_i^\ep(x)>\ep^\theta$, then $\nabla v_i^\ep(x)=\nabla u_i^\ep(x)$. Set $\sigma=u_i^\ep(x)$, then at those points $x$, we have that
$d_\rho(x,\text{supp}\,f_j)\geq 1+\tau r\geq  1+m \tau r$, where $m=\sigma/\sup_{\partial\Om}f_i\leq 1$. Moreover, $d_\rho(x,\partial\Om)>2r$ and 
$\ep\leq \sigma_\ep^{2\alpha}r^{2\beta}\leq  \sigma^{2\alpha}r^{2\beta}$. We can therefore apply Corollary \ref{convedisjointcor} and we get that
$$|\nabla u_i^\ep(x)|\leq \frac{C_0}{r}.$$
 This concludes the proof that the functions $v^\ep_i$ are Lipschitz continuous in $\Omega_i^{r}$. 
Therefore, we can extract a subsequence $(v_i^{\ep_l})_l$ uniformly convergent to a Lipschitz  continuous function 
$u_i$ in $\Omega_i^{r}$ as $l\to+\infty$. By the definition of the $v_i$'s, this implies that there exists a subsequence  $(u_i^{\ep_l})_l$ uniformly convergent to the same function $u_i$ in 
$\Omega_i^{r}$ as
$l\to+\infty$. Taking $r \rightarrow 0$ and using a diagonalization argument, we can find a subsequence of $(u_i^{\ep})_\ep$ converging locally uniformly to a Lipschitz function $u_i$ in $\Om_i$. 
This ends the proof of Claim 2.

Claims 1 and 2 yield the convergence, up to a subsequence, of $u_i^\ep$ to a  continuous function $u_i$ which is locally Lipchitz  in both $\Om_i$ and $B_i$.
The fact that the supports of the limit functions are at distance greater or equal than 1, is a consequence of Claims 1 and 2 and Lemma \ref{uj=0closui}. 
Moreover, from the proof of Claim 2 and Corollary \ref{convedisjointcor}, we infer that the limit function $u_i$ is harmonic inside its support, i.e. (ii). 
To conlude the proof of (i), we just need to prove that $u_i$ is Lipschitz in a neighborhood of points belonging to $\p B_i=\p \Om_i \cap \Om$. Let $x_0\in\partial\Om_i\cap\Om$, then from Claim 1, $u_i(x_0)=0$. If $x_0\not\in\partial \{u_i>0\}$, then in a neighborhood 
of $x_0$, $u_i\equiv 0$ and of course  it is Lipschitz there. On the other hand, if $x_0\in\partial \{u_i>0\}$, then, since there exists an exterior $\rho$-tangent ball
of radius 1 at any point of  $\partial\Om_i\cap\Om$ and $u_i$ is harmonic inside its support, a barrier argument implies that there exist $r_0,\,C>0$ such that 
$0\leq u_i(x)=u_i(x)-u_i(x_0)\leq C|x-x_0|$ for any $x\in B_{r_0}(x_0)$. This concludes the proof of (i).

This  concludes the proof of the corollary.

\end{proof}


\section{A semiconvexity property of the free boundaries}
Let $(u_1,\ldots,u_K)$ be the limit of a convergent subsequence of $(u^\ep_1,\ldots,u^\ep_K)$, whose existence is guaranteed by Corollary \ref{convergencecor}.
For $i=1,\ldots,K$, let us denote 
\begin{equation}
\label{Si}
S(u_i):=\{x \in \Omega : u_i>0\}.
\end{equation}
(In the next sections, for simplicity  this set will be represented by $S_i.$) Then the sets $S(u_i)$ have the following semiconvexity property:
\begin{lem}
\label{semiconvexproplem}
Given $S(u_i)$ consider 
$$T(u_i) = \big\{ x\in\Om\,:\, d_\rho(x,S(u_i)) \geq 1\big\}$$
and 
$$S^* (u_i) = \big\{ x\in\Om\,:\, d_\rho(x,T(u_i))> 1\big\}$$
Then 
$\partial S(u_i)\subset \partial S^* (u_i).$
\end{lem}

\begin{proof}
We have that $S^* (u_i) \supset S(u_i)$. 
To prove the 
desired inclusion, for $\sigma>0$ consider the sets
\begin{gather*}
S_\sigma (u_i): = \{ x \in \Omega: u_i >\sigma\}\ ,\\
T_\sigma (u_i): = \{ x\in\Om : d_\rho(x,S_\sigma (u_i )) \geq1\}\\
\noalign{\vskip-6pt}
\intertext{and}
\noalign{\vskip-6pt}
S^*_\sigma (u_i): = \{x\in\Om\,:\, d_\rho(x,T_\sigma  (u_i ) ) >1\}.
\end{gather*}
 Notice that, the union of $\rho$-balls centered at points in  $S_\sigma (u_i)$ coincides with the union of $\rho$-balls centered at points in $S_\sigma^* (u_i)$, i.e.

a) $ (T_\sigma (u_i))^c = \cup\, \B_1 (x)$ for $x\in S_\sigma (u_i)$ and 

b) $(T_\sigma (u_i))^c=\cup\, \B_1 (x)$ for $x\in S_\sigma^* (u_i)$.\\
If $x\in S_\sigma (u_i)$,   from (i) of Corollary \ref{convergencecor} we have that $d_\rho(x,\text{supp}f_j)>1$ for $j\neq i$, and  the locally uniform convergence of $u^\ep_i$ to $u_i$ and Lemma \ref{uj=0closui} imply that, up to subsequences,    $u^\ep_j\le Ce^{-\frac{c\sigma^\alpha r^\beta}{\ep}}$ in $\B_1 (x)$, where $2r=\min\{d_\rho(x,\text{supp}f_i),C(d_\rho(x,\text{supp}f_j)-1)\}$.
 Now, the set where $u^\ep_j$ decays is the same if we had considered $x\in S^*_\sigma(u_i)$, since from (a) and (b) we have
$$\cup_{x\in S_\sigma (u_i)} \B_1 (x)=\cup_{x\in S^*_\sigma(u_i)} \B_1 (x).$$ Therefore  $\frac{H(u_j^\ep)}{\ep^2}$ goes to zero as $\ep$ goes to zero in $S^*_\sigma (u_i)$.
It follows that $\Delta u_i\equiv 0$ in $S_\sigma^* (u_i)$, if $S_\sigma^* (u_i)$ is not empty.  
Now, if $A$ is a connected component of $S_\sigma(u_i)$, then there exists a connected component $A^*$ of $S^*_\sigma (u_i)$ such that $A\subset A^*$. 
Since 
$u_i$ is harmonic and non-negative in $A^*$, 
the strong maximum principle implies that $ u_i>0$ in all  $A^*$, 
that is $A^*\subset A$. 
 We conclude that $A= A^*$. Therefore, any connected component of $S_\sigma(u_i)$ is equal to a connected component of $S^*_\sigma(u_i)$.
 Passing to the limit as $\sigma\to0$, we obtain that   any connected component of $S(u_i)$ is equal to a connected component of $S^*(u_i)$. In particular, 
$\partial S(u_i)\subset \partial S^*(u_i)$.
\end{proof}
From  Lemma \ref{semiconvexproplem} we can conclude that the sets $S(u_i)$ have a tangent $\rho$-ball   of radius 1 from outside at any point of the boundary, as stated in the following corollary. 
\begin{cor}\label{tangebtballcor}
If $x_0 \in \partial S(u_i)\cap\Om$ there is an exterior  tangent ball, $\B_1 (y)$ at $x_0$, in the sense that for  $x \in  \B_1 (y) \cap \B_1 (x_0)$,  all $u_j(x)\equiv 0$ (including $u_i$).
\end{cor}

The following two lemmas about the distance function may be known in the literature and we provide the proof here for the reader's convenience.
\begin{lem}\label{deltadlem}
Let $S$ be a closed set. Let $d_\rho(\cdot, S)$ denote the $\rho$-distance function from $S$. Then, in the set $\{x: d_\rho(x, S)>0\}$, $d_\rho(\cdot, S)$ satisfies in the viscosity sense
\begin{equation*}\Delta d_\rho(\cdot, S)\leq\frac{C}{d_\rho(\cdot, S)}, 
\end{equation*} 
where $C$ is a  constant depending on $n$, $\|D d_\rho(\cdot, S) \|_{L^\infty}$ and the constant $A$ given in \eqref{eq:stmt1}.
\end{lem}
\begin{proof}
We first prove that there exists a smooth tangent function from above at any point of the graph of $d_\rho(\cdot, S)$ in the set $\{d_\rho(\cdot, S)>0\}$. For simplicity we will write $d_S(\cdot)$ instead of $d_\rho(\cdot, S)$.
Let $y_0$ be a point in the complementary of $S$.   Let $x\in\partial S$ be a point where $y_0$ realizes the distance from $S$. 
Assume, without loss of generality, that $x=0$. Then $d_\rho(y_0,0)=\rho(y_0)=d_S(y_0)$. In particular, the ball $\B_{\rho(y_0)}(y_0)$ is contained in $S^c$ and tangent to $S$ at 0.
For any $y\in \B_{\rho(y_0)}(y_0)$, we have that $d_S(y)\leq d_\rho(y,0)=\rho(y)$, therefore the cone,  graph of the function $y\to\rho(y)$, is tangent from above 
to the graph of  $d_S(\cdot)$ at $(y_0, d_S(y_0))$.

Next, let $\psi$ be a test function touching from below $d_S(\cdot)$ at $y_0$, then $\psi$ touches  from below the function $ \rho(y)$ at $y_0$. In particular,
$\Delta \psi(y_0)\leq \Delta \rho(y_0)$. Let us compute $\Delta \rho$. Using  \eqref{eq:stmt1}, we get
$$D^2(\rho)=\frac{1}{\rho}D^2\left(\frac12 \rho^2\right)-\frac{D\rho\otimes D\rho}{\rho}\leq\frac{1}{\rho}(AI_n-D\rho\otimes D\rho),$$
which gives $$\Delta \rho\leq \frac{C}{\rho}.$$ 
In particular,
$$\Delta \psi(y_0)\leq  \frac{C}{\rho(y_0)}=\frac{C}{d_S(y_0)}.$$
This concludes the proof.
\end{proof}

\begin{lem}\label{finiteperlem}
Let $S$ be a closed and bounded set.  Let us denote by $d_\rho(\cdot, S)$ the $\rho$-distance function from $S$  and by 
$(S)_1$ the set at $\rho$-distance 1 from $S$. Then 
 $ (S)_1$ has finite perimeter.
 \end{lem}

\begin{proof}
For simplicity we will write $d_S(\cdot)$ instead of $d_\rho(\cdot, S)$, as in the previous lemma and first we present an heuristic proof  integrating  $\Delta d_S^2$ over the set  $\{0<d_S<1\}$. Since $|D d_S|$ is bounded, from Lemma \ref{deltadlem}, we see that
$$\Delta d_S^2=2|D d_S|^2+2d_S\Delta d_S\leq C.$$
Therefore, integrating $\Delta d_S^2$, we get
\begin{equation*}\begin{split}
C&\geq \int_{\{0<d_S<1\}} \Delta d_S^2dx=\int_{\{d_S=0\}}2d_S Dd_S\cdot n\, d\mathcal{H}^{n-1}+\int_{\{d_S=1\}} 2d_S Dd_S\cdot n\, d\mathcal{H}^{n-1}
\\&=\int_{\{d_S=1\}} 2 Dd_S\cdot n \,d\mathcal{H}^{n-1}\geq c\int_{\{d_S=1\}}d{\mathcal{H}}^{n-1}=c\mathcal{H}^{n-1}(\{d_S=1\}),
\end{split}
\end{equation*}
where $n=Dd_S/|Dd_S|$ is the unit exterior normal. This provides un upper bound for $\mathcal{H}^{n-1}(\{d_S=1\})$ and concludes the heuristic  proof.

To make the argument precise, we need  to correct the regularity problem over the boundary. For that, consider a smooth function $\eta$ with compact support in $(0,1)$ such that $0\leq \eta(\xi)\leq\xi$ for any $\xi\in[0,1]$,
$\eta(\xi)=\xi$ for $\xi\in[\delta,1-\delta]$, $|\eta'|\leq c$  on $(0, 1-\delta)$ and $\eta'(\xi)\leq- c/\delta$ for $\xi\in (1-\delta,1)$,
where $\delta>0$ is a small parameter.  Then, in a weak sense we have
\begin{equation}\label{lap distance}\text{div}(\eta(d_S)Dd_S)=\eta'(d_S)|Dd_S|^2+\eta(d_S)\Delta d_S.\end{equation}
Moreover, from Lemma \ref{deltadlem}, in the set $\{0<d_S<1\}$ we have $$\eta(d_S)\Delta d_S\leq \eta(d_S)\frac{C}{d_S}\leq C$$ in the viscosity sense and therefore in the distributional sense (see, e.g., 
\cite{ishii_equivalence_1995} for the equivalence between viscosity solutions and weak solutions).
Therefore, since $\eta (d_S)$ is a function with compact support in  $\{0<d_S<1\}$, we get
\begin{equation}\label{dist finite perim}\begin{split}
0&=\int_{\{0<d_S<1\}}\text{div}(\eta(d_S)Dd_S)\leq \int_{\{0<d_S<1\}}\eta'(d_S)|Dd_S|^2dx+C\\&
= \int_{\{0<d_S<1-\delta\}}\eta'(d_S)|Dd_S|^2dx+ \int_{\{1-\delta<d_S<1\}}\eta'(d_S)|Dd_S|^2dx+C\\&
\leq \int_{\{1-\delta<d_S<1\}}\eta'(d_S)|Dd_S|^2dx+C\\&
\leq -\frac{c}{\delta} \int_{\{1-\delta<d_S<1\}}|Dd_S|^2dx+C.
\end{split}
\end{equation}
 Now, using the coarea formula and the inequalities above, we get
$$\frac{1}{\delta}\int_{1-\delta}^1\mathcal{H}^{n-1}(\{d_S=t\})dt=\frac{1}{\delta} \int_{\{1-\delta<d_S<1\}}|Dd_S|^2dx\leq C.$$
Finally, taking the limit as $\delta\to0^+$ and using the lower semicontinuity of the perimeter with respect to the convergence in measure, we infer that 
$$\text{Per}(\{d_S=1\})\leq \liminf_{\delta\to0^+}\frac{1}{\delta}\int_{1-\delta}^1\mathcal{H}^{n-1}(\{d_S=t\})dt\leq C.$$
This concludes the proof of the lemma.
\end{proof}

\begin{cor}\label{hausmescor}
The sets $ S(u_i)$, $i=1,\ldots,K$ have finite perimeter. 
\end{cor}
\begin{proof} The corollary is an immediate consequence of Lemma \ref{semiconvexproplem} and Lemma \ref{finiteperlem}.
 \end{proof}

\section{A sharp characterization of the interfaces}
\label{section: strip}
In Section \ref{Lipestsection} we proved that the supports of the limit functions $u_i$'s are at distance at least 1,  one from  each other (see Corollary \ref{convergencecor}). In this section we will prove that they are exactly at distance 1, as stated in the following theorem. 
\begin{thm}\label{lem:6.1} Assume \eqref{mainassumpts} with  $p=1$ in \eqref{H1}.  
Let $(u_1^\ep,\ldots,u_K^\ep)$ be a 
viscosity solution of the problem \eqref{eq:stmt3}  and $(u_1,\ldots, u_K)$ the limit as $\ep\to0$  of a convergent subsequence.  
Let $x_0 \in \partial \{u_i >0\}\cap\Om$, 
then there exists $j\ne i$ such that 
\begin{equation}\label{supportdistex1}\overline{\B_1 (x_0)} \cap \partial \{u_j >0\} \ne \emptyset\ .\end{equation}
\end{thm}

\begin{proof} It is enough to prove the theorem for a point $x_0$ for which $\partial S(u_i)$ has a tangent $\rho$-ball 
from inside, since such points are dense on $\partial S(u_i)$.  Indeed, let $x$ be any point of  $\partial S(u_i)$. Let us consider a sequence of points $(x_k)$ contained in 
$S(u_i)$ and converging to $x$ as $k\to \infty$. Let $d_k$ be the $\rho$-distance of $x_k$ from $\partial S(u_i)$. Then the $\rho$-balls $\B_{d_k}(x_k)$ are contained in 
$S(u_i)$  and there exist points $y_k\in \partial S(u_i)\cap \B_{d_k}(x_k)$ where the $x_k$'s realize the distance from $\partial S(u_i)$. The sequence $(y_k)$ is a sequence of  points  of $\partial S(u_i)$ that have a tangent $\rho$-ball from the inside  and converges to $x$.

Next, remark that from (b) in Corollary \ref{convergencecor}, we have that $d_\rho(x_0,\text{supp}\,f_j)\geq 1$ for any $j\neq i$. If there is a $j$ such that 
$d_\rho(x_0,\text{supp}\,f_j)=1$,  then \eqref{supportdistex1} is obviously true.
Therefore, we can assume that  $d_\rho(x_0,\text{supp}\,f_j)> 1$ for any $j\neq i$. Then, for small $S>0$ we have that $\B_{1+S}(x_0)\cap\supp\,f_j=\emptyset$ and from \eqref{eq:stmt3}, we know that $$\Delta u_j^\ep \geq \frac1{\ep^2} u_j^\ep  \sum_{k\ne j} H(u_k^\ep) \quad\text{in }\B_{1+S}(x_0).$$

We divide the proof in two cases.

a) $\ds H(u) (x) = \int_{\B_1 (x)} u(y) \varphi \big( \rho (x-y)\big) \, \mbox{\rm d}y$

and 

b) $\ds H (u )(x) = \sup_{y\in\B_1 (x)} u(y)\ .$
\medskip

{\em Proof of case} a):    
 Let $S(u_i) = \{x\in \Omega: u_i >0\}$ as in (\ref{Si}). Let $\B_S$ be a small  $\rho$-ball centered at $x_0 \in \partial S(u_i)$.
Then, as a measure, as $\ep\rightarrow0$, up to subsequence
$$\Delta u_i^\ep \big |_{ \B_S (x_0)} \longrightarrow 
\Delta u_i\big|_{ \B_S (x_0)} $$
(that has strictly positive mass, since $u_i$ is not harmonic in $\B_S (x_0)$).

We bound by below 
$$\int_{\B_{1+S}(x_0)} \sum_{j\ne i} \Delta u^\ep_j dx\quad\text{ by }\quad 
\int_{\B_S (x_0)} \Delta u_i^\ep dx.$$
Indeed

\begin{equation}\label{boundlaplacian}
\begin{split}
 \ep^2 \int_{\B_S (x_0)} \Delta u^\ep_i (x)dx&= \sum_{j\ne i}\int_{\B_S (x_0)} \int_{\B_1(x)} u_i^\ep (x) \varphi \big( \rho (x-y)\big) u_j^\ep (y)dydx\\
&=\sum_{j\ne i}\int\int_{\B_S (x_0)\times \B_{1+S} (x_0)} u_i^\ep (x)\chi_{[0,1]} \big( \rho (x-y)\big) \varphi \big( \rho (x-y)\big) u_j^\ep (y)dxdy\\
&\leq \sum_{j\ne i}\int\int_{\B_{2+S} (x_0) \times \B_{1+S} (x_0)} u_i^\ep (x) \chi_{[0,1]}  \big( \rho (x-y)\big)\varphi \big( \rho (x-y)\big) u_j^\ep (y)dxdy\\
&=\sum_{j\ne i}\int_{\B_{1+S} (x_0)}\int_{\B_1(y)}u_i^\ep (x) \varphi \big( \rho (x-y)\big) u_j^\ep (y)dxdy\\
&\le\ep^2\sum_{j\ne i}\int_{\B_{1+S} (x_0)}\Delta u^\ep_j(y)dy,
\end{split}
\end{equation} 
where $\chi_{[0,1]}$ is the indicator function of the set $[0,1]$.

%

Therefore, for any small positive $S$, taking the limit in $\ep$ we get 
$$\int_{\B_{1+S}(x_0)} \sum_{j\ne i}\ \Delta u_j \ge \int_{\B_{S}(x_0)}\Delta u_i>0 $$
which implies that there exists $j\neq i$ such that $u_j$ cannot be identical equal to zero in $\B_{1+S}(x_0)$.  Since $S$ small is arbitrary, the result follows.

The case b) is more involved. We may assume $x_0=0$. 
Let $y_0$ be such that $\B_\mu (y_0) \subset S(u_i)$ and $0\in \partial \B_\mu (y_0)$. 
By Corollary \ref{tangebtballcor}
we know that there exists a $\rho$-ball $\B_1 (y_1)$ such that $\B_1 (y_1) \cap S(u_i) = \emptyset$ and 
$0\in \partial \B_1 (y_1)$.

Let us first prove two claims.

{\em Claim 1:} 
There exists $\mu' < \mu$ and $C_1 >0$ such that in the annulus $\{ \mu' < \rho (x-y_0) <\mu\}$ 
we have 
$$u_i (x) \ge C_1 d_\rho \big( x,\partial \B_\mu (y_0)\big)\ .$$
Since any  $\rho$-ball $\B$ satisfies the uniform interior ball condition, for any point $\bar{x} \in \partial \B_\mu (y_0)$ 
there exists an Euclidian  ball $B_{R_0} (z_0)$ of radius $R_0$ independent of $\bar{x} $ contained in 
$\B_\mu (y_0)$ and tangent to $\partial \B_\mu (y_0)$ at $\bar{x} $. 
Let $m>0$ be the infimum of $u_i$ on the set $\{ x\in \B_\mu (y_0) \mid d(x,\partial \B_\mu (y_0)) \ge R_0/2\}$, where $d$ is the Euclidian distance function,
and let $\phi$ be the solution of 
$$\begin{cases} 
\Delta \phi =0&\text{in }\quad \ds \Big\{ \frac{R_0}2 < |x-z_0| < R_0\Big\}\\
\noalign{\vskip6pt}
\phi =0&\text{on }\quad \partial B_{R_0} (z_0)\\
\noalign{\vskip6pt}
\phi = m&\text{on }\quad \partial B_{\frac{R_0}2} (z_0)
\end{cases}$$
i.e., for $n\ge 3$,
$$\phi (x) = C(n) m \Big( \frac{R_0^{n-2}}{|x-z_0|^{n-2}} -1\Big)\ .$$
Since $u_i$ is harmonic in $\B_\mu (y_0)$ and $u_i\ge \phi$ on $\partial B_{R_0} (z_0)\cup \partial B_{\frac{R_0}2} (z_0)$,  by comparison principle $u_i \ge \phi$ in $\{ \frac{R_0}2 < |x-z_0|<R_0\}$.
In particular, for any $x\in\{ \frac{R_0}2 < |x-z_0|<R_0\}$ and belonging to  the segment between $z_0$ and $\bar{x} $, using that $\phi$ is convex 
in the radial direction,

$$
\frac{\partial \phi}{\partial \nu_i}|_{\partial B_{R_0}(z_0)}=\frac{C(n)(n-2) m}{R_0}
$$
 where $\nu_i$ is the interior normal at $\partial B_{R_0}(z_0)$, and (2.2), we get 
 $$
 u_i(x) \ge \frac{C(n)(n-2) m}{R_0} d (x,\partial B_{R_0} (z_0)) =C(n,R_0) md (x,\partial \B_\mu (y_0)) \ge C_1 d_\rho (x,\partial \B_\mu (y_0))\ .
 $$
 
\noindent Therefore, letting $\bar{x} $ vary in $\partial \B_\mu (y_0)$ we get 
$$u_i(x) \ge C_1 d_\rho (x,\partial \B_\mu (y_0))\quad\text{for any }x\in \B_\mu (y_0)\text{ with }d(x,\partial \B_\mu (y_0))\leq\frac{R_0}{2}.$$
Using  \eqref{eq:stmt2}, 
 Claim~1 follows.
\medskip

%

Next, let $e_0 = y_0/\rho (y_0)$ and fix $\sigma <\mu$ so small that $\B_{\sigma} (\sigma e_0) 
\subset \{ \mu' < \rho (x-y_0) < \mu\} \cap \B_{1+\delta} (y_1)$.
For $r\in [\sigma -\upsilon, \sigma +\upsilon]$  and small $\upsilon<\sigma$, let us define 
$$\underline{u}_i^\ep ¨ : = \inf_{\partial \B_r (\sigma e_0)} u_i^\ep \quad\text{ and }\quad 
\underline{u}_i ¨ := \inf_{\partial \B_r (\sigma e_0)} u_i\ .$$
Since for $r\in [\sigma,\sigma+\upsilon]$, $\partial \B_r (\sigma e_0)\cap (S(u_i))^c\neq\emptyset$ and $u_i\equiv 0$ on $(S(u_i))^c$,  we have 
\begin{equation}\label{ui=zeror>sigma}\underline{u}_i ¨ = 0\ \text{ for }\ r\in [\sigma,\sigma+\upsilon]\ .\end{equation}
By Claim 1, we know that in $B_\sigma (\sigma e_0)$ we have
\begin{align*}
u_i(x) & \ge C_1 d_\rho (x,\partial \B_\mu (y_0))\\&
\geq C_1d_\rho (x,\partial \B_\sigma (\sigma e_0))\\&
=C_1(\sigma-\rho(x-\sigma e_0)).
\end{align*}
We deduce that for $r\in [\sigma -\upsilon, \sigma ]$
$$\underline{u}_i ¨ = \inf_{\partial \B_r (\sigma e_0)} u_i\ge \inf_{\partial \B_r (\sigma e_0)}C_1(\sigma-\rho(x-\sigma e_0))=C_1(\sigma-r).$$
From the previous inequality and \eqref{ui=zeror>sigma}, we infer that
\begin{equation}\label{eq:6.1}
\underline{u}_i ¨ \ge C_1 (\sigma -r)^+, \quad r\in [\sigma -\upsilon, \sigma +\upsilon].
\end{equation}
Next, for $j\ne i$, $r \in [\sigma -\upsilon,\sigma +\upsilon]$, let us define 
$$\bar u_j^\ep ¨ : = \sup_{\partial \B_{1+r}(\sigma e_0)} u_j^\ep \quad \text{ and }\quad 
\bar u_j ¨ : = \sup_{\partial \B_{1+r} (\sigma e_0)} u_j\ .$$
The functions $\underline{u}_i^\ep ¨$ and $\bar u_j^\ep ¨$ are respectively solutions of 
\begin{equation}\label{supinfinethm6.1}\begin{split}
\Delta_r \underline{u}_i^\ep \le \frac1{\ep^2} \underline{u}_i^\ep \sum_{i\ne j} \sup_{\B_1 (\underline{z}_r^i)}
u_j^\ep \\
\Delta_r \overline{u}_j^\ep \ge \frac1{\ep^2} \overline{u}_j^\ep \sup_{\B_1 (\bar z_r^j)} u_i^\ep 
\end{split}\end{equation}
where
 $$\Delta_r u =u_{rr} +\frac{(n-1)}{r}u_r= \frac1{r^{n-1}} \ \frac{\partial}{\partial r} \Big( r^{n-1} \frac{\partial u}{\partial r}\Big)$$
and $\underline{z}_r^i$ and $\bar z_r^j$ are respectively the points where the infimum of $u_i^\ep$ 
on $\partial \B_r (\sigma e_0)$ and the supremum of $u_j^\ep$ on $\partial \B_{1+r} (\sigma e_0)$ 
are attained.  Note that in spherical coordinates $$\Delta u =\Delta_r u + \Delta_{\theta} u$$ and that if we are on a point where $u$ attains a minimum value in the $\theta$  for a fixed $r$ then $\Delta_{\theta} u\geq 0$ and the opposite inequality holds  if we are on a maximum point.
We also remark that 
$$\overline{y}_r^j := \sigma e_0 + \frac{r}{r+1} (\bar z_r^j - \sigma e_0) \in \partial \B_r 
(\sigma e_0) \cap \partial \B_1 (\bar z_r^j)\ ,$$
therefore
\begin{equation}\label{supinfinethm6.12}\sup_{\B_1 (\bar z_r^j)} u_i^\ep \ge u_i^\ep (\bar y_r^j) \ge \underline{u}_i^\ep ¨\ .
\end{equation}
Moreover, since $\B_1 (\underline{z}_r^i) \subset \B_{1+r} (\sigma e_0)$ and $u_j^\ep$ is a 
subharmonic function, we have 
\begin{equation}\label{supinfinethm6.13}\begin{split}
\sup_{\B_1 (\underline{z}_r^i)} u_j^\ep 
& \le \sup_{\B_{1+r}(\sigma e_0)} u_j^\ep\\
& = \sup_{\partial \B_{1+r}(\sigma e_0)} u_j^\ep\\
& = \bar u_j^\ep ¨\ .
\end{split}
\end{equation}
From \eqref{supinfinethm6.1},  \eqref{supinfinethm6.12} and \eqref{supinfinethm6.13}, we conclude that 
\begin{equation}
\label{eq: sup case estimate}
\Delta_r \underline{u}_i^\ep \le \Delta_r \bigg( \sum_{j\ne i} \bar u_j^\ep\bigg)\ .
\end{equation}
In other words, for any $\phi \in C_c^\infty (\sigma -\upsilon,\sigma+\upsilon)$, $\phi \ge 0$, we have 
$$\int_{\sigma-\upsilon}^{\sigma +\upsilon} \underline{u}_i^\ep ¨ \frac{\partial}{\partial r} 
\left( r^{n-1} \frac{\partial}{\partial r} \Big( \frac1{r^{n-1}} \phi ¨\Big) \right) dr
\le \int_{\sigma-\upsilon}^{\sigma +\upsilon} \sum_{j\ne i} 
\bar u_j^\ep¨ \frac{\partial}{\partial r} 
\left( r^{n-1} \frac{\partial}{\partial r} \Big( \frac1{r^{n-1}} \phi ¨\Big) \right) dr\ .$$
Passing to the limit as $\ep \to 0$ along a   uniformly converging subsequence, we get 
$$\int_{\sigma-\upsilon}^{\sigma +\upsilon} \underline{u}_i ¨ \frac{\partial}{\partial r} 
\left( r^{n-1} \frac{\partial}{\partial r} \Big( \frac1{r^{n-1}} \phi ¨\Big) \right) dr
\le \int_{\sigma-\upsilon}^{\sigma +\upsilon} \sum_{j\ne i} 
\bar u_j ¨\frac{\partial}{\partial r} 
\left( r^{n-1} \frac{\partial}{\partial r} \Big( \frac1{r^{n-1}} \phi ¨\Big) \right) dr\ .$$
The linear growth of $u_i$ away from the free boundary given by \eqref{ui=zeror>sigma} and \eqref{eq:6.1}, implies that $\Delta_r \underline{u}_i ¨$ develops a Dirac mass at 
$r= \sigma$ and 
$$ \int_{\sigma-\upsilon}^{\sigma+\upsilon} \underline{u}_i ¨ 
\frac{\partial}{\partial r} \left( r^{n-1} \frac{\partial}{\partial r} \Big( \frac1{r^{n-1}} \phi ¨\Big)\right) dr > 0,$$ for $\upsilon$ small enough.
Hence, $\Delta_r (\sum_{j\ne i} \bar u_j)$ is a positive measure in $(\sigma-\upsilon,\sigma+\upsilon)$  and therefore there exists $j\ne i$ such that $u_j$ cannot be identically equal to 
zero in the ball $\B_{1+\sigma} (\sigma e_0)$. 
Since $\sigma$ small is arbitrary, the result follows.
\end{proof}


\section{Classification of singular points and Lipschitz regularity in dimension 2 }\label{singpoints2dsec}

In this section we study singular points in dimension 2.  We will always assume \eqref{mainassumpts} with  $p=1$ in \eqref{H1}. 
 From the results of the previous sections we know that the solutions $u_1^\ep,\ldots, u_K^\ep$ of system \eqref{eq:stmt3}, through a subsequence, converge as $\ep\rightarrow 0$ to  functions $u_1,\ldots, u_K$ which are locally  Lipschitz continuous in $\Om$ and harmonic inside their support.
For $i=1,\ldots, K$,  let us denote  the interior of the support of $u_i$ by $S_i$ as in (\ref{Si}) and  the union of the interior of the supports of all the other functions by
\begin{equation}\label{Ci}
C_i:=\cup_{j\neq i} S_j.
\end{equation}
Since  the sets $S_i$ are disjoint we have $\partial C_i=\cup_{j\neq i} \partial S_j.$  
 From Theorem \ref{lem:6.1}  we know that $S_i$ and $C_i$ are at $\rho$-distance 1, therefore for any point $x\in\partial S_i$
there is a point $y\in \partial  C_i$  such that $\rho(x-y)=1$. We say that $x$ realizes at $y$ the distance from $  C_i$. 

\begin{defn} A point $x\in\partial S_i$  is a {\em singular} point if it realizes the distance from $C_i$   to at least two points in 
$\partial  C_i$. We say that $x\in\partial S_i$ is a {\em regular} point if it is not singular.
\end{defn}

 \begin{figure}
\begin{center}
\includegraphics[scale=0.4]{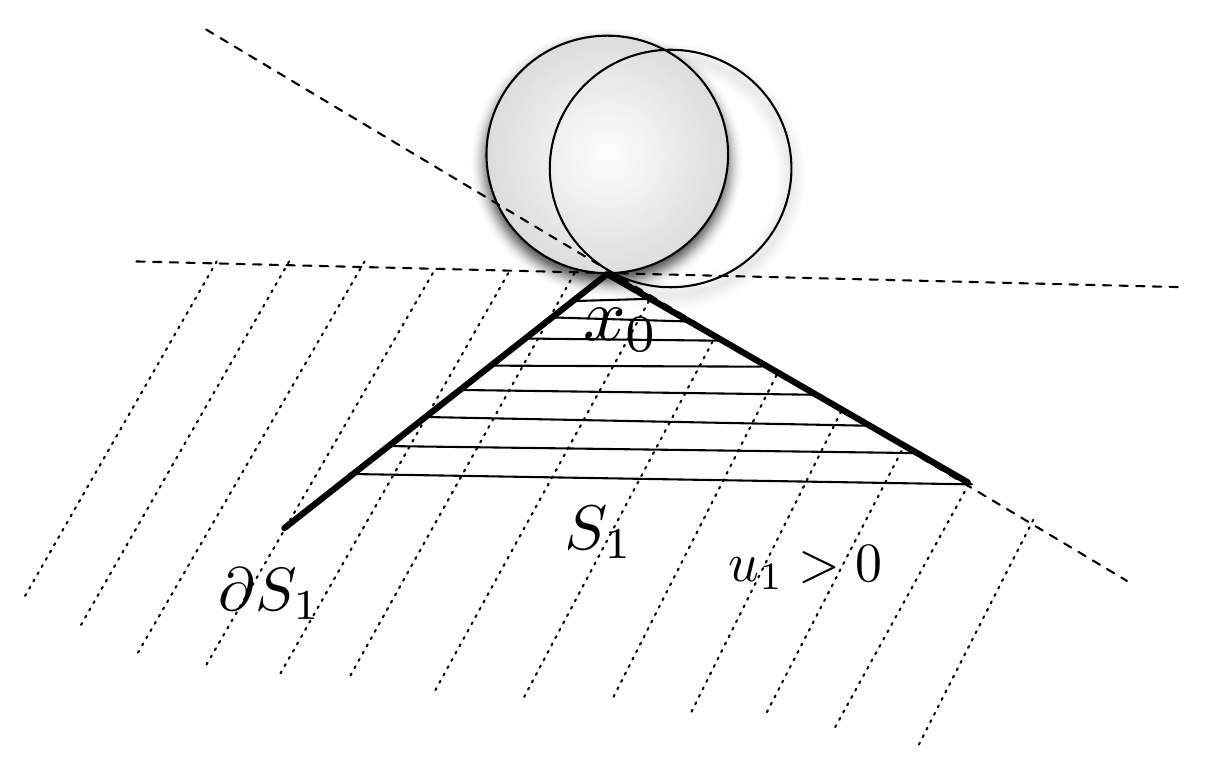}
\caption{Asymptotic cone at $x_0$}
\label{fig:cone}
\end{center}
\end{figure}

 Geometrically, we can describe regular and singular points as follows. Let $x\in \partial S_i$ be a singular point and $y_1,y_2\in\partial C_i$ points where $x$ realizes the distance from $C_i$. Then the balls $\B_1(y_1)$ and  $\B_1(y_2)$ are tangent to $\partial S_i$ at  $x$.  Consider the convex cone determined by the  two tangent lines to the two tangent $\rho$-balls  $\B_1(y_1)$ and  $\B_1(y_2),$ which does not intersect the two $\rho$-balls. The intersection of all cones generated by all  $\rho$-balls of radius 1,  tangent  at $x$ and with center in $C_i$ defines a convex asymptotic cone centered at $x$, see Figure \ref{fig:cone}. 
 If $x \in \partial S_i $ is a regular point, then there is only one point $y\in\partial  C_i$ where $x$ realizes the distance from $ C_i$. In this case, 
 the two tangent balls coincide and therefore, by definition the asymptotic   cone at $x\in\partial S_i$ is an  half-plane.
 We will show that at regular points $\partial S_i$ is the graph of a differentiable function.
 If $\theta\in[0,\pi]$ is the opening of the cone at $x$, we say that $S_i$ has an angle $\theta$ at $x$. Regular points correspond to
  $\theta=\pi$.  When $\theta=0$ the tangent cone is actually a semi-line and $S_i$ has a cusp at $x$. We will show, later on in this section, that, assuming additional hypothesis on the boundary data and the domain $\Om$,  the case $\theta=0$ never occurs and therefore   the free boundaries are Lipschitz curves of the plane. 

\begin{lem}\label{tangentballstangcone}
Let $\C=\{(x_1,x_2): x_2 \geq \alpha |x_1| \}$, $\alpha\geq 0$,  be the asymptotic cone of $S_i$ at $0\in\partial S_i$. Then there exist $y_1,\,y_2\in\partial C_i$ such that 
the balls $\B_1(y_1)$ and 
$\B_1(y_2)$ are tangent respectively to the lines $x_2=\pm  \alpha x_1$ at 0.
\end{lem}
\begin{proof}
Let $y_1, y_2 \in \B_1(0)$ be such that $x_2= \alpha x_1$ is a tangent line to $\B_1(y_1)$ at 0 and $x_2= -\alpha x_1$ is a tangent line to $\B_1(y_2)$ at 0. Suppose  by contradiction that  $y_1, y_2 \notin \p C_i$. Then,  any $y \in C_i$ such that $\rho(y-0)=1$ must lie in the smaller arc in $\p \B_1(0)$ between $y_1$ and $y_2$.  Moreover, there exists  $\delta>0$ such that all  $\rho$-balls $\B_1(y)$ have at most as  tangent lines  at 0 the lines $x_2=\pm  (\alpha-\delta) x_1$. Then the asymptotic cone at 0 must contain the cone $\{(x_1,x_2): x_2 \geq (\alpha -\delta) |x_1| \}$, which is not possible. 
%
\end{proof}
\begin{lem}\label{freeboudarygraphlem}
Assume that $S_i$ has an angle $\theta\in(0,\pi]$ at $x_0\in\partial S_i$. Then, there exists a neighborhood $U$ of $x_0$, a system of coordinates $(x_1,x_2)$ and a locally Lipschitz function $\psi:(-r,r)\to\real$, for some $r>0$, such that in the system of coordinates $(x_1,x_2)$, we have that  $x_0=(0,0)$ and
$$\partial S_i\cap U=\{(x_1,\psi(x_1)\,:\,x_1\in (-r,r)\}.$$
If in addition $\theta=\pi$, then $\varphi$ is differentiable at 0.
\end{lem}
\begin{proof}
Let $\C$ be the convex asymptotic cone of $S_i$ at $x_0$. Let us fix a system of coordinates $(x_1,x_2)$ such that the $x_2$ axis coincides with the axis of the cone and 
is oriented such that the cone is above the $x_1$ axis. Then we have that $x_0=(0,0)$ and $\C=\{(x_1,x_2): x_2 \geq \alpha |x_1| \}$ with $\alpha=\mathrm{tan(\frac{\pi-\theta}{2})}$. To prove that in this system of coordinates, $\partial S_i$ is the graph of a function  in a small neighborhood of $x_0$, it suffices to show that there exists a small $r>0$ such that, for any $|t|<r$, the vertical line $\{x_1=t\}$, intersects $\partial S_i\cap B_r(0)$ at only one point.
Suppose by contradiction that there exists a sequence $(t_n)$ such that $t_n\to0$ as $n\to+\infty$, and the  line $\{x_1=t_n\}$ intersects $\partial S_i\cap B_r(0)$ at
two distinct points $(t_n,a_n)$ and $(t_n,b_n)$ with $b_n>a_n$. Assume, without loss of generality, that $t_n>0$ for any $n$. 
By Lemma \ref{tangentballstangcone} there exist   $y_1$, $y_2 \in \p C_i$ that realize the distance from 0, and such that $\B_1(y_1)$ is tangent to the line $\{(x_1,x_2): x_2=\alpha x_1\}$  at 0 and  $\B_1(y_2)$ is tangent to $\{(x_1,x_2): x_2=-\alpha x_1\}$  also at 0. For instance, in the particular case of the  Euclidean norm, we would have $y_1= \left(\sqrt{\ds \frac{1}{1+\alpha^2}}, - \alpha \sqrt{\ds \frac{1}{1+\alpha^2}}\right)$ and $y_2= \left(-\sqrt{\ds \frac{1}{1+\alpha^2}},- \alpha \sqrt{\ds \frac{1}{1+\alpha^2}}\right) $. In general, what we can say is that the $x_2$ coordinate of $y_1$ and $ y_2$ is a negative value $-c$. We have that   $\B_1(y_1) \cap \B_1(y_2) \neq \emptyset$, since $\theta >0$. Moreover, 
  $S_i \cap(\B_1(y_1)\cup\B_1(y_2))=\emptyset$. Then, both points $(t_n,a_n)$ and $(t_n,b_n)$ must be above 
$\B_1(y_1)\cup\B_1(y_2)$ for $n$ large enough.  Next, let $y^a_n,y^b_n \in \partial C_i$ be points where $(t_n,a_n)$ and $(t_n,b_n)$, respectively, realize the distance from $C_i$. Then the $\rho$-balls 
$\B_1(y^a_n)$ and $\B_1(y^b_n)$ are exterior tangent balls to $\partial S_i$ at $(t_n,a_n)$ and $(t_n,b_n)$, respectively.  Recall  that 
the $\rho$-distance between the points $(t_n,a_n)$ and $(t_n,b_n)$ converges to 0 as $n\to+\infty$, and so, the point $y^a_n$  has to belong to the lower half $\rho$-ball  $\partial \B_1(t_n,a_n)\cap \{x_2<a_n\}$ 
for $n$ large enough.
Indeed, if not the tangent $\rho$-ball $\B_1(y^a_n)$ would contain $(t_n,b_n)$ for $n$ large enough. Similarly,  $y^b_n$  has to belong to the upper half $\rho$-ball
 $\partial \B_1(t_n,b_n)\cap\{x_2>b_n\}$ for $n$ large enough. This implies that the tangent $\rho$-ball $\B_1(y^b_n)$ will converge to a tangent ball to $S_i$ at 0,  $\B_1(y^b)$, with
  $y^b\in \{x_2\geq0\}$. On the other hand, by the definition of the asymptotic cones, all the centers of the  tangent balls  at 0 must belong to the set 
 $\partial \B_1(0)\cap\{x_2\leq -c\}$, where $-c<0$ is the $x_2$ coordinate of the points $y_1,\,y_2$ defined above.
 Therefore, we have  reached a contradiction. We infer that there exists $r>0$ such that  $\partial S_i$ is the graph of a function $\psi:(-r,r)\to\real$. 
 Since $\partial S_i$ is a closed set, $\psi$ is continuous.
 
 Let us prove that $\psi$ is Lipschitz continuous at 0. If $\C=\{x_2\geq \alpha|x_1|\}$ is the tangent cone of $S_i$ at $x_0$ in the system of coordinates $(x_1,x_2)$, 
 then for $r>0$ small enough we have
 $$\{x_2\geq 2\alpha|x_1|\}\subset S_i\cap B_r(0)\subset \left\{x_2\geq \frac{\alpha}{2}|x_1|\right\},$$
that is, for $|x_1|<r$, 
$$\frac{\alpha}{2}|x_1|\leq \psi(x_1)=\psi(x_1)-\psi(0)\leq  2\alpha|x_1|.$$
Therefore, $\psi$ is Lipschitz at 0.

Next, assume that $\theta=\pi$.  Then, we have that $y_1=y_2$, and  $x_0$ is a regular point. Therefore,  $\B_1(y_1)\subset \{x_2<0\}$   is the unique tangent ball to the graph of $\psi$ at $x_0=(0,0)$. Moreover, the tangent cone is the half plane  $\{x_2\geq 0\}$. Let us show that $\psi$ is differentiable at 0.
Assume by contradiction that there exists a sequence of positive points $(x_1^n) \in (-r,r)$ such that $x_1^n\to0$ as $n\to+\infty$ and 
\begin{equation}\label{psistregulpointderiv}\lim_{n\to+\infty}\frac{\psi(x_1^n)}{x_1^n}=\beta\neq0.\end{equation}
Since there exists a tangent ball from below to the graph of $\psi$ at 0 contained in $\{x_2<0\}$, we must have $\beta>0$.
For any point $(x_1^n,\psi(x_1^n)) \in \p S_i$ there exists a point $y_n\in\partial C_i$ such that $\B_1(y_n)$ is tangent to $S_i$ at $(x_1^n,\psi(x_1^n))$.
Let $y_2\in \p C_i$ be the limit of a converging subsequence of $(y_n)$. Then the $\rho$-ball $\B_1(y_2)$ is an exterior  tangent ball at $S_i$ at 0.
Equation \eqref{psistregulpointderiv} gives $\psi(x_1^n)\geq \frac{\beta}{2}x_1^n$ for $n$ large enough, i.e., the points $(x_1^n,\psi(x_1^n))$ of the free boundary are above the line $\{x_2=\beta/2|x_1|\}$. This implies that $y_1\neq y_2$, that is the limit $\rho$-ball $\B_1(y_2)$ must be different from $\B_1(y_1)$.  This is in contradiction with the fact that $x_0$ is a regular point. Therefore we must have
$$\lim_{x_1\to0^+}\frac{\psi(x_1)}{x_1}=0.$$
Similarly, one can prove that 
$$\lim_{x_1\to0^-}\frac{\psi(x_1)}{x_1}=0.$$ We conclude that $\psi$ is differentiable at 0 and $\psi'(0)=0$.

\end{proof}
\begin{lem}\label{freeboudarygraphlemC1}
Assume that there exists an open set $U$ of $\real^2$ such that any point of $U\cap \partial S_i$ is regular. Then $U\cap \partial S_i$ is a $C^1$-curve of the plane.
\end{lem}
\begin{proof}
Let $y_0\in\partial S_i\cap U$.  By Lemma \ref{freeboudarygraphlem}, 
there exists a differentiable function $\psi$ and a small $r>0$,  such that, 
 in the  system of coordinates $(x_1,x_2)$ centered at $y_0$ and with the $x_2$ axis in the direction of the inner normal of $\partial S_i$ at $y_0$, 
$\partial S_i\cap B_r(y_0)$ is the graph of  $\psi$. Moreover, in this system of coordinates,  $\psi(y_0)=\psi'(y_0)=0$.
By Corollary \ref{tangebtballcor}, there exists a tangent ball from below, with uniform radius,  at any point of the graph of $\psi$. 
This implies that for any $|x_1^0|<r$, there exists a $C^2$ function $\varphi_{x_1^0}$ tangent from below to  the graph of $\psi$ at $x_1^0$ and such that 
$|\varphi_{x_1^0}''|\leq C$, for some $C>0$ independent of $x_1^0$.
Therefore we have, for any $|x_1|<r$, 
\begin{equation*}\begin{split}\psi(x_1)&\geq \varphi_{x_1^0}(x_1)\geq\varphi_{x_1^0}(x_1^0)+\varphi_{x_1^0}'(x_1^0)(x_1-x_1^0)-C|x_1-x_1^0|^2
\\&=\psi(x_1^0)+\psi'(x_1^0)(x_1-x_1^0)-C|x_1-x_1^0|^2.\end{split}\end{equation*}
Now, let us show that $\psi$ is of class $C^1$. Fix a point $x_1^0$ and consider a sequence $(x_1^l)$ converging to $x_1^0$ as $l\to+\infty$. 
Let $p$ be the limit of a convergent subsequence of $(\psi'(x_1^l))$. Passing to the limit in $l$ the inequality,
$$\psi(x_1)\geq \psi(x_1^l)+\psi'(x_1^l)(x_1-x_1^l)-C|x_1-x_1^l|^2,$$ we get
$$\psi(x_1)\geq \psi(x_1^0)+p(x_1-x_1^0)-C|x_1-x_1^0|^2,$$ for any $|x_1|<r$. 
Since $\psi$ is differentiable at $x_1^0$, we must have $p=\psi'(x_1^0)$. 
\end{proof}

\begin{lem}\label{finiteconncompslem}
Assume  that the supports of the boundary data, $f_i$'s,  on $(\partial \Om)_1$ have a finite number of connected components. Then the sets $S_i$'s  have a finite number of connected components.
\end{lem}
\begin{proof}
Consider all the connected components of $S_i$, $S_i^j$, $i=1,\ldots,K$ and $j=1,2,3,\ldots$. Remark that for any $i$ and $j$
 $$\partial S_i^j\cap \{x\in (\partial\Om)_1\,:\,f_i(x)>0\}\ne \emptyset.$$
 Indeed, if not we would have $u_i= 0$ on  $\partial S_i^j$ and $\Delta u_i\ge 0$ in $S_i^j$. The maximum principle then would imply $u_i\equiv 0$ in $S_i^j$, which is not possible. Moreover, by continuity, 
 $\partial S_i^j$ must contain one connected component of the set  $\{x\in(\partial\Om)_1\,:\,f_i(x)>0\}$. 
 For this reasons we say that the components of $S_i$  reach the boundary of $\Om$. This implies that the connected components of $S_i$  are finite.  
 \end{proof}

\subsection{Properties of singular points}
We start by proving  three lemmas that will allow to estimate the growth of the solutions near the singular points. The first lemma claims that  positive functions which are superharmonic (subharmonic)  in  a cone and vanish on its boundary, have at least (at most) linear   growth  away from the  boundary of the cone  far from the vertex, with a slope that degenerate 
in a H\"older fashion approaching the vertex.  The  power just  depends on the opening of the cone. The second and third lemmas generalize these estimates to domains which are sets of points at $\rho$-distance greater than 1 from a closed bounded set. 
 Then we prove that  the set of singularities is a set of isolated points and we give a characterization. For the set  $ S_i$  which has finite perimeter, we denote by
 $\p^*S_i$  the reduced boundary, that  is the set of points whose blow-ups converge to half-planes and the essential boundary,  $\p_* S_i$, are all points except points of Lebesgue density zero and one. Moreover, $\mathcal{H}^{1} (\p_* S_i \setminus \p^*S_i)=0$. For more details see \cite{ambrosio_functions_2000, evans_measure_2015}. 

\begin{lem}\label{precious}
Let $v$ be a nonnegative Lipschitz function defined on $B_1\subset \real^n$, such that $\Delta v$ is locally a Radon measure on $B_1$ and such that  $v$ smooth on $S = \{v>0\}$. Assume that $S$ is a set of finite perimeter. Then, for every smooth $\phi $ with compact support contained in $B_1$
$$
\int_{B_1}  \Delta v \, \phi =\int_{S}  \Delta v \, \phi dx - \int_{\p^* S}  \frac{\p v}{\p \nu_{S}}\, \phi \,d\mathcal{H}^{n-1}
$$
where $\nu_{S}$ is the measure-theoretic outward unit normal and $\p^*S$ is the reduced boundary. \end{lem}
\begin{proof}
As a distribution and integrating by parts 
$$
 \int_{B_1}  \Delta v \, \phi=\int_{S}  v \Delta\phi dx = \int_{S} \mathrm{div} (v \nabla \phi)-\mathrm{div} (\nabla v \phi) +\Delta v \phi dx.$$
 Applying the generalized Gauss-Green theorem  (see \cite{chen_gauss_2009}, and also \cite{ambrosio_functions_2000, evans_measure_2015} for more details) we obtain the result.
\end{proof}
\begin{lem}\label{growthcond} Let $\theta_0\in(0,\pi]$. 
Let $\C$ be the cone defined in polar coordinates by $$\C=\{(\varrho,\theta)\,|\,\varrho\in[0,+\infty),\, 0\le\theta \le\theta_0\}.$$ 
Let $u_1$ and $u_2$  be respectively a superharmonic  and subharmonic 
positive function in the interior of  $\C \cap B_{2r_0}$, such that $ u_1\ge u_2=0 $ on $\partial \C \cap B_{2r_0}$. Then for any $r<r_0/3$ there exist  $R=R(\theta_0,r),$ and constants $c, \,C>0$   depending on respectively  $(\theta_0,u_1, r_0)$ and $(\theta_0,u_2, r_0)$, but independent of $r$, such that $\text{for any }x\in [r,3r]\times \left[0,R\right]$ we have
\begin{enumerate}
\item $
u_1(x)\geq  cr^{\alpha}d(x,\partial \C)
$
\item $u_2(x)\leq Cr^{\alpha}d(x,\partial \C)$
\end{enumerate} where $\alpha$ is given by
\begin{equation*}1+\alpha=\frac{\pi}{\theta_0}.\end{equation*}
\end{lem}

\begin{proof}
Let us introduce the function 
\begin{equation}\label{harm cone}v(\varrho,\theta):=\varrho^{1+\alpha}\sin((1+\alpha)\theta).\end{equation}
Notice that  $v$ is harmonic in the interior of $\C$, since it   is the imaginary part of the function $z^{1+\alpha}$, where $z=x+iy$, which is holomorphic in the set 
$ \mathbb C\setminus (- \infty, 0]$. Moreover $v$ is positive inside $\C$ and vanishes  on its boundary.  
By a barrier argument,  $u_1$ has at least linear growth away   from the boundary of $\C$, meaning for $\rho\in[r_0/2,3r_0/2]$ (far from the vertex and from $\partial B_{2r_0}$) $$u_1(x)\geq k d(x,\partial \C),$$  
for $k=c_0\min_{x\in\C\atop d(x,\partial \C)\ge s_0} u_1,$ and for $x\in \{x\in \C: r_0/2<|x|<3r_0/2, \: d(x,\partial\C)\leq s_0\}$
where $c_0$ and $s_0$ depend on $r_0$ and $\theta_0$. 
Therefore, we can find a constant $c>0$ depending on  $u_1$,  $r_0$ and $\theta_0$, such that 
$$ u_1\ge cv \quad\text{on }\C\cap \partial B_{r_0}.$$
Since in addition $u_1\geq cv=0$  on $\partial C\cap B_{r_0}$,  the comparison principle implies
\begin{equation}
\label{uconenehavior}
 u_1\ge cv\quad\text{in }\C\cap B_{r_0}.\end{equation} 
Since $v$ is increasing  in the radial direction and if we are near $\partial \C$ it is  also increasing in the $\theta$ direction,
%
for $r\le |x| \le 3r$, with $r$ such that $r\le \frac{r_0}{3}$ and $d(x, \C)\le R$
with $ R=r\min\left\{1,\tan\left(\frac{\theta_0}{2}\right)\right\}$,
\begin{equation*} u_1(x)\geq cv(x)\geq Cr^{\alpha}d(x,\partial \C) 
\end{equation*} 
and (a) follows.

To prove (b) similarly, we have 
 \begin{equation}\label{uconenehavioru2}u_2\le C v\quad\text{in }\C\cap B_{r_0},\end{equation} 
where  $C$ depends on $(\theta_0,u_2,r_0)$ but it is independent of $r$.
In particular, 
for $r\le |x| \le 3r$ and $d(x, \C)\le \frac{R}{2}$
$$u_2(x) \leq Cv(x)\leq \tilde{C} r^\alpha d(x,\partial \C).$$


\end{proof}

\begin{lem}
\label{general cone}
Let $\Om$ be an open set, $C$ be a closed subset of $\Om$ and $S=\{x\in\Om\,|\,d_\rho(x,C)\geq 1\}$. 
Let $S_1$ be a connected component of $S$.  Assume that $\partial S_1=\Gamma_1\cup\Gamma_2$, with $\Gamma_1\cap\Gamma_2=\{0\}$  and $S_1$ has an angle 
$\theta_0 \in (0,\pi]$ at $0\in\partial S_1$.   Let $u_1$ be  a superharmonic  
positive function in  $S_1 \cap B_{2r_0}(0)$,  with $u_1=0 $ on $\p S_1  \cap B_{2r_0}(0).$ 
 Then, there exists a sequence $(x_h)\subset \Gamma_1$ of regular points convergent to zero,    $x_h\to0\:\text{as }h\to0,$ and 
 there exist balls $B_{R_h}(z_h)\subset S_1$ tangent to $\partial S_1$ at $x_h$, where $R_h\geq c|x_h|$, such that 
 $$u_1(x) \ge c R_h^{\alpha_\delta}d(x,\partial B_{R_h}(z_h)) \qquad \mbox{for any}\quad  x\in B_{R_h}(z_h) \setminus B_{\frac{R_h}{4}}(z_h),$$
 where $\alpha_\delta$ is  given by
$$ 1+\alpha_\delta=\frac{\pi}{\theta_0-\delta}.$$
\end{lem}
\begin{proof}
Since $\theta_0 \in(0,\pi]$  for any  $0<\delta<\theta_0$, there exist
$r_\delta>0$ and a cone $\C^1_\delta$ centered at 0  with opening $\theta_0-\delta$ such that
\begin{equation*}\C^1_\delta\cap B_{ r_\delta}(0)\subset  S_1\cap B_{ r_\delta}(0).\end{equation*}
Take a sequence of points $t_h\in \partial\C^1_\delta\cap B_{ r_\delta}(0)$ converging to 0 as $h\to0$. Let
$$r_h:=d(t_h,0)\qquad  \mbox{and} \qquad
R_h:=r_h\min\left\{1,\tan\left(\frac{\theta_0-\delta}{2}\right)\right\}.$$
Then, for $h$ small enough, 
there exist balls $B_{R_h}(s_h)\subset \C^1_\delta\cap B_{ r_\delta}(0) $ such that $t_h\in\partial B_{R_h}(s_h)$. Consider a system of polar coordinates $(\varrho,\theta)$ centered at 0.
Moving the balls $B_{R_h}(s_h)$ along the $\theta$ direction until it touches $\Gamma_1$, we can find a sequence of regular points $x_h$ in that region, such that 
$d(x_h,0)\leq cr_h$ and balls 
$B_{R_h}(z_h)\subset  S_1\cap B_{ r_\delta}(0)$ such that $x_h\in\partial B_{R_h}(z_h)$.  Observe that the center of the ball, $z_h$, remains inside the cone $\C_\delta^1$, that is, for $h$ and $\delta $ small enough, we have that  $z_h\in\C_\delta^1$ and $d(z_h, \partial\C_\delta^1)\geq \frac{R_h}{2}$.  
Let us introduce the barrier function
$$\phi(x):=\frac{m}{\log 4}\log\left(\frac{R_h}{|x-z_h|}\right),\quad  \mbox{where}\quad
m=\inf_{ \partial B_\frac{R_h}{4}(z_h)}u_1. $$
Then $\phi$ satisfies
\[
\begin{cases}
  \Delta \phi=0&\text{in } B_{R_h}(z_h) \setminus B_{\frac{R_h}{4}}(z_h)  \\
  \noalign{\vskip6pt}
  \phi=0&   \text{on }  \partial B_{R_h}(z_h)\\
  \noalign{\vskip6pt}
  \phi=m&    \text{on }  \partial B_\frac{R_h}{4}(z_h).
\end{cases}
\]
Since $ u_1\geq \phi$ on $ \partial B_{R_h}(z_h)\cup \partial B_\frac{R_h}{4}(z_h)$ the comparison principle then implies 
\begin{equation*}u_1\geq \phi\qquad \text{in } B_{R_h}(z_h) \setminus B_{\frac{R_h}{4}}(z_h).\end{equation*}
If $\nu_1$ is the inner normal vector of $B_{R_h}(z_h)$,  then for  $x \in \partial B_{R_h}(z_h),$  
$$\frac{\partial \phi}{\partial\nu_1}(x)= \frac{m}{R_h\log 4 },$$
and the convexity of $\phi$ in the radial direction gives, for any $x\in B_{R_h}(z_h) \setminus B_{\frac{R_h}{4}}(z_h)$ 
 \begin{equation*}u_1(x)\geq \frac{m}{R_h\log 4} d(x,\partial B_{R_h}(z_h)).\end{equation*}
 Let us estimate $m$. Since $d(z_h,\partial C_\delta^1)\geq\frac{R_h}{2}$, we have that $d(x,\partial C_\delta^1)\geq\frac{R_h}{4}$ for any $x\in B_{\frac{R_h}{4}}(z_h)$. As in Lemma~\ref{growthcond}, consider the harmonic function $v(x) $, introduced in \eqref{harm cone},  defined on the cone $\C^1_\delta$ ($\alpha=\alpha_\delta$) and the comparison principle result stated in \eqref{uconenehavior}. Then
  \begin{equation*}
 m\geq c \min_{ \partial B_{\frac{R_h}{4}}(z_h)} v\geq \min\left\{v\left(r_h-\frac{R_h}{4},\frac{\theta_0-\delta}{8}\right), v\left(\frac{3r_h}{4},\frac{\pi}{16} \right)\right\}= c_1 \left(\frac{3r_h}{4}\right)^{\alpha_\delta +1} \end{equation*} where $c_1=c_1(u_1, r_\delta,\theta_0-\delta)$. 
Then, since $\frac{r_h}{R_h}\geq 1$ we conclude that for any   $x\in B_{R_h}(z_h) \setminus B_{\frac{R_h}{4}}(z_h)$,
 $$u_1(x) \ge c R_h^{\alpha_\delta}d(x,\partial B_{R_h}(z_h)).$$
 This concludes the proof of the lemma.
\end{proof}

\begin{lem}
\label{general conebis}
Let $\Om$ be an open set, $C$ be a closed subset of $\Om$ and $S=\{x\in\Om\,|\,d_\rho(x,C)\geq 1\}$.  Let $S_1$ be a connected component of $S$. 
Assume that $S_1$ has an angle 
$\theta_0 \in [0,\pi]$ at $0\in\partial S_1$.   Let $u_2$ be  a subharmonic  
positive function in  $S_1 \cap B_{2r_0}(0)$,  with $u_2=0 $ on $\p S_1  \cap B_{2r_0}(0).$ 

  Then, for any  $0<\delta<\theta_0$, there exists $r_\delta>0$ such that for any $r<r_\delta/5$ there exist  $R=R(\theta_0,r),$ and  a constant $C>0$   depending on  $(\theta_0+\delta,u_2, r_\delta)$, but independent of $r$, such that 
\begin{equation}
u_2(x)\leq Cr^{\beta_\delta}d(x,\partial S_1)\quad  \mbox{for any} \quad x\in \left(B_{3r}(0)\setminus B_r(0)\right) \cap  \left\{ x\in S_1: d(x , \p S_1)\leq \frac{R}{4}  \right\}\\
\end{equation}
where  $\beta_\delta$ is given by
\begin{equation*} \begin{split}
1+\beta_\delta=\frac{\pi}{\theta_0+\delta}.
\end{split}
\end{equation*}
\end{lem}
\begin{proof}
 For any   $\delta>0$, there exist
$r_\delta>0$,   a cone  $\C^2_\delta$ centered at  $0$   and with opening $\theta_0+\delta$ such that
\begin{equation*} S_1\cap B_{ r_\delta}(0)\subset \C^2_\delta\cap  B_{ r_\delta}(0).\end{equation*} 
Take any $r<r_\delta$ and let $y \in \p S\cap (B_{3r}(0)\setminus B_r(0))$ and $r_y:=d(y,0)\in(r,3r)$.
\noindent  Since $S$ is at $\rho$-distance 1 from $C$, for any point of the boundary of $S_1$ there exists an exterior tangent $\rho$-ball of radius 1.
This implies that  for $r$ small enough,  there exists  $w_y$ such that  the Euclidian ball $B_{R_y}(w_y)$ is contained in the complement of $S$ and 
$y\in \partial B_{R_y}(w_y)$, where $R_y$ is defined by
$$
R_y=r_y\min\left\{1,\tan\left(\frac{\theta_0+\delta}{2}\right)\right\}.
$$ 
 Let us take now  as barrier the function 
 $$\psi(x):=\frac{M}{\log \frac{3}{2}}\log\left(\frac{|w_y-x|}{R_y}\right) \quad \mbox{with}\quad M=\sup _{ \partial B_{\frac{3}{2}R_y}(w_y)}u_2. $$ 
Then $\psi$ satisfies
 \[
\begin{cases}
  \Delta \psi=0&\text{in } B_{\frac{3R_y}{2}}(w_y)\setminus B_{R_y}(w_y)  \\
  \noalign{\vskip6pt}
  \psi=M &   \text{on }  \partial B_{\frac{3R_y}{2}}(w_y)\\
  \noalign{\vskip6pt}
  \psi=0&    \text{on }  \partial B_{R_y} (w_y).
\end{cases}
\]
Using the comparison principle with $u_2$, the concavity of $\psi$ in the radial direction gives that for any $x \in B_{\frac{3R_y}{2}}(w_y)\setminus B_{R_y} (w_y)$
  $$
  u_2 \leq \frac{M}{R_y\log (\frac{3}{2})}d(x, \p B_{R_y}(w_y)).
  $$
Let us estimate $M$.
Consider again a system of polar coordinates $(\varrho,\theta)$ centered at 0 and the harmonic function $v(x) $, introduced in \eqref{harm cone},   defined on the
 cone $\C^2_\delta$ ($\alpha=\beta_\delta$).       By definition of $v$, $R_y$, and taking into account  \eqref{uconenehavioru2}, for $\delta$, $r$ small enough, 
  \begin{equation*}\label{max v}M\leq C\max_{ \partial B_{\frac{3R_y}{2}}(w_y)} v\leq C v\left(4r_y, \frac{\theta_0+\delta}{2} \right)=C_1(4r_y)^{\beta_\delta+1} =\tilde{C_1}r_y^{\beta_\delta} \frac{R_y}{\min\{1,\tan(\frac{\theta_0 +\delta}{2})\}} \end{equation*}
we see that for any $x\in B_{\frac{3R_y}{2}}(w_y)\setminus B_{R_y}(w_y)$ and belonging to the segment $y+s(y-w_y)$, $s\in\left(0,\frac{1}{2}\right)$, we have 
  \begin{equation}
  \label{bound uj}u_2(x)\leq CM d(x,\partial B_{R_y}(w_y))= CM d(x,\partial S_1)\leq Cr_y^{\beta_\delta}d(x,\partial S_1).\end{equation}
   
 Letting the tangent ball moving along  $\partial S_1\cap (B_{3r_y}(0) \setminus B_{r_y}(0))$, we get (b).

\end{proof}

\begin{lem}\label{singisolated}
 Assume    \eqref{mainassumpts} with  $n=2$ and $p=1$ in \eqref{H1}. Assume  in addition that the supports on $\partial \Om$ of the boundary data $f_i$'s  have a finite number of connected components. Let  $(u_1^\ep,\ldots, u_K^\ep)$ be a viscosity solution of the problem \eqref{eq:stmt3} and $(u_1,\ldots, u_K)$ the limit as $\ep\to0$  of a convergent subsequence. The set of singular points of $\Om$ is a set of isolated points.
\end{lem}

\proof

  Suppose by contradiction that there exists a sequence of distinct singular points $(y_k)_{k \in \mathbf{N}} $ such that $y_k \in \partial S_j$ and 
  $y_k\to y\in\Om$ as $k\to+\infty$. 
  Since by Lemma \ref{finiteconncompslem} the connected components of the sets $S_i$, $i=1,\ldots,K$ are finite, we may assume without loss of generality that the points 
  $y_k$ belong to the same connected component  of $S_j$, which we denote by $S_j^1$. If there exists $\theta_{max}<\pi$ such that  $S_j^1$  has an  angle  smaller than $\theta_{max}$ at $y_k$ for any $k$, then, there exists 
  $\overline{k}$ such that starting from $y_{\overline{k}}$, after a finite number of singular points $S_j^1$ would be an isle and not reach the boundary. 
  Therefore we would have  $u_j=0$ on $\partial S_j^1$ and $\Delta u_j=0$  in $S_j^1$, and the  maximum principle would imply $u_j\equiv 0$ in  $S_j^1$, which is a contradiction.  We infer that there exists a $k\in \mathbf{N}$ such that the angle at $y_k$ is close to $\pi$. 
   In particular, if  $x_1^k$ and $x_2^k$ are  points in $C_j$ that realize the $\rho$-distance  from $S_j$ at $y_k,$  then $\rho$-distance between $x_1^k$ and $x_2^k$ is less than one.
  
  Next, suppose that $x_i^k$ and $x_2^k$ belong to the same connected component of $S_i$, for some $i\neq j$.  Then, by Theorem \ref{lem:6.1} we know that $\partial S_i\cap \overline{\B_1(y_k)}$ has to contain the arc of the unit $\rho$-ball between $x_1^k$ and $x_2^k$. If not, there would be points in the curve connecting $x_1^k$ and $x_2^k$ which do not realize the distance from $C_i$.  Any point inside this arc is a regular point at $\rho$-distance 1 from $y_k$. Consider any of them, for instance 
  the middle point of the arc, denoted by  $x_k$. We want to compare the mass of the Laplacian of $u_i$ at $x_k$ with the mass of the Laplacian at $u_j$ at $y_k$, across the free boundaries. Let us first assume $H$ defined as in \eqref{H1}. For  $\sigma<\frac{1}{8} d_\rho(x_1^k,x_2^k)$ let us define
\begin{equation*}
D_\sigma(x_k):=\{x\in \B_{\sigma}(x_k)\,|\,d(x,\partial \C_i)\leq \sigma^2\},\end{equation*}
where $\C_i$ is the asymptotic  cone to $S_i^1$ at $x_k.$ Note that, since $x_k$ is a regular point,   $\partial \C_i$ is  the tangent line to $\partial S_i^1$ at $x_k$, and so $\C_i$ has opening  $\pi$. 
Let  $(D_\sigma(x_k))_1$  be the set of points at $\rho$-distance less than 1 from $D_\sigma(x_k)$,   then we have  that
\begin{equation}\label{mass laplacian D }
\int_{D_\sigma(x_k)}\Delta u_i\leq \sum_{j\neq i} \int_{(D_\sigma(x_k))_1}\Delta u_j
\end{equation}
as in \eqref{boundlaplacian} with $(D_\sigma(x_k))_1$ in place of $\B_{1+S}(x_0)$.
By the Hopf Lemma, we obtain
  \begin{equation}
  \label{arc side bound}
 \int_{D_\sigma(x_k)}\Delta u_i= \int_{\partial S_i\cap D_\sigma(x_k) }\frac{\partial u_i}{\partial \nu_i} \:d\mathcal{H}\geq c\: \mathcal{H}(\partial S_i\cap D_\sigma(x_k) ) =\tilde{C}\sigma \end{equation}
 where $\nu_i $ is the inner normal vector.

Now we estimate $ \int_{(D_\sigma(x_k))_1}\Delta u_j $. 
From Corollary \ref{hausmescor} we know that $S_j$ is a set of finite perimeter. Therefore  by Lemma \ref{precious} and Lemma \ref{general conebis}  we obtain the following estimate
   \begin{equation}
  \label{funny side}
 \int_{(D_\sigma(x_k))_1}\Delta u_j= \int_{\partial^* S^1_j \cap (D_\sigma(x_k))_1 }\frac{\partial u_j}{\partial \nu_{S^1_j}} \:d\mathcal{H}\leq C \sigma^{\beta_\delta}\: \mathcal{H}(\partial^* S^1_j\cap (D_\sigma(x_k) )_1) 
  \end{equation}
  where $\nu_{S_j}$ is the measure-theoretic inward unit normal to $S^1_j$ and $\beta_\delta>0$.
Since, for some constant $c$
  $$
  \partial S^1_j\cap (D_\sigma(x_k) )_1 \subset \partial S^1_j \cap \B_{ c\sigma}(y_k),
   $$
  by \eqref{eq:stmt2}, there exists $ \tilde{c}_2$, that for simplicity we will still name $c$, such that $\partial S^1_j\cap (D_\sigma(x_k) )_1 \subset \partial S^1_j \cap B_{ c\sigma}(y_k)$. Then
  \begin{equation}\label{final touch} \mathcal{H}(\partial^* S^1_j\cap (D_\sigma(x_k) )_1) \leq  \text{Per}(\partial S^1_j \cap B_{ c \sigma}(y_k) ) .\end{equation}
  To estimate $ \text{Per} (\partial S^1_j \cap B_{c\sigma}(y_k))$, 
consider   \eqref{lap distance} in the distributional sense. Then, take    a smooth function $0\leq \phi\leq 1$ with compact support  contained in
 $B_{c\sigma}(y_k) \cap \{x: 0< d(x,S_i) <1 \}$ and such that $\phi\equiv 1$ on the set $B_{c\sigma}(y_k) \cap \{x: 1-\delta< d(x,S_i) <1-\epsilon \}$, for $0<\ep<\delta$ and $\delta$ as introduced in the definition of  $\eta $  in the proof of Lemma \ref{finiteperlem}. Then, for $d_{S_i}(\cdot)=d_\rho(\cdot, S_i)$ we have that 

\begin{equation*}
\begin{split}
0=\int_{B_{c\sigma}(y_k) \cap \{x: 0< d_{S_i} <1\}} \text{div}(\eta(d_{S_i})Dd_{S_i}) \phi 
&= \int_{B_{c\sigma}(y_k) \cap \{x: 0< d_{S_i} <1\}} \eta'(d_{S_i})|Dd_{S_i}|^2\phi dx\\
&+\int_{B_{c\sigma}(y_k) \cap \{x: 0< d_{S_i} <1\}} \eta(d_{S_i})\Delta d_{S_i}   \phi \\
&\leq \int_{B_{c\sigma}(y_k) \cap \{x: 0< d_{S_i} <1\}} \eta'(d_{S_i})|Dd_{S_i}|^2\phi dx+ C \sigma.\\
  \end{split}
  \end{equation*}
Proceeding as in Lemma \ref{finiteperlem} we obtain that 
  \begin{equation}\label{funnybound}
  \text{Per} (\partial S^1_j\cap B_{c\sigma}(y_k))\leq  C\sigma.
  \end{equation}
 Putting together \eqref{mass laplacian D }, \eqref{arc side bound}, \eqref{funny side},   \eqref{final touch} and \eqref{funnybound} we obtain
$$
 C \sigma^{1+\beta_\delta}\geq\tilde{C}\sigma
 $$ and  we get a contradiction for $\sigma$ small enough.  
  %
%
 In the case \eqref{H2} the proof follows the same steps using \eqref{eq: sup case estimate}.

 Therefore we must have that $x_1^k$ and $x_2^k$ belong to different components  of $C_j$ for any $k\geq\overline{k}$. In particular, since the  distance between them is less than one, they must belong to two different components of the same population. 
 Suppose that $x_1^{\overline{k}}\in S_i^1$ and $x_2^{\overline{k}}\in S_i^2$,  for $i\neq j$.  Consider the consecutive two points $x_1^{\overline{k}+1}$
 and $x_2^{\overline{k}+1}$ which realize the distance at  $y_{\overline{k}+1}$, and again belong to two different components of $C_j$.  Since $S_j^1$ (to which $y_{\overline{k}}$ belongs)  and $S_i^2$  reach the boundary of $\Om$, the point $x_2^{\overline{k}+1}$
 must belong to a connected component  different from $S_i^1$. Iterating the procedure, we construct a sequence of distinct points  belonging  to connected components, each 
  different from the others. This is in contradiction with Lemma \ref{finiteconncompslem}.
  We conclude that singular points are isolated.

\endproof


\begin{thm}\label{equalanglethm} Assume    \eqref{mainassumpts} with  $n=2$ and $p=1$ in \eqref{H1}. Let  $(u_1^\ep,\ldots, u_K^\ep)$ be a  viscosity solution of the problem \eqref{eq:stmt3} and $(u_1,\ldots, u_K)$ the limit as $\ep\to0$  of a convergent subsequence. For $i\neq j$, let $x_0\in\partial S_i\cap\Om$ and $y_0\in \partial S_j\cap\Om$ be points such that $S_i$ has an angle $\theta_i\in[0,\pi]$ at  $x_0$,  
$S_j$ has an angle $\theta_j\in[0,\pi]$ at  $y_0$ and $\rho(x_0-y_0)=1$. Then we have
\begin{equation}\label{angleeq}\theta_i=\theta_j.\end{equation}
If $x_0\in\partial S_i\cap\partial \Om$ and $y_0\in \partial S_j\cap\Om$, then
\begin{equation}\label{angleatboundineq}\theta_i\leq \theta_j.\end{equation}
\end{thm}
\begin{proof}
By  Lemma \ref{finiteconncompslem}, the connected components of the sets $S_i$'s are finite. 
Assume $x_0\in\overline{\Om}$ and $y_0\in\Om$. 
Without loss of generality we can assume that $x_0=0$.  It suffices to show the theorem for $y_0$  belonging to a region that is side by side with $S_i$, in the sense that 0 is the limit as $h\to 0$ of interior regular points $x_h\in\partial S_i\cap\Om$ with the property that  $x_h$  realizes the distance from $S_j$ at $y_h\in\partial S_j\cap\Om$  interior  points, with $y_h\to y_0$ as $h\to 0$.  
Let $\C_i$  be the asymptotic cone at $0$. 
Let us first suppose for simplicity that $\partial S_i$ and $\partial S_j$ are  locally a cone around 0 and $y_0$ respectively. In particular, $\theta_i,\,\theta_j>0$. We will explain later on how to handle the general case. 
\medskip

\noindent{\em  Proof of Theorem \ref{equalanglethm} when $\partial S_i$ and $\partial S_j$ are  locally  cones .}
We assume that  there exists $r_0>0$ such that  
$\partial S_i\cap B_{2r_0}= \C_i \cap B_{2r_0}$, where $B_{2r_0}$ is the Euclidian ball centered at 0 of radius $2r_0$. When  $x_0 \in \partial \Omega$ we are just interested in the side of the cone $ \C_i$ contained in $\Omega.$

 If $(\varrho,\theta)$ is a system of 
polar coordinates in the plane centered at zero, we may assume that $\C_i$ is the cone given by
 \begin{equation*}\C_i=\{(\varrho,\theta)\,|\,\varrho\in[0,+\infty),\, 0\leq \theta \leq \theta_i\}.\end{equation*}
Let us first consider the case \eqref{H2}.
 Let us assume that $x_h=(2r_h,0)$, with $r_h>0$.
We know that $r_h\to 0$ as $h\to 0$, then we can  fix $h$ so small that $r_h<r_0/3$.  By Lemma \ref{growthcond} applied to  $u_1=u_i,$ we have 
 \begin{equation}\label{uigreaterthanlinearal}u_i(x)\geq  cr_h^{\alpha}d(x,\partial S_i)\quad\text{for any }x\in [r_h,3r_h]\times \left[0,R_h\right] ,\end{equation}
where 
\begin{equation}\label{alphadelta}1+\alpha=\frac{\pi}{\theta_i}\geq 1.\end{equation}

Now, we repeat an argument similar to the one in the proof of Theorem \ref{lem:6.1}. We look at $\inf u_i$ in small circles of radius $r$ that go across the free boundary of $u_i$ and we look at $\sup u_j$ in circles of radius $r+1$ across the free boundary of $u_j$, then  we compare the mass of the correspondent  Laplacians. Precisely, there exists a small $\sigma>0$ and $e\in S_i$ such that $\B_\sigma(e)\subset  [r_h,3r_h]\times \left[0,R_h\right]$ and $x_h\in\partial \B_\sigma(e)$. In particular,  in $\B_\sigma(e)$ the function $u_i$ satisfies \eqref{uigreaterthanlinearal}.
For $\upsilon<\sigma$ and $r\in[\sigma-\upsilon,\sigma+\upsilon]$,  we define 
\begin{equation}
\label{bounds}
\underline{u}_i ¨ := \inf_{\partial \B_r (e)} u_i\, \quad\text{and} \quad 
\bar u_j¨ : = \sup_{\partial \B_{1+r}(e) } u_j\ .
\end{equation}
In what follows we denote by $C$ and $c$ several constants independent of $h$.
For $r\in[\sigma-\upsilon,\sigma]$, by \eqref{uigreaterthanlinearal}  we have 
\begin{equation*}\begin{split} \underline{u}_i&\ge  \inf_{\partial \B_r(e)} cr_h^{\alpha} d(x,\partial S_i)
\geq  \inf_{\partial \B_r(e)}Cr_h^{\alpha} d_\rho(x,\partial S_i)\ge Cr_h^{\alpha}(\sigma-r).
\end{split}\end{equation*}
For $r\in[\sigma,\sigma+\upsilon]$,  the ball $\B_r(e)$ goes across $\partial S_i$, therefore we have $\underline{u}_i =0$.
Hence
\begin{equation}\label{uibarbehavior}\begin{split}& \underline{u}_i(r) \ge Cr_h^{\alpha}(\sigma-r)\quad\text{for  }r\in[\sigma-\upsilon,\sigma]\\&
  \underline{u}_i(r)=0\quad\text{for }r\in[\sigma,\sigma+\upsilon].
\end{split}\end{equation}

\noindent Next, let us study the behavior of $\bar u_j$. First of all, let us show that 
\begin{equation}\label{distanceeyh}d_\rho(e,\partial S_j)=\rho(e-y_h)=1+\sigma.\end{equation} Since $d_\rho(e,\partial S_i)=\sigma$ and $d_\rho(S_i,S_j)\geq 1$, it is easy to see  that $d_\rho(x,\partial S_j)\geq 1+\sigma$.
The function  $\rho$ is also called a Minkowski norm and from  known results about  Minkowski norms,  if we denote by $T$ the Legendre transform $ T:\real^n\to\real^n$ defined by $T(y)=\rho(y)D\rho(y)$, then 
$T$ is a bijection with inverse $T^{-1}(\xi)=\rho^*(\xi)D\rho^*(\xi)$, where $\rho^*$ is the dual norm defined by $\rho^*(\xi):=\sup\{y\cdot\xi\,|\,y\in\B_1\}$.
Now, the ball $\B_1(y_h)$ is tangent to $\partial S_i$ at $x_h$ and therefore is also tangent to $\B_\sigma(e)$ at  $x_h$. This implies that $D\rho(e-x_h)=-D\rho(x_h-e)=D\rho(x_h-y_h)$.
 Consequently we have 
\begin{equation*}\begin{split} e-x_h&=T^{-1}(T( e-x_h))=T^{-1}(\sigma D\rho(e-x_h))=T^{-1}(\sigma D\rho(x_h-y_h))\\&
=\sigma T^{-1}(T(x_h-y_h)=\sigma(x_h-y_h).
\end{split}\end{equation*}
We infer that 
\begin{equation}\label{eyhrelation}e=x_h+\sigma(x_h-y_h)\end{equation} and 
$$\rho(e-y_h)=(1+\sigma)\rho(x_h-y_h)=1+\sigma,$$
which proves \eqref{distanceeyh}. 
As a consequence 
 $\partial\B_{1+r}(e)\cap S_j=\emptyset$ for $r\in[\sigma-\upsilon,\sigma)$, while  if  $r\in(\sigma,\sigma+\upsilon]$ then 
$\partial\B_{1+r}(e)\cap S_j\neq\emptyset$ and   $\partial \B_{1+r}(e)$ enters inside $S_j$ at $\rho$-distance at most $r-\sigma$ from the boundary of  $S_j$.
In particular we have 
\begin{equation}\label{ujoutsideSj}\bar u_j ¨ =0\quad\text{for } r\in[\sigma-\upsilon,\sigma].\end{equation}

\noindent Next, if $\theta_j$ is the angle of $S_j$ at $y_0$, let  $\beta$ be defined by
\begin{equation}\label{betadelta}1+\beta=\frac{\pi}{\theta_j}\geq1.\end{equation}
Remark that $y_h$ is at $\rho$-distance $2r_h$ from $y_0$. Again by Lemma \ref{growthcond} applied to  $u_2=u_j,$ (after a rotation and a translation), we have
the following estimate 
$$u_j(x)\le C r_h^{\beta}d(x,\partial S_j)\le C r_h^{\beta}d_\rho(x,\partial S_j),$$  in a neighborhood of $y_h$. As a consequence, recalling in addition that  the ball  $\B_{1+r}(e)$ enters in $S_j$ at $\rho$-distance $r-\sigma$ from the boundary,
for $r\in[\sigma,\sigma+\upsilon]$ we get
 $$\bar u_j= \sup_{\partial \B_{1+r}(e) } u_j\leq Cr_h^{\beta}(r-\sigma).$$  The last estimate and \eqref{ujoutsideSj} imply

\begin{equation}\label{ujbarbehavior}\bar u_j(r)\le Cr_h^{\beta}(r-\sigma)^+,\quad\text{for }r\in[\sigma-\upsilon,\sigma+\upsilon].
\end{equation}

\noindent Now, we want to compare the mass of the Laplacians of $\underline{u}_i$ and $\bar u_j$. Define as in (\ref{bounds})
$$\underline{u}_i^\ep ¨ : = \inf_{\partial \B_r(e)} u_i^\ep, \quad\quad \bar u_k^\ep ¨ : = \sup_{\partial \B_{1+r}(e)} u_k^\ep,\,k\neq i.$$
For $\sigma$ and $\upsilon$ small enough, the ball $\B_{r}(e)$ is contained in $\Om$ for any $r\leq\sigma+\upsilon$, and thus 
$$\Delta u_i^\ep=\frac1{\ep^2} u_i^\ep \sum_{k\ne i} H(u_k^\ep)\quad\text{in }\B_{r+\sigma}(e).$$
On the other hand, since $x_h$ is an interior  regular point that realizes its distance from $S_j$ at an interior point, $y_h$, its distance from the support of the boundary data $f_k$ is greater than 1, for any $k\neq i$. We infer that, for $\sigma$ and $\upsilon$ small enough and $r\leq\sigma+\upsilon$,
$$\Delta u_k^\ep\geq \frac1{\ep^2} u_k^\ep \sum_{l\ne k} H(u_l^\ep)\quad\text{in }\B_{1+r}(e).$$
Hence, arguing as in the proof of Theorem \ref{lem:6.1}, we see that 
\begin{equation}\label{laplacinec1reg}\Delta_r \underline{u}_i^\ep \leq \sum_{k\neq i}\Delta_r\bar u_k^\ep\quad\text{in }(\sigma-\upsilon,\sigma+\upsilon),\end{equation}
where $\Delta_r u=\frac{1}{r}\frac{\partial}{\partial r}\left(r\frac{\partial u}{\partial r}\right)$. Since $x_h$ is a regular point of $\partial S_i$ that realizes the distance
 from $S_j$ at $y_h\in\partial C_i$, the ball $\B_{1+\sigma+\upsilon}(e)$ does not intersect the support of the functions $u_k$ for $k\neq j$ and small $\upsilon$ and $\sigma$. 
Therefore, multiplying inequality \eqref{laplacinec1reg} by a positive test function $\phi\in C_c^\infty \left(\sigma-\upsilon,\sigma+\upsilon\right)$, integrating by parts in $\left(\sigma-\upsilon,\sigma+\upsilon\right)$ and passing to the limit as $\ep\rightarrow 0$ along a converging subsequence, the only surviving function on the right-hand side is $\bar u_j$ and we get
\begin{equation}\label{laplacinec1reg2}\int_{\sigma-\upsilon}^{\sigma+\upsilon}\underline{u}_i¨\frac{\partial}{\partial r}\left(r\frac{\partial}{\partial r}\left(\frac{1}{r}\phi¨\right)\right)dr\leq \int_{\sigma-\upsilon}^{\sigma+\upsilon} \bar u_j¨\frac{\partial}{\partial r}\left(r\frac{\partial}{\partial r}\left(\frac{1}{r}\phi¨\right)\right)dr.\end{equation}
Let us choose a function $\phi$ which is increasing and $(\sigma-\upsilon,\sigma)$ and decreasing in $(\sigma,\sigma+\upsilon)$ and hence with maximum at $r=\sigma$, and let us estimates the left and the right hand-side of the last inequality. Estimates  \eqref{uibarbehavior} imply that $\frac{\partial \underline{u}_i}{\partial r}(\sigma^-)\le -Cr_h^\alpha$. Therefore, for small $\upsilon$ we have
\begin{equation*}\begin{split}\int_{\sigma-\upsilon}^{\sigma+\upsilon}\underline{u}_i¨\frac{\partial}{\partial r}\left(r\frac{\partial}{\partial r}\left(\frac{1}{r}\phi¨\right)\right)dr&
=-\int_{\sigma-\upsilon}^{\sigma} \frac{\partial \underline{u}_i}{\partial r}r\frac{\partial}{\partial r}\left(\frac{1}{r}\phi ¨\right)dr\\&
=-\int_{\sigma-\upsilon}^{\sigma} \left(\frac{\partial \underline{u}_i}{\partial r}(\sigma^-)+o_{\sigma-r}(1)\right)r\frac{\partial}{\partial r}\left(\frac{1}{r}\phi ¨\right)dr\\&
\ge-\int_{\sigma-\upsilon}^{\sigma} \frac{\partial \underline{u}_i}{\partial r}(\sigma^-)\left(\frac{\partial \phi}{\partial r}-\frac{1}{r}\phi¨\right)dr\\&
-o_{\upsilon}(1)\int_{\sigma-\upsilon}^{\sigma} \left(\frac{\partial \phi}{\partial r}+\frac{1}{r}\phi¨\right)dr\\&
\ge-\frac{\partial \underline{u}_i}{\partial r}(\sigma^-) \left[\phi(\sigma)-\phi(\sigma)\log\left(\frac{\sigma}{\sigma-\upsilon}\right)\right]\\&
-o_{\upsilon}(1) \left[\phi(\sigma)+\phi(\sigma)\log\left(\frac{\sigma}{\sigma-\upsilon}\right)\right]\\&
\ge (Cr_h^\alpha-o_{\upsilon}(1))\phi(\sigma).
\end{split}\end{equation*}
Similarly, using \eqref{ujbarbehavior} and integrating by parts, we get
\begin{equation*}\int_{\sigma-\upsilon}^{\sigma+\upsilon} \bar u_j¨\frac{\partial}{\partial r}\left(r\frac{\partial}{\partial r}\left(\frac{1}{r}\phi¨\right)\right) dr\leq 
(Cr_h^\beta+o_{\upsilon}(1))\phi(\sigma).\end{equation*}
From the previous estimates and \eqref{laplacinec1reg2}, letting $\upsilon$ go to 0, we obtain
$$r_h^\alpha\leq C r_h^\beta,$$
and therefore, for $h$ small enough
$$\beta\le \alpha.$$
Recalling the definitions \eqref{alphadelta} and  \eqref{betadelta} of $\alpha$ and $\beta$ respectively, we infer that
$$\theta_i\leq \theta_j.$$ This proves \eqref{angleatboundineq}. If $x_0=0$ is an interior point of $\Om$, 
exchanging the roles of $u_i$ and $u_j$, we get the opposite inequality
$$\theta_j\leq \theta_i,$$
and this proves \eqref{angleeq}  for $H$ defined as in \eqref{H2}.

Next, let us turn to the case  \eqref{H1}. 
Again we compare the mass of Laplacians of $u_i$ and $u_j$ across the free boundaries. For $\sigma<r_h$ let us define
\begin{equation}
\label{eq: set Dsigma}
D_\sigma(x_h):=\{x\in \B_{\sigma}(x_h)\,|\,d(x,\partial S_i)\leq \sigma^2\}.\end{equation}
Then, if we denote by $(D_\sigma(x_h))_1$ 
 the sets of points at $\rho$-distance 
 less than 1,
 we have  that
\begin{equation} \label{ineqmasslap0}
\int_{D_\sigma(x_h)}\Delta u_i\leq \sum_{k\neq i} \int_{(D_\sigma(x_h))_1}\Delta u_k,
\end{equation}
as in \eqref{boundlaplacian} with $(D_\sigma(x_h))_1$ in place of $\B_{1+S}(x_0)$.
By Lemma \ref{growthcond} the normal derivative of $u_i$  with respect to the inner normal $\nu_i $, at any point on the boundary $\partial \C_i $ with distance to the vertex between $r_h$ and $3r_h$ is greater than $ cr_h^\alpha$, then  
$$
\int_{D_\sigma(x_h)}\Delta u_i = \int_{\partial \C_i \cap D_\sigma(x_h)} \frac{\partial u_i}{\partial \nu_i}  d\mathcal{H}\geq c  \int_{2r_h -c\sigma}^{2r_h +C\sigma} r_h^{\alpha} dr = Cr_h^{\alpha}\sigma.
$$
Remark that
$$(D_\sigma(x_h))_1\cap\partial S_j\subset \B_{c \sigma}(y_h)\cap\partial S_j$$ therefore, for $\sigma$ small enough, again from Lemma \ref{growthcond} we have  
 $$
\int_{(D_\sigma(x_h))_1}\Delta u_j   \leq C r_h^{\beta}\sigma.
$$
 Then for $r_h$ small enough we obtain  that
$$\beta\le \alpha$$ and therefore
$$\theta_i\leq \theta_j.$$
If $x_0=0$ is an interior point of $\Om$, exchanging the roles of $u_i$ and $u_j$ we get the opposite inequality
$$\theta_j\leq \theta_i.$$
This concludes the proof of the  theorem in the particular case in which $\partial S_i$ and $\partial S_j$ are  locally a cone around 0 and $y_0$ respectively. 

We are now going to explain how to adapt the proof in the general case. 
\medskip

\noindent{\em  Proof of Theorem \ref{equalanglethm} in the general case.}
If $\theta_i=0$, then $\theta_i\leq\theta_j$. Assume
$\theta_i \in (0,\pi]$ and $\theta_j \in [0,\pi]$, then  for any  $0<\delta<\theta_i$, there exist
$r_\delta>0$,  a cone $\C^i_\delta$ centered at 0 and with opening $\theta_i-\delta$, and a cone $\C^j_\delta$ centered at  $y_0$   and with opening $\theta_j+\delta$ such that
\begin{equation*}\C^i_\delta\cap B_{ r_\delta}(0)\subset  S_i\cap B_{ r_\delta}(0)\quad\text{and}\quad  S_j\cap B_{ r_\delta}(y_0)\subset \C^j_\delta\cap  B_{ r_\delta}(y_0).\end{equation*}
Let $(x_h)_h$ be the sequence of regular points on $\partial S_i\cap\Om$ given by Lemma \ref{general cone} (consider $\Gamma_1$ the closest side to $S_j$) and let $r_h=d(0,x_h)$. Denote by $y_h$ the point on $\partial S_j\cap\Om$
at $\rho$-distance 1 from $x_h$. Then, $d_\rho(y_h,y_0)\leq c r_h$. 
Now, the proof of the theorem proceeds like in the previous case and we can compare the mass of the laplacians across the free boundaries of $u_i$ and $u_j$.

Let us first consider the case \eqref{H1}.  For $\sigma>0$  take $D_\sigma(x_h)$ and $(D_\sigma(x_h))_1$ as defined as in \eqref{eq: set Dsigma}. 
For $\sigma$ small enough, by Lemma  \ref{singisolated} $\partial S_i \cap D_\sigma(x_h)$ does not contain singular points and by Lemma \ref{freeboudarygraphlemC1}
it is a $C^1$ curve of the plane. 
  
By Lemma \ref{general cone}
$$
\int_{D_\sigma(x_h)}\Delta u_i = \int_{\partial S_i \cap D_\sigma(x_h)} \frac{\partial u_i}{\partial \nu_i}  d\mathcal{H}\geq  Cr_h^{\alpha_\delta}\sigma.
$$
Remark that
$$(D_\sigma(x_h))_1\cap\partial S_j\subset \B_{c \sigma}(y_h)\cap\partial S_j$$ therefore, for $\sigma$ small enough, from Lemma \ref{general conebis}, 
 as in the proof of  Lemma  \ref{singisolated},  
 we have  
 $$
\int_{(D_\sigma(x_h))_1}\Delta u_j   \leq \tilde{C} r_h^{\beta_\delta}\sigma.
$$
 Then for $h$ small enough, we obtain  that
$$\beta_\delta\le \alpha_\delta$$ and therefore
$$\theta_i\leq \theta_j.$$
If $x_0=0$ is an interior point of $\Om$, exchanging the roles of $u_i$ and $u_j$ we get the opposite inequality
$$\theta_j\leq \theta_i.$$

Next, let us turn to the case \eqref{H2}.
Then, we define, for $r\in[R_h-\upsilon,R_h+\upsilon]$, 
 \begin{equation*}
\underline{u}_i ¨ := \inf_{\partial \B_{r} (z_h)} u_i\, \quad\text{and} \quad 
\bar u_j¨ : = \sup_{\partial \B_{1+r}(z_h) } u_j\ .
\end{equation*}
 Arguing as before, and using the  Lemma   \ref{general cone}  we get 
$$\beta_\delta\leq\alpha_\delta,$$ and therefore, letting $\delta$ go to 0, we finally obtain
$$\theta_i\leq\theta_j.$$

Remark in particular that if $\theta_i>0$ then $\theta_j>0$. If $x_0=0$ is an interior point of $\Omega$, exchanging the roles of $u_i$ and $u_j$ we get the opposite inequality $\theta_j\leq\theta_i$.
  
\end{proof}

An immediate  corollary of Theorem \ref{equalanglethm} is the $C^1$-regularity of the free boundaries when $K=2$  and under the following additional assumptions on 
  $\Om$, $f_1$ and $f_2$:
\begin{equation}\label{c1regass1}\Omega:=\{(x_1,x_2)\in\real^2\,|\,g(x_2)\leq x_1\leq h(x_2),\,x_2\in[a,b]\},\quad b-a\ge4\end{equation}
where 
\begin{equation}\label{c1regass2}
\begin{cases}g,h: [a,b]\rightarrow\real \text{ are Lipschitz functions with }\\
-m_2\leq g\leq -m_1\leq M_2\leq h\leq M_1,\quad M_2\geq -m_1+4;
\end{cases}
\end{equation}
  the boundary data are such that
\begin{equation}\label{c1regass3}
\begin{cases} f_1\equiv 1,\,f_2\equiv 0\quad\text{on } \{x_1\leq g(x_2)\},\\
f_1\equiv 0,\,f_2\equiv 1\quad\text{on } \{x_1\geq h(x_2)\},\\
f_1\text{  is monotone decreasing  in }x_1\text{ on }\{x_2\leq a\}\cup\{x_2\geq b\},\\
 f_2\text{ is monotone  increasing in }x_1\text{ on }\{x_2\leq a\}\cup\{x_2\geq b\}.
 \end{cases}
 \end{equation}
These assumptions  imply that   $-u_1$ and $u_2$ are  monotone increasing in the $x_1$ direction.
Then we have the following
\begin{cor}\label{corc1reg}Assume    \eqref{mainassumpts} with  $p=1$ in \eqref{H1}. Assume in addition  $K=n=2$,  \eqref{c1regass1}, \eqref{c1regass2} and  \eqref{c1regass3}.
  Then the sets $\partial S_i$, $i=1,2$, are of class $C^1$.
\end{cor} 
\begin{proof}
We know that the sets $\partial S_i$
	 are curves of the plane at $\rho$-distance 1, one from  each other. 
	  Suppose by contradiction that  
$\partial S_1$ has an angle $\theta<\pi$ at $y_0$. In particular, there exist two $\rho$-balls  of radius 1, centered at two points  
$z,w\in \partial S_2$ that are tangent to $\partial S_1$ at $y_0$. Then, by the monotonicity property of the $u_i$'s and Theorem \ref{lem:6.1}, the arc of the $\rho$-ball of radius 1 centered at $y_0$ between the points $z$ and $w$ must be all in $\partial S_2$. This means that any point inside this arc, which is a regular point of $\partial S_2$, is at $\rho$-distance 1 from the singular point $y_0\in\partial S_1$. This  contradicts  Theorem \ref{equalanglethm}. 
We have shown that any point of the free boundaries is regular. Then by Lemma \ref{freeboudarygraphlemC1}   the free boundaries are  of class $C^1$.  This concludes the proof. 
 \end{proof}
 Another corollary of Theorem \ref{equalanglethm} is that the number of singular points is finite.

\begin{cor}\label{sigulapoints2dthm}Assume    \eqref{mainassumpts} with  $n=K=2$ and $p=1$ in \eqref{H1}. 
Assume in addition  that the supports on $\partial \Om$ of the boundary data $f_1$ and $f_2$ have a finite number of connected components. Then 
 singular points form a finite set. \end{cor}

 \begin{proof}
 From Lemma \ref{finiteconncompslem}, $S_1$ and $S_2$ have a finite number of connected components. Moreover, we recall that any connected component has to reach the boundary.
  
Let $x_0$ be a singular point belonging to  the boundary of the support of one of the limit functions $u_i$. W.l.o.g. let us assume $x_0\in\partial  S_1$.
 Let  $y_1,y_2 \in \partial S_2$ two different points where $x_0$ realizes the distance from 
  $S_2$, ($y_1,y_2\in\partial\B_1(x_0)\cap \partial  S_2$, see Figure \ref{fig:forbidenarc}).
We can   choose $y_1$  such that $\B_1(x_0)$ is the limit as $k\rightarrow+\infty$ of balls $\B_1(x_k)$   with $x_k\in\partial S_1$, tangent to points
 $y_k\in\partial  S_2$  with $y_k\to y_1$ and $x_k\to x_0$ as $k\rightarrow+\infty$. 
  Theorem \ref{equalanglethm}  implies that $S_2$ has an angle at $y_1$ and $y_2$ and the intersection of the arc  on  $\partial\B_1(x_0)$ between $y_1$ and $y_2$ with $\partial C_1$ must have empty interior. 
  This means that near $y_1$ there are points on $\partial  S_2$ outside $\overline{\B_1(x_0)}$. 
  These points are at distance greater than 1 from $x_0$
  and from any other point of $\partial S_1$ close to $x_0$ and must realize the distance from $S_1$ outside $\B_1(y_1)$, see Figure \ref{fig:forbidenarc}.  Therefore if we take a sequence $z_k$ of such points  converging to $y_1$ and we consider the corresponding tangent balls centered at points that are in $\partial S_1$ where the $z_k$'s realize the distance, we obtain a second tangent ball $\B_1(x_1)$ for $y_1$  with $x_1\neq x_0$.

 \begin{figure}
\begin{center}
\includegraphics[scale=0.4]{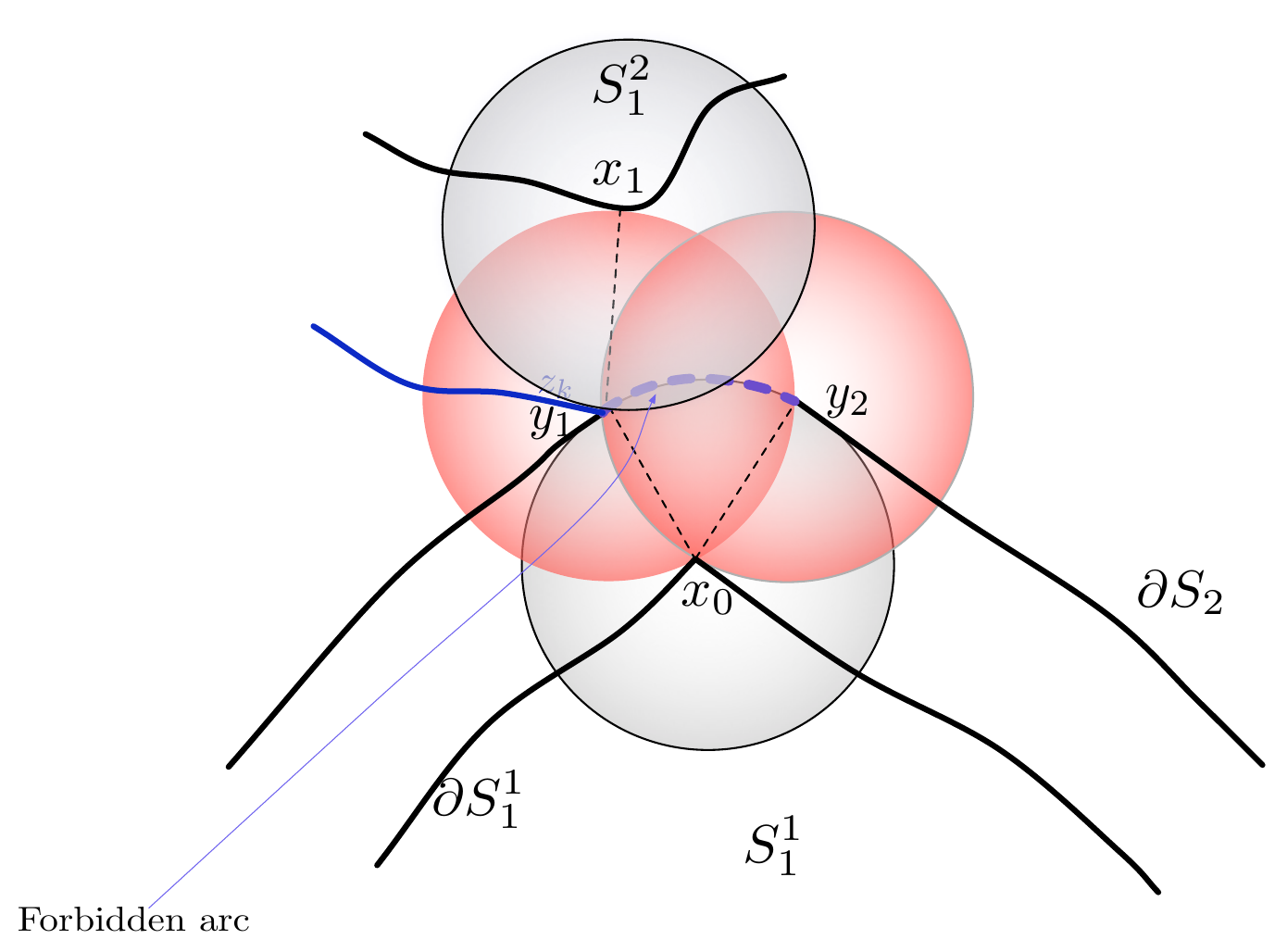}
\caption{Forbidden arc}
\label{fig:forbidenarc}
\end{center}
\end{figure}
 Now, let us denote by $S_1^1$ the connected component of $S_1$ whose boundary contains $x_0$.  Remember that since $S_1$ and $S_2$ are at $\rho$-distance 1, we have $u_1\equiv 0$  in $\overline{\B_1(y_1)}\cup\overline{\B_1(y_2)}$.  Moreover, since  the connected components of  $S_2$  whose boundaries  contain $y_1$ and $y_2$ must reach  the boundary of $\Om$, they separate the components of $S_1$  whose boundaries contain $x_0$ and $x_1$. Therefore $x_1$ must belong to the boundary of different components of $S_1$. 
 The same argument that we have used for $x_1$ and $x_0$ proves also that $y_1$ and $y_2$ must belong to the boundary of different components of $C_1$. 

 We conclude that a singular point  $x_0$ of $S_1$ involves at least four different connected components and there correspond to it another singular point, $x_1$, belonging  to a different component of $S_1$ (see Figure \ref{fig:hastobethis}).
\begin{figure}
\begin{center}
\includegraphics[scale=0.4]{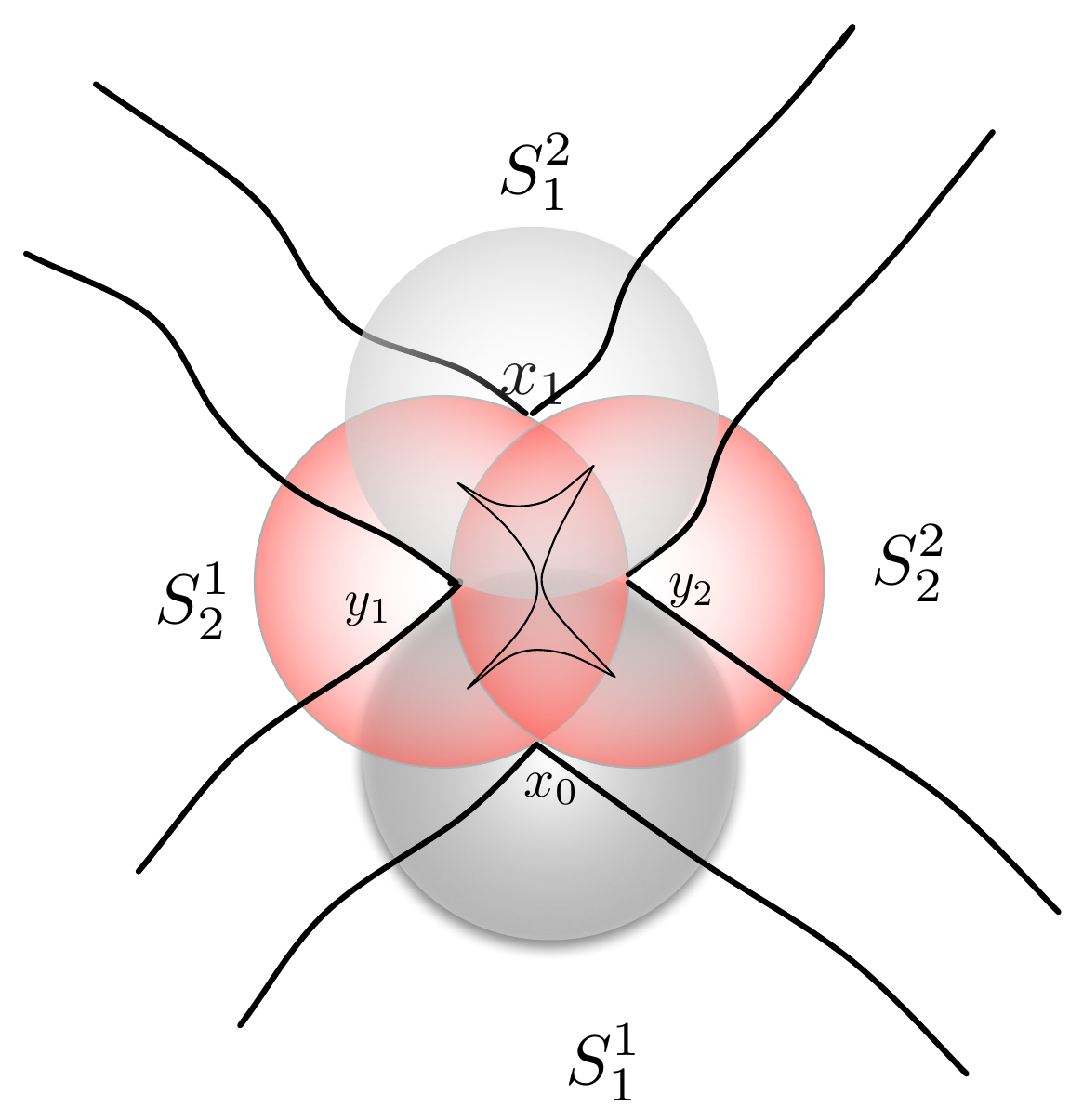}
\caption{A singular point involving  four components}
\label{fig:hastobethis}
\end{center}
\end{figure}
Assume w.l.o.g. that $x_1\in\partial S_1^2$. 
 Since all the connected components  must reach the boundary of $\Om$, $x_1$ is the only singular point of $S_1^2$ corresponding to a singular point of $S_1^1$.  Since the connected component of $S_1$ are finite, we infer that there is a finite number of singular points on $\partial S_1^1$. This argument applied to any connected component of $S_1$ shows that singular points of $S_1$ are finite. This concludes the proof of the theorem.  
 
   \end{proof}



\subsection{Lipschitz regularity of the free boundaries}
In this section, we will show, under some additional assumptions on the domain $\Om$ and the boundary data $f_i$, that we can construct a solution of  problem   \eqref{eq:stmt3} such that 
the free boundaries   $S_i$  of the limiting functions have the following properties: if  $S_i$ has an angle $\theta$ at a singular point, then $\theta>0$. 
This result can be rephrased by saying that the free boundaries are Lipschitz curves of the plane.  
Let us make the assumptions precise. We assume that the domain $\Om$  has the property that for any point of the boundary there are tangent 
$\rho$-balls of radius 
$1+\eta$, with $\eta>0$   contained  in $\Om$ and in its complementary. Precisely:
\begin{equation}
\label{Om2dasslip}
\begin{cases}
\Om \text{ is a bounded domain of $\real^2$;}\\
\text{$\exists\eta>0$ such that  $\forall\,x\in\partial\Om$,  $\exists\,\B_{1+\eta}(y),\,\B_{1+\eta}(z)$ such that}\\ 
\text{\quad$x\in\partial \B_{1+\eta}(y)\cap \partial \B_{1+\eta}(z)$, $\B_{1+\eta}(y)\subset \Om$, and $\B_{1+\eta}(z)\subset \Om^c$}.\\
\end{cases}
\end{equation}
On the boundary data $f_i$, $i=1,\ldots,K$, we assume, 
\begin{equation}
\label{fidasslip}
\begin{cases} 
 f_i\equiv 1\text{ in }\text{supp}\,f_i;\\
 \exists\text{ $c>0$ s. t.   }\forall x\in\partial\Om\cap\supp\,f_i,\, |\B_r(x)\cap \supp\,f_i|\ge c|\B_r(x)|,\\
 d_\rho(\text{supp}\,f_i,\text{supp}\,f_j)\ge1,\,i\neq j, \\
 d_\rho(\text{supp}\,f_i\cap\partial\Om,\text{supp}\,f_{i+1}\cap\partial\Om)=1,\text{ where  }f_{K+1}:=f_1;\\
 \Gamma_i:=\text{supp}\,f_i\cap\partial \Om\text{ is a connected  ($C^2$-) curve of }\partial\Om.\\
 \end{cases}
 \end{equation}
We are going to build a solution of  \eqref{eq:stmt3} such that  the support of any limiting function $u_i$ contains a full neighborhood of $\Gamma_i$ in $\Om$   
with Lipschitz boundary. Then we prove that the free boundaries are Lipschitz. 
  In order to do it, we first prove the existence of a solution $(u_1^\ep,\ldots,u_K^\ep)$ of an obstacle problem associated to system \eqref{eq:stmt3}. Then we show that the functions $u_i^\ep$'s never touch the obstacles, implying that $(u_1^\ep,\ldots,u_K^\ep)$ is actually solution of  
  \eqref{eq:stmt3}. We consider obstacle functions $\psi_i$, for $i=1,\ldots,K$ defined as follows.
  Let $y_1^i,\,y_2^i$ be the endpoints of the curve $\Gamma_i$. For $0<\mu<\lambda<1$, we set:
  $$\Gamma_i^\mu:=\{x\in\Om^c\,|\,d(x,\Gamma_i)=\mu\},$$
  $$\Gamma_i^{\mu,\lambda}:=\{x\in\Gamma_i^\mu\,|\,d(x,y_1^i),\,d(x,y_2^i)\geq \lambda\}.$$
  For $\mu$ and $\lambda$ small enough, $\Gamma_i^{\mu,\lambda}$ is a $C^{1,1}$ curve of $\Om^c$ with endpoints $z_1^i$, $z_2^i$ such that $d(z_l^i,y_l^i)=\lambda$, $l=1,2$. We finally set
  \begin{equation}\label{Ai}A_i:=\{x\in\Om\,|\,d(x,\Gamma_i^{\mu,\lambda})< \lambda\}=\Om \cap\left(\displaystyle\cup_{x\in\Gamma_i^{\mu,\lambda}}B_\lambda(x)\right).\end{equation} 
  
\begin{figure}
{\includegraphics[trim={0 3.7cm 0 0cm}, clip=true, width=13cm]{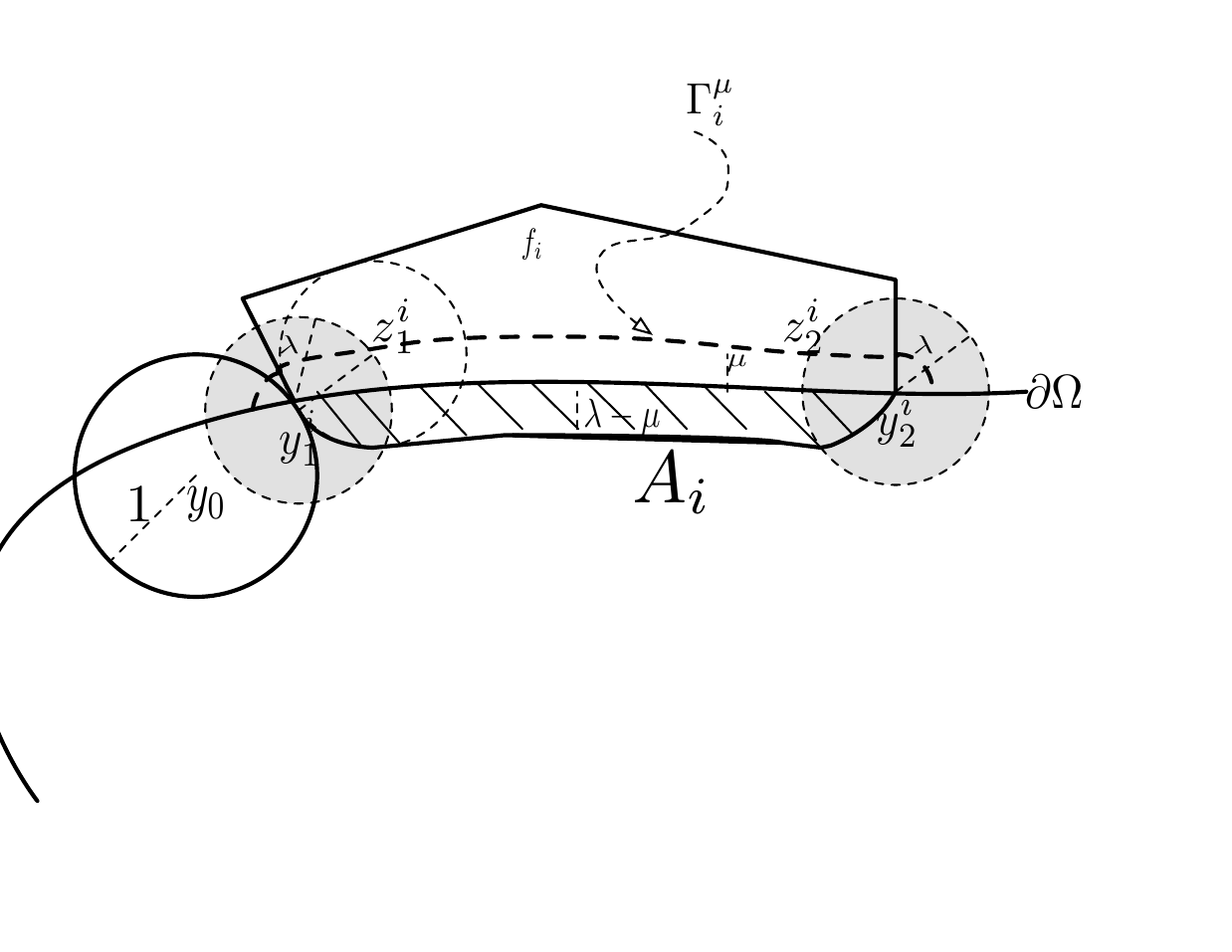}}
\caption{Construction of obstacle}
\label{figure:obstacleconst}
\end{figure}

  Remark that 
 $$\partial A_i=\Gamma_i\cup ( \partial A_i\cap\Om),$$ where 
 $\partial A_i\cap\Om$ is given by the union of two arcs contained respectively in the balls $B_\lambda(z_1^i)$ and $B_\lambda(z_2^i)$, and a curve contained in the set of points of $\Om$ at distance $\lambda-\mu$ from $\Gamma_i$, (see Figure \ref{figure:obstacleconst}). 
Denote by $\alpha_l^i$ the angle of $A_i$  at $y_l^i$, $l=1,2$. Remark that  
\begin{equation}\label{alphabehavior}
\begin{cases}
\alpha_l^i\to\frac{\pi}{2}+o_\lambda(1)&\text{if }\mu\to 0\\
\alpha_l^i\to 0&\text{if }\mu\to\lambda,
\end{cases}
\end{equation}
where $o_\lambda(1)\to0$ as $\lambda\to0$.

We take as obstacles the functions   $\psi_i:(\Om)_1\to\real$ defined as the solutions of the following problem, for $i=1,\ldots,K$,
\begin{equation}\label{psi_i}\begin{cases}\Delta \psi_i=0&\text{in }A_i\\
\psi_i=f_i&\text{on } (\partial \Om)_1\\
\psi_i=0&\text{in }\Omega\setminus A_i.
\end{cases}
\end{equation}
In this section we deal with the solution $(u_1^\ep,\ldots,u_K^\ep)$ of the following obstacle system problem: for $i=1,\ldots,K$,

\begin{equation}\label{obstaclepb}
\begin{cases}
u_i^\ep\geq \psi_i&\text{in }\Omega,\\
\ds \Delta u_i^\ep (x) \leq \frac1{\ep^2} u_i^\ep (x) \sum_{j\ne i} H(u_j^\ep) (x)&\text{in }\Omega,\\
\ds \Delta u_i^\ep (x) = \frac1{\ep^2} u_i^\ep (x) \sum_{j\ne i} H(u_j^\ep) (x)&\text{in }\{u_i^\ep>\psi_i\}\\
\noalign{\vskip6pt}
\ds u_i^\ep = f_i &\text{on }(\partial\Omega)_1.\\
\end{cases}
\end{equation}
 In the whole section we make the following assumptions:
 \begin{equation}\label{finalassmpt2dlip}\begin{cases}\ep>0,\\
 \text{\eqref{Om2dasslip} and  \eqref{fidasslip} hold true},\\
  \text{$H$ is either of the form (\ref{H1}) with $p=1$, or (\ref{H2}) and 
(\ref{varphidecay}) holds true};\\
\text{For $i=1,\ldots,K$,  $A_i$ and $\psi_i$ are defined by \eqref{Ai} and 
\eqref{psi_i} respectively.} 
\end{cases}
\end{equation}

\begin{thm}\label{thm:existenceobstacle}Assume \eqref{finalassmpt2dlip}. 
Then, there exist continuous positive functions $u_1^\ep,\ldots,u_K^\ep$, depending on the pa\-ra\-meter $\ep$, 
viscosity solutions of the problem \eqref{obstaclepb}. In particular 
\begin{equation}\label{usatisfyoutssupppsi}\ds \Delta u_i^\ep (x) = \frac1{\ep^2} u_i^\ep (x) \sum_{j\ne i} H(u_j^\ep) (x)\quad\text{in }\Om\setminus A_i.\end{equation}
 Moreover, for $i=1,\ldots,K$,  
\begin{equation}\label{uisubostc}\Delta u_i^\ep\geq 0\quad \text{in }\Om,\end{equation}
in the viscosity sense. 
\end{thm}
\begin{proof}
The proof of the existence of a solution $(u_1^\ep,\ldots,u_K^\ep)$ of \eqref{obstaclepb} is a slightly modification of the proof of Theorem~\ref{thm:existence}. Here \begin{equation*}\begin{split}
\Theta = \Big\{ &(u_1, u_2,\ldots, u_K)\,|\,u_i:\Om\rightarrow\real \text{ is continuous, }\, \psi_i\le u_i \le \phi_i\ \text{ in }\ \Omega,\,
u_i = f_i\ \text{ on }\ (\partial\Omega)_1\Big\}.
\end{split}
\end{equation*} 
In the set $\Om\setminus A_i$, we have that $u_i^\ep>0=\psi_i$ which implies \eqref{usatisfyoutssupppsi}. Inequality \eqref{uisubostc} is a consequence of the following facts: 
in the set $\{u_i^\ep>\psi_i\}$ we have $\Delta u_i^\ep  = \frac1{\ep^2} u_i^\ep  \sum_{j\ne i} H(u_j^\ep) \geq 0$; 
in the interior of the set $\{u_i^\ep=\psi_i\}$, $\Delta u_i^\ep=\Delta \psi_i=0$;  the free boundaries 
$\partial \{u_i^\ep>\psi_i\}$ have locally finite $n-1$-Hausdorff  measure, see \cite{caffarelli_obstacle_1998}.
\end{proof}
 \begin{thm}\label{convergencecor2dlip} Assume  \eqref{finalassmpt2dlip}.  Let $(u_1^\ep,\ldots,u_K^\ep)$ be
viscosity solution of the problem  \eqref{obstaclepb}.
Then, there exists a subsequence $(u_1^{\ep_l},\ldots,u_K^{\ep_l})$ and continuous functions $ (u_1,\ldots, u_K)$ defined on $\overline{\Omega}$, such that 
\begin{equation*} (u_1^{\ep_l},\ldots,u_K^{\ep_l})\to  (u_1,\ldots, u_K)\quad\text{as }l\to+\infty,\quad\text{a.e. in }\Om\end{equation*} and the convergence of 
$u_i^{\ep_l}$ to $u_i$ is locally uniform in the support of $u_i$. 
Moreover, we have:
\begin{itemize}
\item[i)] the $u_i$'s  are  locally Lipschitz continuous in $\Om$, in particular, there exists $C_0>0$ such that, if $d_\rho(x,\partial\Om)\ge r$, then
\begin{equation}\label{ulipestim2d}|\nabla u_i(x)|\le\frac{C_0}{r}.
\end{equation}
 \item[ii)]the $u_i$'s have disjoint supports, more precisely:
$$u_i\equiv 0\quad\text{in the set }\quad  \{x\in\Om\,|\,d_{\rho}(x,\supp\, u_j)\le1\}\quad\text{for any }j\neq i.$$
\item[iii)] $\Delta u_i =0$ when $u_i >0$.
\item[iv)]$u_i\geq \psi_i$ in $\Om$.
\item[v)]$u_i=f_i$ on $\partial\Om$.
\end{itemize}
\end{thm}
\begin{proof}
The convergence theorem is again a consequence of  Lemma  \ref{uj=0closui}, Corollary \ref{convedisjointcor} and Lemma \ref{uj=0closfi} which hold true with 
$\supp\, f_i$ and $\supp\, f_j$ replaced respectively by $\supp\,\psi_i=A_i$ and 
$\supp\,\psi_j=A_j$ (in Lemma  \ref{uj=0closui} and Corollary \ref{convedisjointcor}), and $\overline{\Gamma}_j^\sigma$ defined as the set $\{\psi_j\ge\sigma\}$ (in Lemma \ref{uj=0closfi}).
Estimates \eqref{lipcor5.4} of Corollary \ref{convedisjointcor} in particular imply   \eqref{ulipestim2d}.
Property (iv) is an immediate  consequence of $u^\ep_i\geq \psi_i$ in $\Om$. Finally, (v) is implied by the fact that $\psi_i\leq u^\ep_i\leq \phi_i$  in $\Om$, and 
$\phi_i=\psi_i=f_i$ on $\partial\Om$, where $\phi_i$ is given by \eqref{phibarrier}.
\end{proof}
As proven in Corollary \ref{tangebtballcor}, one can show that the free boundaries satisfy the exterior $\rho$-ball condition with radius 1, that they have  finite 1-Hausdorff dimensional measure and that the distance between the support of two different functions is precisely one.
We are now going to prove that, if $\lambda-\mu$ is small enough, then any  solution of the obstacle problem  \eqref{obstaclepb} never touches the obstacles inside the domain $\Om$. To this aim, we first need the following lemma:
\begin{lem}
Assume  \eqref{finalassmpt2dlip}. Then, there exists $c>0$ such that, for $i=1,\ldots, K$, we have
\begin{equation}\label{normalderivativespsi}\frac{\partial\psi_i}{\partial \nu_i}(x)\leq -\frac{c}{\lambda-\mu}\quad\text{for any }x\in \partial A_i\cap\Om,\end{equation}
where $\nu_i$ is the exterior normal vector to the set $A_i$.
\end{lem}
\begin{proof}
Fix any point $x_0\in \partial A_i\cap\Om$. Then, by definition of $A_i$, there exists a point $z\in\Om^c$ such that $d(z,\partial\Om)=\mu$, $B_\lambda(z)\cap\Om\subset A_i$ 
 and $x_0\in\partial B_\lambda(z)$. Consider now the ring $\{x\,|\,\mu<|x-z|<\lambda\}$ and the barrier function $\phi$ solution of 
 $$\begin{cases}\Delta\phi=0&\text{in }\{x\,|\,\mu<|x-z|<\lambda\}\\
 \phi=1&\text{on }\partial B_\mu(z)\\
 \phi=0&\text{on }\partial B_\lambda(z).\\
 \end{cases}
$$
The function $\psi_i$ is harmonic in $B_\lambda(z)\cap\Om$, $\psi_i\ge0=\phi$ on $\partial B_\lambda(z)\cap\Om$ and  $\psi_i=1\geq\phi$ on $\partial\Om\cap  B_\lambda(z)$. 
Therefore by the comparison principle, we have that
$\psi_i(x)\ge\phi(x)$ for any $x\in B_\lambda(z)\cap\Om$, and this implies  \eqref{normalderivativespsi} at $x=x_0$.
\end{proof}

  \begin{thm}\label{u>psi} Assume  \eqref{finalassmpt2dlip}. Let $(u_1,\ldots,u_K)$ be the limit of a converging subsequence of $(u_1^\ep,\ldots,u_K^\ep)$, solution of  \eqref{obstaclepb}. Set $a:=\lambda-\mu$. Then, there exists $a_0>0$ such that for any $a<a_0$, we have, for $i=1,\ldots,K$, 
  \begin{equation}\label{u>psoeq}u_i>\psi_i\quad\text{in }\overline{A}_i\cap\Om.
  \end{equation}
\end{thm}

\begin{proof}
In order to prove \eqref{u>psoeq}, it is enough to show that
\begin{equation}\label{u>psibpooubndary}u_i(x)>\psi_i(x),\quad\text{for any }x\in\partial A_i\cap\Om.
\end{equation}
Indeed, if \eqref{u>psibpooubndary} holds true, since by \eqref{psi_i} and Theorem \ref{convergencecor2dlip}, both $u_i$ and $\psi_i$ are harmonic in $A_i$, the strong maximum principle  implies $u_i>\psi_i$ in $A_i$. This and \eqref{u>psibpooubndary} give \eqref{u>psoeq}.
Suppose by contradiction that there exists a point $x_0\in\partial A_i\cap\Om$ such that $u_i(x_0)=\psi_i(x_0)=0$. Then, by \eqref{normalderivativespsi}, we have that
\begin{equation}\label{normaledeivu}
\frac{\partial u_i}{\partial \nu_i}(x_0)\le \frac{\partial\psi_i}{\partial \nu_i}(x_0)\leq -\frac{c}{\lambda-\mu}=-\frac{c}{a}. 
\end{equation}
Assumptions  \eqref{Om2dasslip}  imply that if the angles $\alpha^i_l$  of $A_i$  at $y^i_l$, $l=1,2$,  are small enough,   the sets defined by 
 $$\Sigma_i=\{y:y=x+\nu_i(x), x\in \partial A_i\cap\Om\}$$ 
 and
 $$\Sigma^-_i=\{ y: y= x+t\nu_i(x), x\in \partial A_i\cap\Om, 0<t<1\}$$
 are compactly supported  in $\Omega$  and 
\begin{equation}\label{acondi 2}
d_\rho(x_0,\supp\,\psi_j)>1\quad\text{for any }j\ne i.
\end{equation} Therefore, by \eqref{alphabehavior}, we can choose $a$ so small that \eqref{acondi 2} holds true. Moreover,  from \eqref{acondi 2}, there exists a  small $\sigma>0$ 
 such that $\B_{1+\sigma}(x_0)\cap\supp\,\psi_j=\emptyset$, $j\neq i$,  and from \eqref{obstaclepb}, we know that  $$\Delta u_j^\ep \geq \frac1{\ep^2} u_j^\ep   H(u_i^\ep)\qquad \text{in }\B_{1+\sigma}(x_0)$$ 
(consider $u^\ep_j$ extended by zero if the ball falls out of $\Om$). When $H$ is defined as in \eqref{H1} with $p=1$, arguing as  in \eqref{ineqmasslap0} in proof of Theorem \ref{equalanglethm} we obtain that
\begin{equation*} 
\sum_{j\neq i} \int_{(D_\sigma(x_0))_1}\Delta u_j\geq \int_{D_\sigma(x_0)}\Delta u_i.
\end{equation*}
Now, since $u_i\geq\psi_i>0$ in $A_i$ and $u_i(x_0)=0$, the point $x_0$ belongs to  $\partial \{u_i>0\}\cap \partial A_i\cap\Om$. Since $\partial A_i\cap\Om$ has an interior  tangent ball and $\partial \{u_i>0\}$ has a exterior tangent ball,  we know  that $x_0$ is a regular point. Since the set of regular points is an open set, see Lemma \ref{singisolated}, for  $\sigma$ small enough we have
\begin{equation}\label{eq: leftestimation0}
\int_{D_\sigma(x_0)}\Delta u_i\geq - \int_{\partial \{u_i>0\} \cap D_\sigma(x_0)} \frac{\partial u_i}{\partial \nu_i}d\mathcal{H},
\end{equation}
where $\nu_i $ is still the exterior normal vector to $A_i$. On another hand, if $y_0$ is  the point  that realizes the distance one with $x_0$,  assume w.l.o.g. that $y_0 \in \partial \supp u_j$,  $y_0$ has to be in $\Sigma_i$ and $y_0$ has to be a regular point. Then,  for  $\rho$ small enough such that $\partial \{u_j>0\}\cap B_{\rho}(y_0)$ is $C^1$ we have\begin{equation*}
\begin{split}
 \int_{B_\rho(y_0)} \Delta u_j &= - \int_{\partial \{u_j>0\} \cap B_\rho(y_0)}  \frac{\partial u_j}{\partial \nu_j}   d\mathcal{H}.
\end{split}
\end{equation*}
Now, using the fact that  for $\sigma $ small enough such that $\rho >c \sigma$,  
$\supp u_j \cap (D_\sigma(x_0))_1 \subset \B_{c\sigma} (y_0)$, we have
\begin{equation}\label{eq: leftestimation} 
 \int_{B_{c\sigma}(y_0)} \Delta u_j \geq \int_{(D_\sigma(x_0))_1}\Delta u_i .
\end{equation}
Putting all together,  dividing  \eqref{eq: leftestimation0} and  \eqref{eq: leftestimation} respectively by $\mathcal{H}(\partial \{u_i>0\} \cap D_\sigma(x_0))$ and $\mathcal{H}(\partial \{u_j>0\}\cap B_{c\sigma}(y_0))$, and passing to the limit when $\sigma \rightarrow 0$
 we obtain
\begin{equation} \label{nucondu_j}
- \frac{\partial u_j}{\partial \nu_j}(y_0)\geq - c\frac{\partial u_i }{\partial \nu_i}(x_0) \geq\frac{\tilde{c}}{a}.
\end{equation}We are now going to show that \eqref{nucondu_j} yields a contradiction.
Indeed, the point $y_0$ realizes its distance from the set $\{u_i>0\}$ at $x_0$, therefore the ball $\B_1(y_0)$ is tangent to $\{u_i>0\}$ at $x_0$. Moreover, since $A_i\subset\{u_i>0\}$,
 the ball
$\B_1(y_0)$ is tangent to $A_i$ at $x_0$. On the other hand, for  $a$ small enough, by assumption \eqref{Om2dasslip}, $\B_1(y_0)$ is contained in $\Om$. In particular, the $\rho$-distance of $y_0$ from $\partial\Om$ is greater than $1$. Therefore, from \eqref{ulipestim2d}, we infer that $|\nabla u_j(y_0)|\leq C_0$, which is in contradiction with \eqref{nucondu_j} for $a$ small enough.


When $H$ is defined as in \eqref{H2}, we argue as in case (b)  in the proof of Theorem \ref{lem:6.1} and similarly, we get a contradiction for $a$ small enough.

\end{proof}

\begin{cor}\label{freeobstasolcor}
Under the assumptions of Theorem \ref{u>psi}, if $a<a_0$ then $(u_1^\ep,\ldots,u_K^\ep)$ is solution of the following problem
\begin{equation}\label{obstaclepbfree}
\begin{cases}
u_i^\ep\geq \psi_i&\text{in }\Omega,\\
\ds \Delta u_i^\ep (x) = \frac1{\ep^2} u_i^\ep (x) \sum_{j\ne i} H(u_j^\ep) (x)&\text{in }\Omega\\
\ds u_i^\ep = f_i &\text{on }(\partial\Omega)_1.\\
\end{cases}
\end{equation}
In particular, $(u_1^\ep,\ldots,u_K^\ep)$ is solution of \eqref{eq:stmt3}. 
\end{cor}
We are now ready to show that free boundaries are Lipschitz.

\begin{thm}\label{lipfreeboundary}
Let $(u_1^\ep,\ldots,u_K^\ep)$ be the solution of  \eqref{eq:stmt3} given by Corollary \ref{freeobstasolcor}. Let $(u_1,\ldots,u_K)$ be the limit as $\ep\to0$ of a converging subsequence, then the free boundaries $\partial\{u_i>0\}$,  $i=1,\ldots,K$, are Lipschitz curves of the plane.
\end{thm}
\begin{proof}

By contradiction let's assume that the free boundaries are not Lipschitz. This would imply that there exists at least one singular point with asymptotic cone with zero opening. 

Let $ x_0$ be an  interior singular point   with asymptotic cone with zero angle. W.l.o.g. suppose $x_0 \in \partial \{u_1 >0\}$. Let $e_1$ be the  line perpendicular  to the cone axis and passing through $x_0$, in which we choose an orientation such that the cone is below the axis  $e_1$.  As we proved in Theorem \ref{equalanglethm} and Corollary \ref{sigulapoints2dthm} there exist $y_0 $ and $y_1$, with  $y_0, y_1 \in\cup_{j\neq1} \partial \{u_j>0\}$ singular points at distance one from $x_0$  with asymptotic cones with zero opening. Also, by Theorem \ref{lem:6.1} for any regular point $x \in \partial \{u_1 >0\}\cap B_1(x_0)$ there exists a correspondent $y \in \cup_{j\ne 1}\partial \{u_j>0\} $  such that
$$
y= x+\nu(x)
$$
with $\nu (x)$ the external normal vector to $\partial \{u_1 >0\}$   at $x$. Observe that $y_0, y_1 $ must lie on  $e_1$. In fact, let $x_n^l \in \partial \{u_1 >0\}$ be regular points converging to $x_0,$ $x_n^l \rightarrow x_0$ as $n\to+\infty$,  from the left side of the cone axis and  let $x_n^r \in \partial \{u_1 >0\}$ be the regular points such that
 $x_n^r \rightarrow x_0$  as $n\to+\infty$, from the right side of the cone axis. 
Then, the limit of the normal vectors $\nu (x_n^l ) \rightarrow \nu^l$ and $\nu (x_n^r) \rightarrow \nu^r$, are both on the direction $e_1$ since they are orthogonal to the cone axis. Let $y_0$ and $y_1$ be w.l.o.g.  the points defined by
$$
y_0=x_0 + \nu^l \qquad y_1=x_0 + \nu^r.
$$
So we have to have three singular points at distance one, all on the line $e_1$. Repeating the same argument and using $y_1$ as the reference singular point  now, we conclude that there must exist another singular point, $y_2$,
with 0 opening cone,  at distance one from $y_1$ and also on the axis $e_1$. Iterating, we will be able to proceed until the prescribed boundary of the domain stops us from finding the next point. We will have all singular points with cone with zero opening aligned on the axis $e_1$, until we reach the boundary $\partial \Omega$ and we cannot proceed with this process, i.e., until we cannot obtain the next point aligned in the direction of $e_1$ which implies that  $\partial \Om$ crosses the axis  $e_1$ and the distance of $y_k$ to the boundary of $\Om$ along $e_1$ is less or equal than 1.

   Now, there are two cases: either  $y_k\in\partial\Om$ or $y_k\in\Om$.  If $y_k\in\partial \Om$ assume w.l.o.g. that $y_k\in\partial \{u_1 >0\}$. Since   $u_1\geq \psi_1$ we have $A_1\subset \{u_1>0\}$ and that $y_k$ must coincide with one of the points $y_l^1$, $l=1,2$, endpoints of the curve $\Gamma_1$.  Indeed, by the forth  assumption in  \eqref{fidasslip},  no points of $\partial\{u_1>0\}$ are on $\partial \Omega$ between the curves $\Gamma_1$ and 
   $\Gamma_2$,  and $\Gamma_1$ and $\Gamma_K$. Assume w.l.o.g. that $y_k=y_1^1$.
Let $\theta$ be the angle of $\partial \{u_1>0\}$ at $y_1^1$. Then,  from \eqref{angleatboundineq} of Theorem \ref{equalanglethm} applied to $y_k=y_1^1$ and $y_0=y_{k-1}$, we get $\theta=0$. On the other hand, since  $A_1\subset \{u_1>0\}$ then $\theta\geq\alpha_1^1>0$, where $\alpha_1^1$ is the angle of $A_1$ at $y_1^1$. We have obtained  a contradiction. 
\begin{figure}
\label{fig:lips}
{\includegraphics[trim={0 0cm 0 0cm}, clip=true, width=12cm]{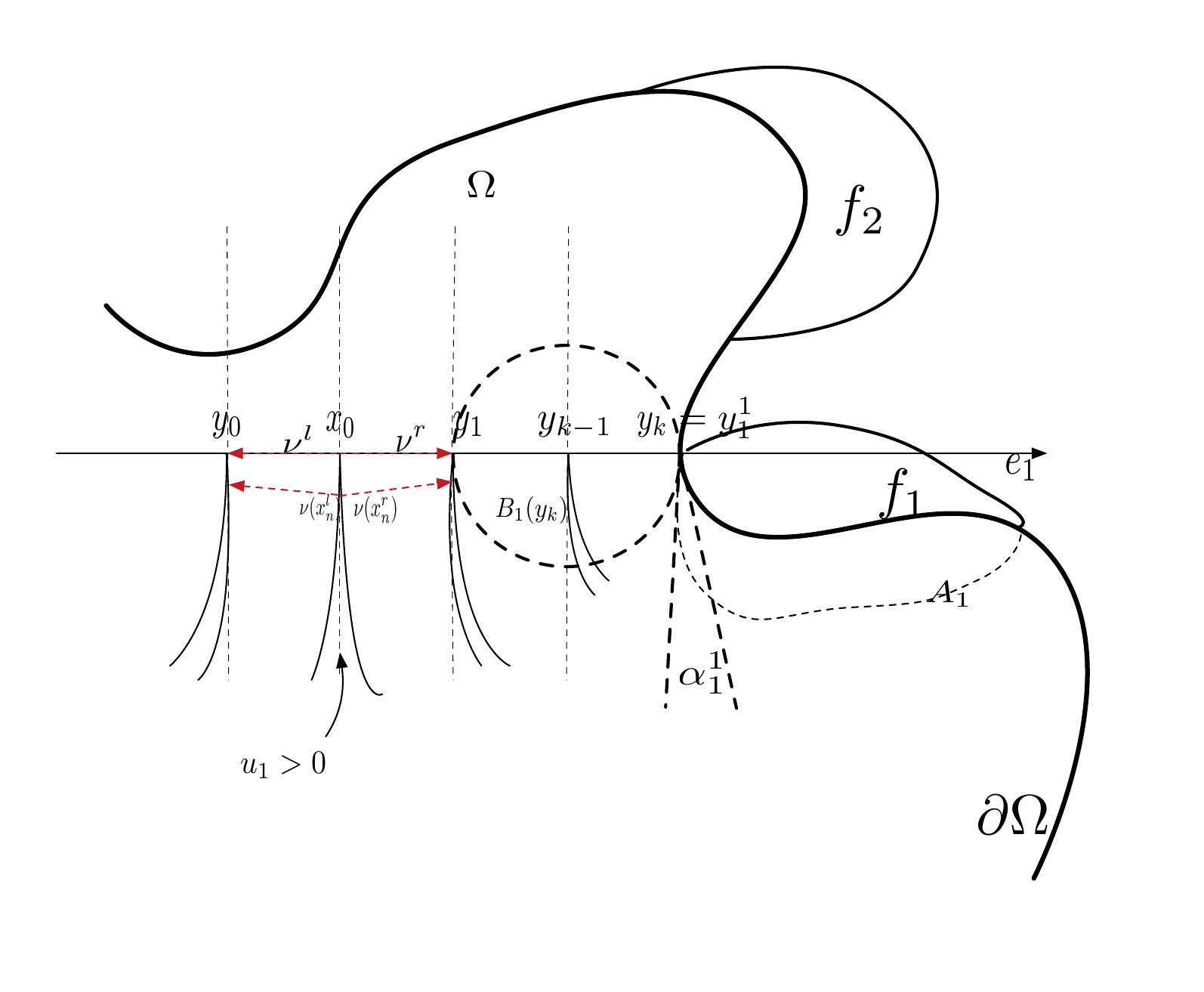}}
\caption{Contradiction in the case $y_k \in \partial  \Om$}
\end{figure}
Suppose now that $y_k$ is an interior point. Again, assume w.l.o.g. that $y_k\in\partial \{u_1 >0\}$.   Let $z_k \in \partial \Om$ be the closest point to $y_k$ in the direction $e_1$ and $d(y_k, z_k)=l<1$.   
Recall that  by (\ref{Om2dasslip}) there is an exterior tangent ball at $z_k$, $B_{1+\eta}$,  so once the axis $e_1$ is crossed, $\Om$  will remain outside of the tangent ball at $z_k$ and so $\partial \Om$ will not cross again $e_1$ in $\overline{\B}_1(y_k)$.
We know that $z_k $ cannot belong to $\partial \{u_j>0\}$ since it does not respect the distance one and also $A_j\subset \{u_j>0\}$. 
\begin{figure}
{\includegraphics[trim={0 0cm 0 0cm}, clip=true, width=12cm]{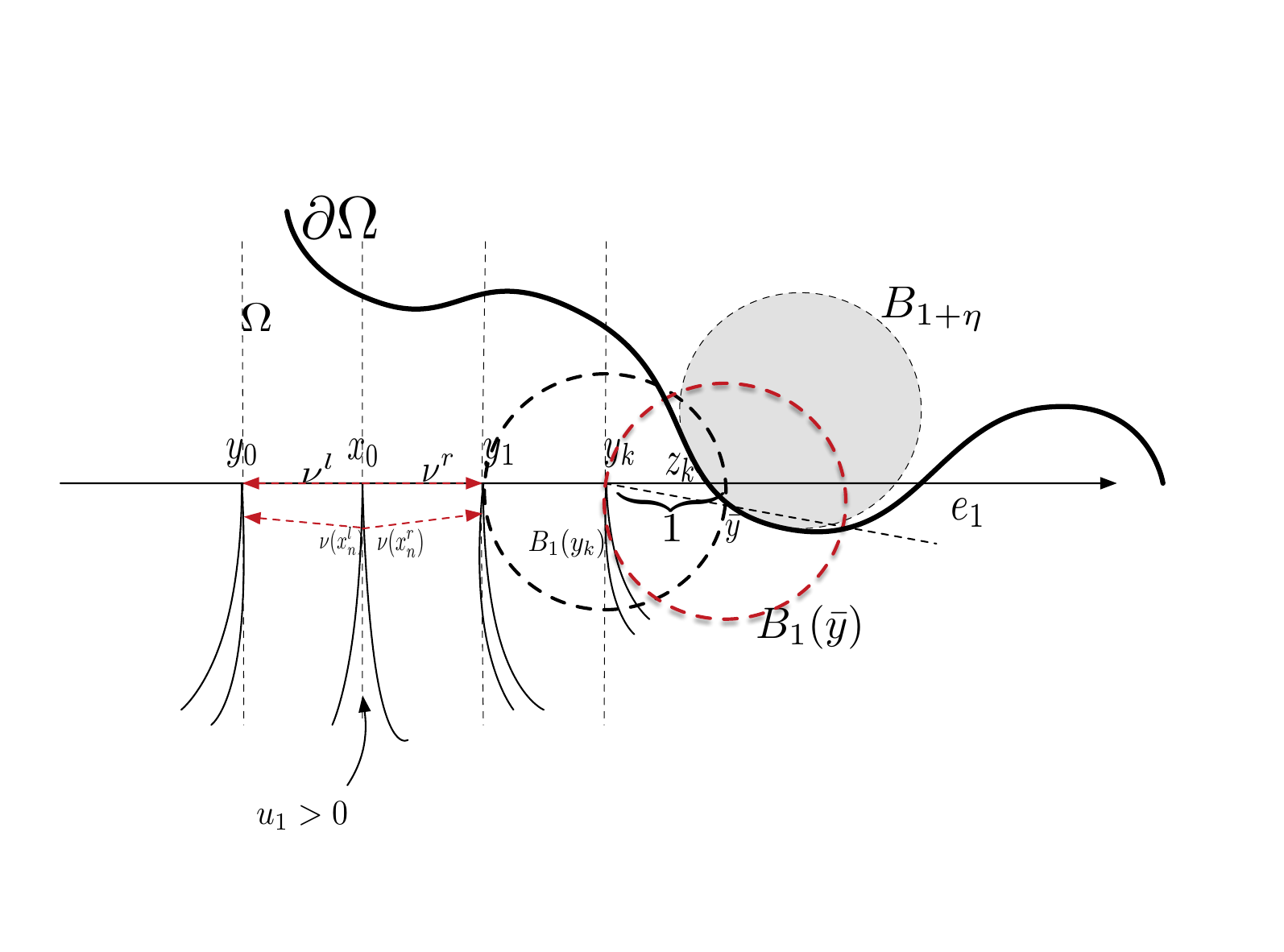}}
\caption{Contradiction in the case $y_k  \in \Om$}
\label{fig:lips}
\end{figure}
And by Theorem \ref{lem:6.1} for any point on the free boundary  there exists a correspondent point at distance one belonging to the support of another function. Taking in account the previous case, the only option is that the  point that realizes the distance from $y_k$, $ \bar{y}$, belongs to  $B_1(y_k)$ and it must be such that the angle between $e_1$ and the line that contains both $y_k$ and $\bar{y}$ is strictly positive, see Figure \ref{fig:lips}. Therefore,  we must conclude that $B_1(\bar{y})\cap \{u_1>0\}\neq \emptyset$.

We have obtained  a contradiction. We conclude that the free boundaries cannot have a zero angle at a singular point, therefore they are Lispschitz curves of the plane. \end{proof}
\section{ A relation between the normal derivatives at the free boundary}\label{normalderrelatsec}
In this section we restrict ourself to the following case: 
\begin{equation}\label{normderassum}
\begin{cases}K=2\\
H\text{ defined like in \eqref{H1}, with}\\
p=1,\,\varphi\equiv 1\text{ and }\rho\text{ the Euclidian norm}.\\
\end{cases}
\end{equation}
Therefore, the system \eqref{eq:stmt3} becomes
$$\Delta u_1^\ep (x) = \frac1{\ep^2} u_1^\ep (x) \int_{B_1 (x)} u_2^\ep (y)\, \text{d}y\quad\text{in }\Omega,$$
$$\Delta u_2^\ep (x) = \frac1{\ep^2} u_2^\ep (x) \int_{B_1 (x)} u_1^\ep (y)\, \text{d}y\quad\text{in }\Omega,$$
where we denote by $B_1 (x)$ the Euclidian ball of radius 1 centered at $x$. 
Let $(u_1,u_2)$ be the limit functions of a converging subsequence that we still denote $(u_1^\ep,u_2^\ep)$ and  for $i=1,2$ let
$$S_i:=\{u_i>0\}.$$
From Section \ref{section: strip} we know that the  $u_i$'s have disjoint support and that there is a strip of width  exactly one that separates  $S_1$ and $S_2$.
 Moreover, Corollary \ref{tangebtballcor} guarantees that at any point of the boundary of the two sets, the principal curvatures are less or equal 1. 
 For $i=1,2$, let $x_i\in\partial S_i$ be such that $x_1$ is at distance 1 from $x_2$, $\partial S_i$ is of class $C^2$ in a neighborhood of $x_i$, 
 and all the principal curvatures  of 
 $\partial S_i$ at $x_i$ are strictly less than 1. 
 Without loss of generality we can assume $x_1=0$ and $x_2=e_n$, where $e_n=(0,\ldots,1)$.
Let us denote by $u^1_\nu(0)$ and $u^2_\nu(e_n)$  the exterior normal derivatives of $u_1$ and $u_2$ respectively at 0 and $e_n$. Note that the two normals have opposite direction. 
We want to deduce a relation between $ u_\nu^1(0)$ and $u^2_\nu (e_n)$. Let us start by recalling some basic properties about the level surfaces of the distance function to a set.

\subsection{Level surfaces of the distance function to a set. Some basic Properties}\label{distancepropertiessub8}
Consider a bounded open set $S$ and its boundary $\partial S,$ of the class $C^2$.  
  Let $\varkappa_i(x)$  be the  principal curvatures of $\partial S$ at $x$ (outward is the positive direction). Assume 
  that for any point $x\in \partial S$ there exists a tangent ball  $B_R(z)$ to $\partial S$ at $x$ such that  $B_R(z)\subset S^c$. In particular the principle curvatures 
  satisfy $\varkappa_i(x)\leq1/R$, $i=1,\ldots,n-1$.  Then:
\begin{enumerate}
\item[a)] the distance function to $S,$ $d_S (x)=d(x,\overline{S}),$ is defined and is $C^{2}$ as long as $$0<d_S (x)<R.$$ 
In the following lemma, which may be known in the literature,  we provide a proof of the $C^{1,1}$-regularity for a more general set, which is not necessary $C^2$, it may have edges as well, but it has the property that for any tangent ball 
there exists a ``clean area", in the sense explained below.
For the $C^2$-regularity in the case of $C^2$-boundaries, see for instance Theorem 14.16 in \cite{gilbarg_elliptic_2001}.

Given a bounded closed set $F$, we say that $\Pi$ is a supporting hyperplane at $x\in\partial F$, if $x\in\Pi$ and there exists a ball $B\subset F^c$ such that $B$ is tangent to $\Pi$ 
at $x$.
\begin{lem}\label{didtancefunctionc11}
Let $F$ be a  bounded closed set.
Assume  that there exists $R>0$ such that, 
for any $x\in\partial F$ and any supporting hyperplane $\Pi$ at $x$, there is a ball $B_R(z)$ tangent to $\Pi$ at $x$ such that 
$B_R(z)\subset F^c$. Let us denote by $d_F(x)=d(x,F)$ the distance function from $F$. Then $d_F$ is of class $C^{1,1}$ in the set $\{0<d_F<R\}$. 
\end{lem}
\begin{proof}
Let $y_0\in \{0<d_F<R\}$. To prove that $d_F$ is of class $C^{1,1}$ at $y_0$, we show that there are  smooth 
functions whose graphs are tangent from below and above the graph of $d_F$ at $(y_0,d_F(y_0))$.  As proven in Lemma \eqref{deltadlem}, the distance function from a closed bounded set 
has always  a smooth tangent function from above. Indeed, 
let $x\in\partial F$ be a point where $y_0$ realizes the distance from $F$. 
Assume, without loss of generality, that $x=0$. Then $d(y_0,0)=|y_0|=d_F(y_0)$. Moreover, the ball $B_{|y_0|}(y_0)$ is contained in $F^c$ and tangent to $F$ at 0.
For any $y\in B_{|y_0|}(y_0)$, we have that $d_F(y)\leq d(y,0)=|y|$. Therefore the cone,  graph of the function $y\to|y|$ (which is smooth at $y_0\neq 0$) is tangent from above 
to the graph of  $d_F$ at $(y_0, d_F(y_0))$. 

Next, we prove the existence of a smooth function tangent from below.  Note that the tangent line to $B_{|y_0|}(y_0)$  at 0 is a supporting hyperplane to $F$ at $0$. Therefore, there exists a ball $B_R(z)$ tangent to $F$ at 0 such that 
$ B_R(z)\subset F^c$. We must have $z=Ry_0/|y_0|$. Moreover, since $ B_R\left(R\frac{y_0}{|y_0|}\right)\subset F^c$, for any $y\in B_R\left(R\frac{y_0}{|y_0|}\right)\cap \{0<d_F<R\}$, we have that 
$$d_F(y)\geq d\left(y, \partial  B_R\left(R\frac{y_0}{|y_0|}\right)\right)=R-d\left(y,R\frac{y_0}{|y_0|}\right)$$ and $d_F(y_0)=|y_0|=R-d\left(y_0,R\frac{y_0}{|y_0|}\right).$
That is to say,  the cone,  graph of the function $y\to R-d\left(y,R\frac{y_0}{|y_0|}\right)$ is tangent by below to the  graph of  $d_F$ at $(y_0, d_F(y_0))$. 
We conclude that $d_F$ is $C^{1,1}$ at $y_0$. 
\end{proof}
Let  $S(k)$ denote the surface that is at distance $k$ from $S$
$$
S(k):=\{x: d_S(x)=k\}, $$
then, for $k<1+\ep$ and $x \in S(k)$, there is a unique point $x_0 \in S(0)$, such that $x= x_0+ k  \nu (x_0)$ where $\nu (x_0) $ is the unit normal vector at $x_0$ in the positive direction.  More precisely, if we denote  $K:=\max \{ |\ka_i (x)|: 1\leq i \leq n-1, x \in \p S\}$ and $f(x,t):=x+ t \nu(x)$, then   $f$ is a diffeomorphism   between $\p S \times (-k, k)$ and the neighborhood of  $\p S,$ $N_k (S)= \{x+ t \nu(x): x\in\partial S,\, |t|< k    \}$ with $k < \frac{1}{K}.$ 
\item[b)] \label{trajectories} for all $x_0 \in \p S$ if we consider the linear transformation $x_t=x_0 + t \nu (x_0)$ we obtain $S(t).$ 
Hence, since the tangent plane for each $S(t)$ is always perpendicular to $\nu (x_0),$ the eigenvectors of the principal curvatures remain constant along the trajectories of $d_S,$ for $d_S<1+\ep.$
\item[c)] the curvatures of $S(k)$ satisfy, see Figure \ref{fig:curvatures}
$$ \ka_i(x_0+k\nu(x_0))=\frac{1}{\frac{1}{\ka_i (x_0)}-k}=\frac{\ka_i(x_0)}{1-\ka_i(x_0)k}, \qquad i=1, \ldots, n-1, \quad k<1+\ep$$ for $x_0\in\partial S$.
\begin{figure}
\begin{center}
\includegraphics[scale=0.4]{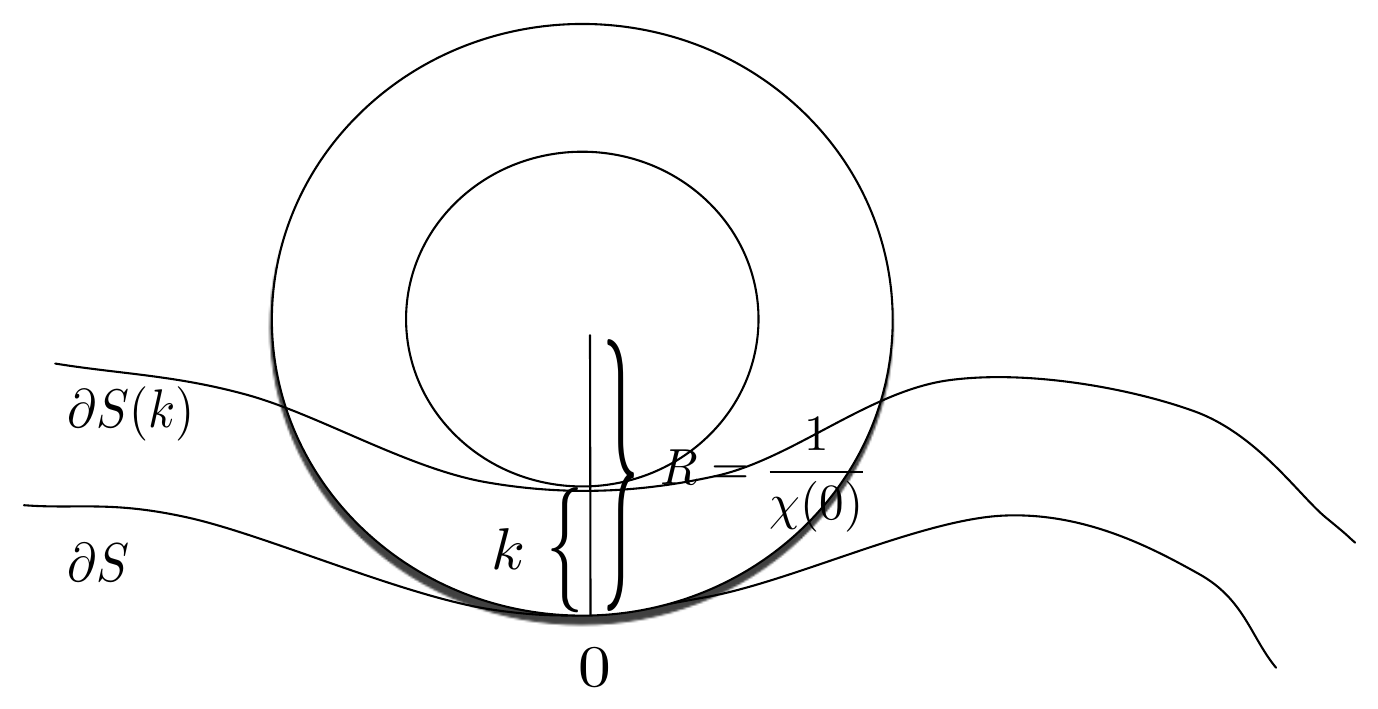}
\caption{Curvatures relation}
\label{fig:curvatures}
\end{center}
\end{figure}
%
%
%
%
\item[d)] for $x_0 \in \p S$, the ball $B_1(x_0)$ touches $S(1)$ at the point $x_0 + \nu(x_0),$ where $\nu$ is the outward normal. Moreover, it separates quadratically from $S(1),$ 
that is, for any small $r>0$ and for any $x\in B_r(x_0 + \nu(x_0))\cap \partial B_1(x_0)$, we have that $d(x,S(1))\leq Cr^2$, for some $C>0$.
\end{enumerate}

\subsection{Free boundary condition}
Following Subsection \ref{distancepropertiessub8}, we denote by $\ka_i(0)$ the principal curvatures of $\partial S_1$ at 0 where outward is the positive direction and by
$\ka_i(e_n)=\frac{\ka_i(0)}{1-\ka_i(0)}$, the principal curvatures of $\partial S_2$ at $e_n$. Remark that since the normal vectors to  $S_1$  and  $ S_2$ 
respectively at  0 and $e_n$, have opposite directions, for $ \ka_i(e_n)$  the inner direction of $S_2$ is the positive one.
The main result of this section is the following:
\begin{thm} 
\label{thm: fbcond}
Assume \eqref{normderassum}. Let $0\in \partial S_1$ and $e_n\in  \partial S_2$. Assume that   $\partial S_1$ is of class $C^2$   in
  $B_{4h_0}(0)$ and that the principal curvatures satisfy: $\ka_i(0)<1$ for any $i=1,\ldots,n-1$. 
Then, we have the following relation:
$$
\frac{u_\nu^1(0)}{u^2_\nu (e_n)}=\prod_{i=1\atop \ka_i(0)\neq 0 }^{n-1} \frac{\ka_i(0)}{\ka_i (e_n)}\quad\text{if }\ka_i(0)\neq 0\text{ for some }i=1,\ldots,n-1,
$$ and
\begin{equation*}u_\nu^1(0)=u^2_\nu (e_n)\quad\text{if }\ka_i(0)= 0\text{ for any }i=1,\ldots,n-1.\end{equation*}
\end{thm}
In order to prove Theorem \ref{thm: fbcond}, we first prove
 a lemma that relates the mass of the Laplacians of the limit functions across the interfaces. For a point $x$ belonging to a neighborhood of $\partial S_1$ around 0,
 let us denote by $\nu(x)=\nu(x_0)$ the exterior  normal vector at $x_0\in\partial S_1$, where $x_0$ is the unique point such that $x=x_0+t\nu(x_0)$, for some small $t>0$. From (a) in 
  Subsection \ref{distancepropertiessub8}, $\nu(x)$ is well defined. 

\begin{lem}
\label{lem: mass laplacian} Under the assumptions of Theorem \ref{thm: fbcond}, for small $h< h_0$, let  
$$D_h:=B_h (0) \cap \{x: d(x,\partial S_1)\leq h^2\}$$ and 
$$E_h:=\{y\in\real^n\,|\,y=x+ \nu (x), x \in D_h\}.$$
Then
$$\int_{D_h}\Delta u_1  = \int_{E_h}\Delta u_2 .
$$
\end{lem}
\begin{proof}
Remark that the surface $E_h\cap\partial S_2$ is  of class $C^2$  for $h$ small enough, being $\ka_i(0)<1$ for  $i=1,\ldots,n-1$, see Subsection \ref{distancepropertiessub8}.
The Laplacians of the $u_i$'s are positive measures and 
$$ \int_{D_h}\Delta u_1  = \lim_{\ep \rightarrow 0}\int_{D_h}\Delta u^\ep_1(x)  \:\text{d}x 
 =
\lim_{\ep \rightarrow 0} \frac{1}{\ep^2} \int_{D_h} \int_{B_1 (x)} u^\ep_1  (x) u^\ep_2 (y)\: \text{d}y\text{d}x,
$$ and
$$ \int_{E_h}\Delta u_2= \lim_{\ep \rightarrow 0}\int_{E_h}\Delta u^\ep_2(y)  \:\text{d}y 
 =
\lim_{\ep \rightarrow 0} \frac{1}{\ep^2} \int_{E_h} \int_{B_1 (y)} u^\ep_1  (x) u^\ep_2 (y)\:\text{d}x \text{d}y.
$$
Let $s$ be such that $\ep^{\frac{1}{4\alpha}}<s <h$, where $\alpha$ is given by  Lemma \ref{uj=0closui}.
We split the set $D_h$ in the following way
$$D_h=D_{h,s}^+\cup D_{h,s}^-\cup D_{h,s},$$ where
\begin{equation*}D_{h,s}^+:=\{x \in D_h\,|\,d(x,\p S_1)>s^2\text{ and }u_1(x)>0\},\end{equation*}
\begin{equation*}D_{h,s}^-:=\{x \in D_h\,|\,d(x,\p S_1)>s^2\text{ and }u_1(x)=0\},\end{equation*}
\begin{equation*}D_{h,s}:=\{x \in D_h\,|\,d(x,\p S_1)\le s^2\}.\end{equation*}
Similarly
$$E_h=E_{h,s}^+\cup E_{h,s}^-\cup E_{h,s},$$ where
\begin{equation*}E_{h,s}^+:=\{x \in E_h\,|\,d(x,\p S_2)>s^2\text{ and }u_2(x)>0\},\end{equation*}
\begin{equation*}E_{h,s}^-:=\{x \in E_h\,|\,d(x,\p S_2)>s^2\text{ and }u_2(x)=0\},\end{equation*}
\begin{equation*}E_{h,s}:=\{x \in E_h\,|\,d(x,\p S_2)\le s^2\},\end{equation*}
see Figure \ref{fig:mass laplacian}.
Since $\partial S_1$ is a smooth surface around 0, and $\Delta u_1=0$ in $S_1$, we have that $u_1$ grows linearly away from the boundary in a neighborhood of 0.
This and the uniform convergence of $u_1^\ep$ to $u_1$, imply that there exists $c>0$ such that 
 $u^\ep_1(x)> cs^2$, for any $x\in D_{h,s}^+$ for $\ep$ small enough. Then,
by Lemma \ref{uj=0closui}, $u^\ep_2 (y) \leq a e ^{-\frac{b{(cs^2)}^{\alpha}}{\ep}},$ ($a,b$ positive constants),
 for $y \in B_1(x)$ and any  $x\in  D_{h,s}^+$. In an analogous way,
 if $y \in E_{h,s}^+$, we know that for $\ep$ small enough $u^\ep_2(y)> cs^2$ and
by Lemma \ref{uj=0closui},  $u^\ep_1 (x) \leq a e ^{-\frac{b{(cs^2)}^{\alpha}}{\ep}}$
for $x \in B_1(y). $ Since we have chosen $s$ such that $s^{2\alpha} >\ep^{\frac{1}{2}}$, we have that 
 $u^\ep_2(y)=o(\ep^2)$  uniformly in $y$, for any $ y\in\cup_{x \in D_{h,s}^+} B_1 (x)$ and $u^\ep_1(x)=o(\ep^2)$   uniformly in $x$, for any $x\in \cup_{y \in E_{h,s}^+} B_1 (y).$ 
 Remark that 
$$D_{h,s}^-  \subset \cup_{y \in E_{h,s}^+} B_1 (y).$$ 
Therefore we have
\begin{equation}\label{firstreduction}\begin{split}
 \frac{1}{\ep^2} \int_{x \in D_{h}} \int_{y \in B_1 (x)} u^\ep_1 (x) u^\ep_2 (y)\text{d}y\text{d}x
 &= \frac{1}{\ep^2}\int_{x \in D^+_{h,s}} \int_{y\in B_1 (x)} u^\ep_1 (x) \underbrace{u^\ep_2 (y)}_{negligible}  \text{d}y\text{d}x  \\
&+\frac{1}{\ep^2} \int_{x \in D_{h,s}}  \int_{ y\in B_1 (x)} u^\ep_1 (x) u^\ep_2 (y)  \text{d}y\text{d}x \\
&+\frac{1}{\ep^2} \int_{x \in D_{h,s}^-}  \int_{ y\in B_1 (x)} \underbrace{u^\ep_1 (x)}_{negligible} u^\ep_2 (y)  \text{d}y\text{d}x\\
&=\frac{1}{\ep^2} \int_{x \in D_{h,s}}  \int_{y\in  B_1 (x)} u^\ep_1 (x) u^\ep_2 (y) \text{d}y\text{d}x + o(1).
\end{split}\end{equation}
\begin{figure}
\begin{center}
\includegraphics[scale=0.5]{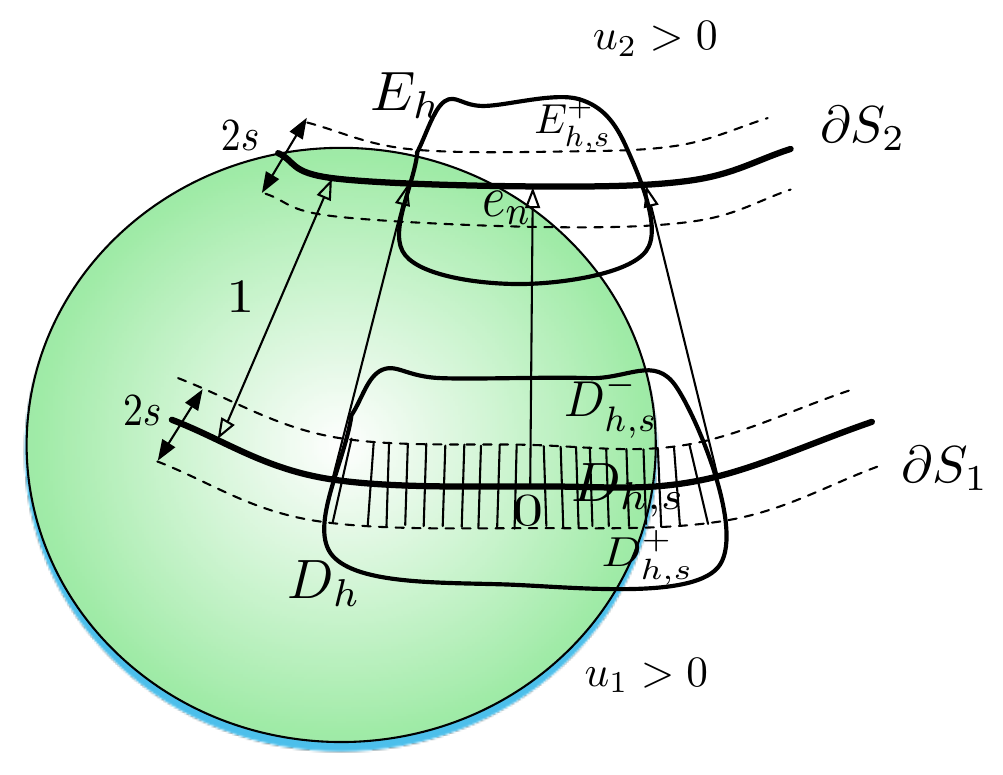}
\caption{Relation between the mass of the Laplacians}
\label{fig:mass laplacian}
\end{center}
\end{figure}
Analogously
\begin{equation}\label{firstreductionu2}
 \frac{1}{\ep^2}\int_{y \in E_h} \int_{x \in B_1(y)} u^\ep_1  (x) u^\ep_2 (y)  \:\text{d}x \text{d}y = \frac{1}{\ep^2}\int_{E_{h,s}} \int_{B_1(y)} u^\ep_1  (x) u^\ep_2 (y)  \:\text{d}x \text{d}y +o(1). 
 \end{equation}
  Next, for fixed $x\in D_{h,s}$,we have $$B_1(x)\cap\{y\,|\,d(y,\partial S_2)>s^2\}\subset  B_{1+h}(0)\cap\{y\,|\,d(y,\partial S_2)>s^2\}\cap\{u_2\equiv 0\}.$$
  Therefore for any $y\in B_1(x)\cap\{y\,|\,d(y,\partial S_2)>s^2\}$, the ball $B_1(y)$ enters in $S_1\cap B_{2h}(0)$ at distance at least $s^2$ from $\partial S_1$. Since 
  $\partial S_1\cap B_{4h}(0)$ is of class $C^2$, $u_1$ has linear growth away from the boundary in $\partial S_1\cap B_{2h}(0)$ and therefore there exists a point in 
  $ B_1(y)$ where $u_1\geq cs^2$ for some $c>0$.
  Like before, Lemma \ref{uj=0closui} implies that $u^\ep_2(y)=o(\ep^2)$. We infer that
  \begin{equation}\label{secondreduction}\frac{1}{\ep^2} \int_{x \in D_{h,s}}  \int_{y\in  B_1 (x)} u^\ep_1 (x) u^\ep_2 (y)  \text{d}y\text{d}x
  =\frac{1}{\ep^2} \int_{x \in D_{h,s}}  \int_{y\in  B_1 (x)\cap\{y\,|\,d(y,\partial S_2)\le s^2\}} u^\ep_1 (x) u^\ep_2 (y)  \text{d}y\text{d}x+o(1).
  \end{equation}   
  Finally, remark that (d) of Subsection \ref{distancepropertiessub8} implies that for $x\in D_{h,s}$
  \begin{equation}\label{thirdreduction}B_1(x)\cap\{y\,|\,d(y,\partial S_2)\le s^2\}\subset E_{h+cs,s} \end{equation} for some $c>0$.
  From \eqref{firstreduction}, \eqref{firstreductionu2}, \eqref{secondreduction} and  \eqref{thirdreduction}, we get
  \begin{equation*}
  \begin{split}
  \int_{D_{h}}\Delta u^\ep_1(x)dx&=
 \frac{1}{\ep^2} \int_{x \in D_{h}} \int_{y \in B_1 (x)} u^\ep_1 (x) u^\ep_2 (y) \text{d}y \text{d}x\\&=\frac{1}{\ep^2} \int_{x \in D_{h,s}}  \int_{y\in  B_1 (x)\cap\{y\,|\,d(y,\partial S_2)\le s^2\}} u^\ep_1 (x) u^\ep_2 (y)  \text{d}y \text{d}x+o(1)\\&
 \leq \frac{1}{\ep^2} \int_{x \in D_{h,s}}  \int_{y\in E_{h+cs,s}} u^\ep_1 (x) u^\ep_2 (y) \text{d}y \text{d}x+o(1)\\&
 \leq \frac{1}{\ep^2}\int_{y\in E_{h+cs,s}} \int_{x \in B_1(y)} u^\ep_1 (x) u^\ep_2 (y) \text{d}x \text{d}y+o(1)\\&
 =\int_{ E_{h+cs}}\Delta u^\ep_2(y)dy+o(1).
 \end{split}
 \end{equation*}
Similar computations give
  \begin{equation*}
   \int_{ E_{h}}\Delta u^\ep_2(y)dy\leq \int_{D_{h+cs}}\Delta u^\ep_1(x)dx+o(1).
 \end{equation*}
 Letting first $\ep$ and then $s$ go to 0, the conclusion of the lemma follows.
 
 \end{proof}

\begin{lem}
\label{lem: areas ratio}
 Under the assumptions of Theorem \ref{thm: fbcond}, let $\Gamma^1_h=\p S_1 \cap B_h(0)$ and let $\Gamma^2_h=\{x+\nu(x): x \in \Gamma^1_h\}.$
 Then we have the limits
\begin{equation}\label{boundarylimits1}
\lim_{h\to 0} \frac{\int_{\Gamma^2_h}  \:dA}{ \int_{\Gamma^1_h}  \:dA} =\prod_{i=1\atop \ka_i(0)\neq 0 }^{n-1} \frac{\ka_i(0)}{\ka_i (e_n)}
\quad\text{if }\ka_i(0)\neq 0\text{ for some }i=1,\ldots,n-1,\end{equation}
 and
\begin{equation}\label{boundarylimits2}
\lim_{h\to 0} \frac{\int_{\Gamma^2_h}  \:dA}{ \int_{\Gamma^1_h}  \:dA} =1\quad\text{if }\ka_i(0)= 0\text{ for any }i=1,\dots,n-1.
 \end{equation}  
\end{lem}

\begin{proof}
Consider the diffeomorphism $f_t(x)=f(x,t)=x+ t \nu(x).$ Then $\Gamma^2_h=f_1(\Gamma^1_h)$ and 
$$
\int_{\Gamma^2_h}  dA= \int_{\Gamma^1_h} |Jf_1(x)| dA,
$$
where $|Jf_1|$ is the determinant of the Jacobian of $f_1$. Taking as basis of the tangent space at 0 the principal directions, $\tau_i$, then the differential of $f_1$ at  $x$ is given by
$$
(df_1)(\tau_i)=\tau_i + (d \nu)(\tau_i) = \tau_i- \ka_i \tau_i. 
$$
So,
$$
|Jf_1(x)|= \prod_{i=1}^{n-1} (1-\ka_i(x))
$$
and 
$$
\frac{\int_{\Gamma^2_h}  dA }{\int_{\Gamma^1_h}  dA} =
\frac{1 }{\mbox{Area}\:(\Gamma^1_h) } \int_{\Gamma^1_h}  \prod_{i=1}^{n-1}   (1-\ka_i(x)) dA.
$$
Passing to the limit when h converges to zero, we obtain
$$
\lim_{h\to 0} \frac{\int_{\Gamma^2_h}  \:dA}{ \int_{\Gamma^1_h}  \:dA}=\prod_{i=1}^{n-1} (1-\ka_i(0)).$$
Now, if $\ka_i(0)\neq 0$ for some $i=1,\dots,n-1$, then
$$\prod_{i=1}^{n-1} (1-\ka_i(0))=\prod_{i=1\atop \ka_i(0)\neq 0}^{n-1} (1-\ka_i(0))=\prod_{i=1\atop \ka_i(0)\neq 0}^{n-1}\left( \frac{1-\ka_i(0) }{\ka_i(0) }\ka_i(0) \right)
= \prod_{i=1\atop \ka_i(0)\neq 0}^{n-1}\frac{  \ka_i(0)}{  \ka_i (e_n)},
$$
and \eqref{boundarylimits1} follows. 

If $\ka_i(0)= 0$ for any $i=1,\dots,n-1$, then 
$$\prod_{i=1}^{n-1} (1-\ka_i(0))=1$$ and we get  \eqref{boundarylimits2}. 

\end{proof}

{\em  Proof of Theorem \ref{thm: fbcond}.} 

Let $\Gamma^1_h=\partial S_1 \cap D_h$ and $\Gamma^2_h=\partial S_2 \cap E_h.$ 
The Laplacians $\Delta u_i,$  are jump measures along $\p S_i$, $i=1,2,$ and satisfy 
\begin{equation*}
\int_{D_h}\Delta u_1 =-\int_{\Gamma^1_h} u^1_\nu\:dA \quad\text{and}\quad
 \int_{E_h}\Delta u_2 =-\int_{\Gamma^2_h} u^2_\nu  \:dA.
\end{equation*}
Then, using Lemma \ref{lem: mass laplacian}  we get
$$
 1= \frac{\int_{D_h}\Delta u_1  }{ \int_{E_h}\Delta u_2  } =   \frac{\int_{\Gamma^1_h} u^1_\nu  \:dA }{ \int_{\Gamma^2_h} u^2_\nu \:dA }, 
 $$
 and so
 $$
  \frac{\fint_{\Gamma^1_h} u^1_\nu  \:dA }{ \fint_{\Gamma^2_h} u^2_\nu \:dA } =   \frac{\int_{\Gamma^2_h}   \:dA }{ \int_{\Gamma^1_h} \:dA }. 
 $$
 Since, when $h\rightarrow 0,$
 $$
 \frac{\fint_{\Gamma^1_h} u^1_\nu  \:dA }{ \fint_{\Gamma^2_h} u^2_\nu \:dA }\rightarrow \frac{u_\nu^1(0) \: }{u^2_\nu (e_n)},
 $$
by Lemma \ref{lem: areas ratio} the conclusion of Theorem \ref{thm: fbcond} follows.

\finedim

%
%
%

\bibliography{CPQ-joint-final}
\bibliographystyle{plain}
%
%
%

\end{document}